\documentclass[11pt]{article}
    
\usepackage{amsmath,amsfonts,amsthm,amscd,amssymb,graphicx}
\usepackage{mathabx}

\usepackage[utf8]{inputenc}

\usepackage{subfigure}
\usepackage{xcolor}

\numberwithin{equation}{section}

\usepackage{hyperref}

\newtheorem{theorem}{Theorem}[section]

\newtheorem{lemma}[theorem]{Lemma}

\newtheorem{proposition}[theorem]{Proposition}

\newtheorem{remark}[theorem]{Remark}

\newcommand{\RR}{\mathbb{R}}
\newcommand{\cO}{\mathcal{O}}

\newcommand{\CC}{\mathbb{C}}

\newcommand{\mS}{\mathfrak{S}}

\newcommand{\cH}{\mathcal{H}}

\def\hw{\hat{w}}
\def\hgamma{\hat{\gamma}}

\def\Frho{\widehat{\rho}}

\def\Trho{\widetilde{\rho}}

\def\Trho{\widetilde{\varphi}}

\def\Frho{\widehat{\rho}}

\def\Trho{\widetilde{\rho}}

\begin{document}

%


\title{Plasmons for the Hartree equations with Coulomb interaction} 

\author{Toan T. Nguyen\footnotemark[1]
\and Chanjin You\footnotemark[1]
}

\maketitle

\footnotetext[1]{Penn State University, Department of Mathematics, State College, PA 16802. Emails: nguyen@math.psu.edu, cby5175@psu.edu. The research is supported in part by the NSF under grant DMS-2054726.}

\maketitle

\begin{abstract}

In this work, we establish the existence and decay of {\em plasmons}, the quantum of Langmuir's oscillatory waves found in plasma physics, for the linearized Hartree equations describing an interacting gas of infinitely many fermions near general translation-invariant steady states, including compactly supported Fermi gases at zero temperature, in the whole space $\RR^d$ for $d\ge 2$. Notably, these plasmons exist precisely due to the long-range pair interaction between the particles. Next, we provide a survival threshold of spatial frequencies, below which the plasmons purely oscillate and disperse like a Klein-Gordon's wave, while at the threshold they are damped by {\em Landau damping}, the classical decaying mechanism due to their resonant interaction with the background fermions. The explicit rate of Landau damping is provided for general radial homogenous equilibria. Above the threshold, the density of the excited fermions is well approximated by that of the free gas dynamics and thus decays rapidly fast for each Fourier mode via {\em phase mixing}. Finally, pointwise bounds on the Green function and dispersive estimates on the density are established. 

\end{abstract}

\tableofcontents



\section{Introduction}


Consider the Hartree equation 
	\begin{equation}
		\label{Hartree}
		\begin{cases}
			\begin{aligned}
				&i \partial_{t} \gamma = [ - \Delta + w \star_x \rho_\gamma, \, \gamma ] \\
				&\gamma_{\vert_{t=0}} = \gamma_{0}
			\end{aligned}
		\end{cases}
	\end{equation}
describing the meanfield dynamics of an interacting gas containing infinitely many fermions in the whole space $\RR^d$ with $d\ge 2$, where the unknown $\gamma(t)$ is a self-adjoint, non-negative, and bounded operator on $L^2(\mathbb{R}^d)$ representing the one-particle density matrix and $\rho_\gamma(t,x) $ is the scalar density function associated to $\gamma(t)$, namely
\begin{equation}\label{def-density}\rho_\gamma(t,x) = \gamma(t,x,x)
\end{equation}
where $\gamma(t,x,y)$ is the integral kernel of $\gamma(t)$. Here, $[A,B] = AB - BA$. The nonlinear convolution term $w \star_x \rho_\gamma$ describes the meanfield interaction between the particles, where $w(x)$ is a given two-body interaction potential function. To fix the setting, we shall focus on the classical Coulomb potential
 \begin{equation}\label{Coulomb}
\omega(x) = \left\{ 
\begin{aligned}
-\log |x| , \qquad \mbox{for}\quad d=2, 
\\
\frac{1}{|x|^{d-2}}, \qquad \mbox{for}\quad d\ge 3,
\end{aligned}
\right.
\end{equation}
for which the long-range interaction (i.e. the fact that $\int w(x) \; dx$ is infinite) plays a role in this work that we shall describe shortly.

The Hartree equations \eqref{Hartree} can be derived from a many-body quantum system in the meanfield limit (e.g., \cite{Bardos1, Bardos2, Schlein1, Schlein2, Knowles, CLS}), and the global-in-time Cauchy problem \eqref{Hartree} has been resolved for trace class initial data with finite energy in \cite{Bove1, Bove2, Chadam, Zagatti}. 
In the case when the density matrix $\gamma(t)$ has the spectral decomposition of the form 
\begin{equation}\label{gamma-ex} \gamma(t) = \sum_j \lambda_j | u_j(t)\rangle \langle u_j(t)|
\end{equation}
for orthonormal functions $u_j(t)$ in $L^2(\RR^d)$, where $| u\rangle \langle v|$ denotes the operator $f\mapsto \langle v,f\rangle u$ from $L^2$ into itself, the problem \eqref{Hartree} reduces to solving a system of coupled Hartree equations for each scalar functions $u_j(t)$, namely 
	\begin{equation}
		\label{Hartree-uj}
		\begin{cases}
			\begin{aligned}
				&i \partial_{t} u_j =  (- \Delta + w \star_x \rho_\gamma ) u_j \\
				&{u_j}_{\vert_{t=0}} = u_j(0)
			\end{aligned}
		\end{cases}
	\end{equation}
for each $j$, in which the coupling is obtained through $\rho_\gamma(t,x) = \sum_j \lambda_j |u_j(t,x)|^2$. This formulation makes perfect sense when $\gamma(t)$ is a finite rank operator (i.e. $\lambda_j =0$, but a finite number of $j$'s) or in the finite trace class $\sum_j |\lambda_j|<\infty$, as studied in the aforementioned works. Evidently, the cubic nonlinear Schr\"odinger or standard Hartree equations are a special case of \eqref{Hartree-uj} when $\lambda_j =0$ for all, but  $\lambda_1\not =0$, which are well studied in the literature (\cite{Tao}). 

Recently, Lewin and Sabin \cite{Lewin1} developed a framework to establish a well-posedness theory of \eqref{Hartree} for a class of initial data with infinite trace as a suitable perturbation of translation-invariant steady states of the form 
\begin{equation}\label{steadystates}
\mu = \mu(-\Delta)
\end{equation}
with a finite constant density $\rho_\mu = \int_{\RR^d} \mu(|p|^2)\; dp$. 
The result was extended to include the delta interaction potential \cite{Chen1}. See also \cite{Lewin3} for uniform estimates in the semi-classical limit. When localized perturbations are sufficiently small in suitable Schatten spaces, the large time behavior and scattering of solutions to \eqref{Hartree} are also established in \cite{Lewin2, Chen2, Collot}. 
We note however that all these results were obtained for short-range interaction potentials, namely $w \in L^1(\RR^d)$, for which the linearized problem around the steady state \eqref{steadystates} has a trivial kernel, see Section \ref{sec-disp} below, leaving it widely open to include long-range interaction between the particles. As it turns out, long-range interaction potentials such as the classical Coulomb interaction \eqref{Coulomb} reveal new physics, as opposed to short range interaction carried out in the previous works (\cite{Lewin2, Chen2, Collot}), which causes substantial fundamental issues as we shall discuss below. This motivates the present work. 

\subsection{Dispersion relation}

Specifically, in this paper, we establish the large time behavior of solutions to the linearized Hartree equations near translation-invariant steady states of the form $\mu(-\Delta)$, which read
	\begin{equation}
		\label{linHartree}
		\begin{cases}
			\begin{aligned}
				&i \partial_{t} \gamma = [ - \Delta, \gamma] + [w \star_{x} \rho_\gamma, \,\mu ] \\
				&\gamma_{\vert_{t=0}}= \gamma_{0},
			\end{aligned}
		\end{cases}
	\end{equation}
where $\rho_\gamma(t,x) $ is the scalar density function associated to $\gamma(t)$. The linearized problem can in fact be solved completely mode by mode for each spatial frequency. Indeed, denoting by $\Trho_\gamma(\lambda,k)$ the Fourier-Laplace transform of the density function $\rho_\gamma(t,x)$,  
we obtain 
\begin{equation}\label{linresolvent}
\begin{aligned}
\Trho_\gamma(\lambda,k) &= \frac{1}{D(\lambda,k)} \Trho_{\gamma^0}(\lambda,k),
\end{aligned}
\end{equation}
for $\lambda \in \CC$ and $k\in \RR^d$, where $\Trho_{\gamma^0}(\lambda,k)$ denotes the Laplace-Fourier transform of the density $\rho_{\gamma^0}(t,x)$ generated by the free Hartree dynamics $i \partial_{t} \gamma^0 = [ - \Delta, \gamma^0]$ with the same initial data $\gamma_0$. Namely, $\rho_{\gamma^0}$ is the density function of $\gamma^0 = e^{it\Delta}\gamma_0 e^{-it\Delta}$, while $D(\lambda,k)$ is the symbol for the linearized problem \eqref{linHartree} in the spacetime frequency space, see Section \ref{sec-FT}. In other words, the resolvent equation \eqref{linresolvent} asserts that the density of the linearized Hartree equations \eqref{linHartree} is the spacetime convolution of an integral kernel against the density of the free Hartree dynamics $\rho_{\gamma^0}$, where the spacetime integral kernel is the inverse Laplace-Fourier transform of $\frac{1}{D(\lambda,k)}$. 

For the free Hartree dynamics $\gamma^0 = e^{it\Delta}\gamma_0 e^{-it\Delta}$, the dispersion of free particles is measured through the Strichartz estimates for the density function: namely, we recall from \cite{Rupert, Rupert2} that 
\begin{equation}\label{free-Sch}
\| \rho_{\gamma^0} \|_{L^p_tL^q_x} \lesssim \| \gamma_0\|_{\mS^{\frac{2q}{q+1}}}
\end{equation}
for any pair $(p,q)$ so that $2/p+d/q = d$ and $1\le q < 1+2/(d-1)$, where $\mS^q$ denotes the $q$-Schatten space of the compact operators $A$ with a finite trace $\operatorname{tr}(|A|^q)$. In fact, for spatially localized initial perturbations $\gamma_0$, the density $ \rho_{\gamma^0}(t,x)$ has its Fourier transform $\Frho_{\gamma^0}(t,k)$ that decays rapidly fast in $\langle kt\rangle$, exponentially fast (or  polynomially fast) for exponentially localized data $\gamma_0$ (resp. polynomially localized), resembling the {\em phase mixing} behavior of the free gases, see Section \ref{sec-freeHartree} for the details. 

It remains to study the spacetime symbol $D(\lambda,k)$ in the resolvent equation \eqref{linresolvent}, also known as the Lindhard dielectric function in quantum mechanics \cite{Lindhard, GV, LM}, a quantum analogue of the classical dielectric function in plasma physics  \cite{Trivelpiece}, which is computed by 
\begin{equation}\label{def-Dintro}
\begin{aligned}
D(\lambda,k)
&= 1 + i\hw(k)\int_{\RR^d} \frac{\mu(|p|^2) - \mu(|p-k|^2)}{\lambda -2ik \cdot p + i|k|^2} \;dp 
\end{aligned}
\end{equation}
where $\hw(k)$ denotes the Fourier transform of the interaction potential $w(x)$, see Lemma \ref{lem-resolvent}. For Coulomb interaction \eqref{Coulomb}, $\hw(k) = |k|^{-2}$.  Note that for each wave number $k\in \RR^d$, the zeros $\lambda_*(k)$ of the dispersion relation 
\begin{equation}\label{0disp}
D(\lambda_*(k),k)=0
\end{equation} 
are the eigenvalues of the linearized Hartree equations \eqref{linHartree}. As a result of \eqref{linresolvent}, the leading dynamics of the density is governed by the residual of $e^{\lambda t}\frac{1}{D(\lambda,k)}$ at $\lambda = \lambda_*(k)$, upon taking the inverse of Laplace transform.

\subsection{Penrose-Lindhard's stable regime}\label{sec-disp}

In view of the density representation \eqref{linresolvent}, a natural stability condition reads
\begin{equation}\label{Penrose-Lindhard}
\inf_{k\in \RR^d} \inf_{ \Re \lambda \ge 0} |D(\lambda,k)|  \ge \theta_0 >0,
\end{equation}
which not only asserts that no growth mode of \eqref{linHartree} exists but also assures the invertibility of the linearized operator, say in $L^p_{t,x}$ spaces, in view of the classical Stein and Marcinkiewicz's theorem for the spacetime Fourier multipliers $\frac{1}{D(\lambda,k)} $. 
The stability condition \eqref{Penrose-Lindhard} bears the name of Penrose in plasma physics \cite{Penrose,Trivelpiece} that plays a central role in establishing the large time behavior of solutions to the nonlinear Vlasov-Poisson system \cite{MV, BMM-apde,GNR1,GNR2, HKNR2}. The stability condition has a quantum analog \cite{Lindhard,GV, LM}. In their important work \cite{Lewin2}, Lewin and Sabin have verified \eqref{Penrose-Lindhard} for a large class of equilibria $\mu(|p|^2)$, for smooth and short-range interaction potentials, under which the large time dynamics of solutions to the nonlinear problem \eqref{Hartree} is governed by that of the free Hartree dynamics \eqref{free-Sch}. The result was first proved for dimension $d=2$ in \cite{Lewin2} and later extended to higher dimensions in \cite{Chen2, Collot}. 

It turns out that for interaction potentials such that $\int w(x) \; dx$ is infinite (i.e. long-range), the stability condition \eqref{Penrose-Lindhard} never holds for {\em any steady states!} New physics arises.  

\subsection{Plasmons}


We shall show that the stability condition \eqref{Penrose-Lindhard} always fails for any reasonable equilibria due to the emergence of oscillatory modes  of \eqref{linHartree} at the low frequency, which we shall refer to as {\em plasmons}, namely the quantum of plasma oscillations found in plasma physics \cite{Trivelpiece, Toan}, see also \cite{Glassey, Glassey1, HKNR3, BMM-lin, Ionescu1}. For the pioneering works and the recent study of plasmons, see  \cite{Fetter, Pines,Nam1, Nam2}. More precisely, for the quantum model \eqref{linHartree}, we show that there is a ``survival threshold'' of wave numbers $\kappa_0$, below which the stability condition \eqref{Penrose-Lindhard} fails and plasmons emerge. Above the threshold, these plasmons are damped due to their resonant interaction with the background particles, known as {\em Landau damping} in plasma physics, which we shall now describe.  


Let $\mu(-\Delta)$ be a general smooth steady state of \eqref{Hartree} with a finite constant density $\rho_\mu$. We set 
\begin{equation}\label{def-Upsilon} \Upsilon: = \sup \Big\{ |p|, \qquad \mu(|p|^2) \not =0 \Big\}
\end{equation}
to be the maximal speed of particle velocities. For positive equilibria, $\Upsilon = \infty$, however for compactly supported equilibria such as Fermi gases at zero temperature, $\Upsilon < \infty$. We then introduce the survival threshold $\kappa_0$ defined through the relation
 \begin{equation}\label{def-introkappa0}
\kappa_0^2 =  \frac{1}{2} \int_{|u|<\Upsilon} \frac{\varphi(u)}{(\Upsilon-u)(\Upsilon+\kappa_0-u)} \, du
\end{equation}
in which $\varphi(u) = \int_{\RR^{d-1}} \mu(u^2 + |w|^2)\; dw$. As equilibria $\mu(\cdot)$ are non negative and decay sufficiently fast as $u\to \Upsilon$, the threshold number $\kappa_0$ is well-defined and finite. Here, we in fact have used $\hw(k) = |k|^{-2}$ in deriving \eqref{def-introkappa0}. For more general long-range interaction potentials, see Remark \ref{rem-genw}. Note that $\kappa_0=0$ if $\Upsilon=\infty$ (e.g., when $\mu(\cdot)$ is a Gaussian), while $\kappa_0$ is nonzero and may be large if $\Upsilon<\infty$ (e.g., when $\mu(\cdot)$ is compactly supported). Throughout the paper, we consider equilibria with connected support, namely $\mu(|p|^2)>0$, whenever $|p|<\Upsilon$.

\bigskip

Our main results, Theorem \ref{theo-mainLandau} below, assert that    

\begin{itemize}

\item {\em Plasmons:} for $0\le |k|\le \kappa_0$, there are exactly two pure imaginary solutions $\lambda_\pm(k) = \pm i \tau_*(k)$ of the dispersion relation $D(\lambda,k)=0$, which obey a Klein-Gordon type dispersion relation 
$$\tau_*(k) \approx \sqrt{1 + |k|^2}$$ 
(and in particular, $\tau_*(k)$ is a strict convex function in $|k|$). These oscillatory modes experience no Landau damping $\Re \lambda_\pm(k)=0$, but disperse in space, since the group velocity $\tau_*'(k)$ is strictly increasing in $|k|$. These oscillations are known as Langmuir's waves in plasma physics \cite{Trivelpiece, Toan} and plasmons in quantum mechanics \cite{Fetter, Pines,Nam1, Nam2}.

\item {\em Landau damping:} as $|k|$ increases past the critical wave number $\kappa_0$, the phase velocity of oscillatory plasmons enters the range of admissible particle velocities, namely $|\nu_*(k)| <\Upsilon$. That is, there are particles that move at the same propagation speed of the waves. This resonant interaction causes the dispersion functions $\lambda_\pm(k)$ to leave the imaginary axis, and thus the purely oscillatory modes get damped. 

In plasma physics, Landau \cite{Landau} computed this damping rate for the linearized Vlasov-Poisson system near Gaussians (and hence, $\kappa_0=0$), which bears his name. The Landau damping rate $\Re \lambda_\pm(k)$ was also well-documented for general equilibria \cite{Trivelpiece}, see also \cite{Toan} for the rigorous derivation, which asserts that the faster equilibria vanish at the maximal speed, the weaker Landau damping is. Our first main theorem, Theorem \ref{theo-mainLandau}, establishes Landau damping for the quantum model \eqref{linHartree}.

\item {\em Penrose-Lindhard's stable regime:} for $|k| > \kappa^+_0$, the strong Penrose-Lindhard stability condition \eqref{Penrose-Lindhard} holds, and therefore the behavior of the density function is governed by that of the free Hartree dynamics as discussed in the previous section. We stress that no monotonicity assumption on the steady states is assumed in dimensions $d\ge 3$. 
   
\end{itemize}

\subsection{Equilibria}\label{sec-mu}

Throughout this paper, we consider equilibria of the form $\mu = \mu(-\Delta)$ as a Fourier multiplier, where $\mu(\cdot)$ is a nonnegative function in $\mathbb{R}_{+}$, satisfying 
\begin{itemize}

\item $\mu(|p|^2) >0$ for all $|p|< \Upsilon$, for $\Upsilon$ defined as in \eqref{def-Upsilon}. 

	\item $\mu$ is of class $C^{n_{0}}$ for some $n_{0} \ge \frac{d+15}{2}$.
	\item In the case when $\Upsilon = \infty$,
		\begin{equation}\label{decay-mu}
			| \mu(e)| \lesssim \langle e \rangle^{-n_{1}}, \qquad \forall e \ge 0
		\end{equation}
		for some $n_{1} \ge n_{0} + \frac{d+1}{2}$.
	\item In the case when $\Upsilon < \infty$,
		\begin{equation}\label{decay-mu2}
			\lim_{|p| \rightarrow \Upsilon} \frac{\mu(|p|^2)}{(\Upsilon - |p|)^{n_{1}}}
		\end{equation}
		exists and is positive for some $n_{1} \ge \min\{ 1,\frac{5-d}{2}\}$.
	\item In dimension $d=2$, we further assume that $\mu' (e) \le 0$ for $e \ge 0$. 
\end{itemize}

Both the regularity and decaying rate assumptions on $\mu(e)$ are not optimal, and we made no attempts in optimizing them in this paper, despite the fact that they play a crucial role in deriving the decay estimates and in calculating the Landau damping rate. We stress that no monotonicity was made on the equilibria in high dimensions $d\ge 3$. 
Apparently, there are many equilibria that satisfy the above assumptions, including any radial functions in $p$ whose support is connected and equal to $\{ |p| < \Upsilon\}$. This includes important physical examples such as Fermi gases, Bose gases, and Boltzmann equilibria at positive temperatures. It also includes the important Fermi gases at zero temperature, namely compactly supported steady states, up to the above regularity requirement on $\mu(\cdot)$. 

\subsection{Main results}

We are now ready to state the first main result of this paper.

\begin{theorem}\label{theo-mainLandau} 
	Let $\mu(|p|^2)$ be the equilibrium as described in Section \ref{sec-mu}, with a finite constant density $\rho_\mu = \int_{\RR^d} \mu(|p|^2)\; dp$.  Let $\Upsilon$ be the maximal particle velocity defined as in \eqref{def-Upsilon}, and $\kappa_{0}$ be the non-negative critical wave number defined through \eqref{def-introkappa0}. The corresponding spacetime symbol $D(\lambda,k)$ for the linearized Hartree equation \eqref{linHartree} satisfies the following:
	\begin{itemize}
		\item For each $k \in \mathbb{R}^d$, $D(\lambda, k)$ is analytic and nonzero in $\Re \lambda >0$.
		\item For $0 \le |k| \le \kappa_{0}$, there are exactly two pure imaginary zeros $\lambda_\pm= \pm i \tau_*(k)$ of the dispersion relation $ D(\lambda,k) = 0$, where $\tau_*(k)$ is sufficiently regular, strictly increasing in $|k|$, with\footnote{Note that by definition \eqref{def-introkappa0}, $\lim_{\Upsilon \to \infty}2\kappa_0\Upsilon = \sqrt{2\rho_\mu} $, and so the identities in \eqref{tau-a01} are consistent when $\kappa_0 = 0$ and $\Upsilon = \infty$.} 
		\begin{equation}\label{tau-a01}
			\tau_{*}(0)=\sqrt{2\rho_\mu} , \qquad \tau_*(\kappa_0) = 2\kappa_0\Upsilon + \kappa_0^2.
		\end{equation}
In particular, there are positive constants $c_0$ and $C_0$ such that 
		\begin{equation}\label{KGdisp}
			c_0 |k|\le \tau'_*(k) \le C_0|k|, \qquad  c_0 \le \tau''_*(k) \le C_0,
		\end{equation}
		for $0\le |k|\le \kappa_0$. 
	
		\item There is a constant $\delta_{0} >0$ so that the zeros $\lambda_{\pm}(k)$ of $D(\lambda, k) =0$ can be extended for $\kappa_{0} \le |k| \le \kappa_{0} +\delta_{0}$, in such a way that they are sufficiently regular in $|k|$, $\lambda_\pm(\kappa_0) = \pm i \tau_*(\kappa_0)$, and $\tau_*(k) = \mp \Im \lambda_\pm(k)$ satisfy \eqref{KGdisp} for $\kappa_{0} \le |k| \le \kappa_{0} +\delta_{0}$.
				 In addition, the following Landau's law of damping holds:
		
		\begin{itemize}
			
			\item[(i)] If $\Upsilon = \infty$, then for $|k| \le \delta_{0}$, we have 
			\begin{equation}\label{Landau-intro}
				\Re \lambda_\pm(k) = {\pi}[u^2 \varphi'(u)]_{u = \nu_* (k)} ( 1 + \mathcal{O}(|k|))
			\end{equation}
			where $\nu_*(k) = \frac{\tau_0 }{2|k|}$ is the phase velocity, with $\tau_0 = \sqrt{2\rho_\mu}$.

			\item[(ii)] If $\Upsilon <\infty$, then for $\kappa_{0} \le |k| \le \kappa_{0} + \delta_{0}$, we have  
			\begin{equation}\label{Landau-cpt-intro}
				\Re \lambda_\pm(k) = - {\frac{\pi}{2 \widetilde{\kappa}_{1}(\kappa_{0}) |k|}} \left[ \varphi(u - \frac{|k|}{2}) - \varphi(u + \frac{|k|}{2}) \right]_{u = \nu_*(k)} (1 + \cO(|k| - \kappa_0)),
			\end{equation}
			where $\nu_*(k) =  \Upsilon + \frac{|k|}{2} - \theta_0(|k| - \kappa_{0})$ is the phase velocity for some positive constant $\theta_0$ (see Theorem \ref{theo-Landau} for the precise value of $\theta_0$).

		\end{itemize}
		\item For any $\delta >0$, there is a constant $c_{\delta}>0$ such that
		\begin{equation}\label{Penrose-intro}
			\inf_{|k| \ge \kappa_{0} + \delta} \inf_{ \Re \lambda \ge 0} |D(\lambda,k)|  \ge c_{\delta}.
		\end{equation}
	\end{itemize}
\end{theorem}

Theorem \ref{theo-mainLandau} confirms the existence of plasmons and the physical discussions given in the previous sections. As a matter of facts, plasmons exist for more general long-range pair interaction potentials $w(x)$, provided that its Fourier transform $\hw(k)$ satisfies
\begin{equation}\hw(k)\ge 0, \qquad \hat{w}(0) = \infty
, \qquad \lim_{|k|\to \infty}\hat{w}(k) <\infty, \qquad \hw'(k) \le 0.\end{equation} 
See the proof of Theorem \ref{theo-LangmuirE} and Remark \ref{rem-genw}. However, their dispersive properties may alter from those stated in Theorem \ref{theo-mainLandau}.  
In addition, our results assert that the Langmuir oscillations survive Landau damping for all the wave numbers less than the survival threshold $\kappa_0$, while Landau's law of damping is present and explicitly computed at $\kappa_0$, see \eqref{Landau-intro}-\eqref{Landau-cpt-intro}. Beyond the threshold $\kappa_0$, the strong stability condition \eqref{Penrose-intro} is ensured, and the free Hartree dynamics may be a good approximation for the large time behavior of solutions to the linearized Hartree problem.

Our next main result captures the leading oscillatory behavior of the linearized Hartree density and  provides quantitative decay estimates of the remainder. Precisely, we obtain the following.  

\begin{theorem}\label{theo-main} 
Fix $N_{2} > d+2$ and $N_3>d$. Let $\gamma_{0}$ be the initial matrix operator whose integral kernel $\gamma_0(x,y)$ satisfies 
\begin{equation}\label{assmp-g0} 
\sum_{|\alpha|+|\beta|\le N_3 }\| (1+|x|+|y|)^{N_2} \partial_x^\alpha \partial_y^\beta\gamma_0\|_{L^1_{x,y}} < \infty. 
\end{equation}
Then, the density function associated with $\gamma(t)$ of the linearized Hartree equations \eqref{linHartree} can be written as
	\begin{equation}
		\rho(t,x) = \sum_{\pm} \rho^{osc}_{\pm}(t,x) + \rho^{r}(t,x),
	\end{equation}
where
	\begin{equation}
		\begin{aligned}
			\|\rho^{osc}_{\pm}(t) \|_{L^{p}_{x}} &\lesssim \langle t \rangle^{-d (1/2 - 1/p)}, \qquad p \in [2,\infty], \\
			\|\rho^{r}_{\pm}(t) \|_{L^{p}_{x}} & \lesssim \langle t \rangle^{-2-d (1- 1/p)} \log (1+t), \qquad p \in [1,\infty].
		\end{aligned}
	\end{equation}
\end{theorem}

Theorem \ref{theo-main} asserts that the leading dynamics of the Hartree density is indeed oscillatory and dispersive, which behaves like a Klein-Gordon wave of the form $e^{\pm i t\sqrt{1 - \Delta_x} }$, while the remainder decays faster and behaves like that of phase mixing for the free gas, whose derivatives gain extra decay. 
In fact, using Theorem \ref{theo-mainLandau}, we are able to describe precisely the oscillatory component $\rho^{osc}_\pm(t,x)$, namely 
\begin{equation}\label{true-Eosc}
\begin{aligned}
\rho^{osc}_\pm(t,x) &= G_\pm^{osc}(t)  \star_{t,x} \rho^0(t,x) 
\end{aligned}
\end{equation}
where $G_\pm^{osc}(t,x)$ is the oscillatory Green function whose Fourier transform is equal to $e^{\lambda_\pm(k)t} a_\pm(k)$, for some compactly supported and smooth Fourier symbols $a_\pm(k)$ and for dispersion functions $\lambda_\pm(k)$ constructed as in Theorem \ref{theo-mainLandau}. Here, the source $\rho^0(t,x)$ is the density generated by the free Hartree dynamics. The stated dispersive estimates resemble from those for the Klein-Gordon's type dispersion $e^{\lambda_\pm(k)t} a_\pm(k)$. However, difficulties arise, since the source term $\rho^0(t,x)$ in the spacetime convolution does not have sufficient decay, see the Strichartz estimates \eqref{free-Sch}, not to mention the lack of pointwise estimates. In Section \ref{sec-freeHartree}, we establish pointwise {\em phase mixing} estimates for the free Hartree dynamics under localized assumptions on the initial data \eqref{assmp-g0}, which are sufficient to derive the Klein-Gordon's dispersive behavior of the oscillatory component. We stress that the Green function and the density representation are obtained for general initial data that may only belong to Schatten spaces without the localization \eqref{assmp-g0}, see Propositions \ref{prop-Greenx} and \ref{prop-densityrep}, respectively.


\subsection{Notation}

We use the notation $\widehat{~\cdot~}$ and $\widetilde{~\cdot~}$ to denote the Fourier transform in $\RR^d$ and the Laplace-Fourier transform in $\RR_+ \times \RR^d$, namely   
\[
\widehat{f}(k) = \int_{\mathbb{R}^{d}} e^{-ix \cdot k} f(x) dx, 
\qquad	\widetilde{f}(\lambda,k) = \int_{0}^{\infty} e^{-\lambda t} \widehat{f}(t,k) dt ,
\]
for $\lambda \in \CC$ and $k \in \RR^d$. 

\section{Fourier-Laplace approach}\label{sec-FT}

In this section, we introduce the Fourier-Laplace approach to study the linearized problem \eqref{linHartree}. 

\subsection{Resolvent equation}

We first derive the spacetime Fourier-Laplace symbol for the linearized operator. See also \cite{Lewin2} for the derivation via the spacetime Fourier approach. Precisely, we obtain the following. 

\begin{lemma}\label{lem-resolvent} Let $\rho(t,x)$ be the density of the linearized Hartree equations \eqref{linHartree}, and $\Trho(\lambda,k)$ be the Fourier-Laplace transform of $\rho(t,x)$. Then, for each $k\in \mathbb{R}^d \setminus \{ 0 \}$ and $\lambda \in \mathbb{C}$, there hold
\begin{equation}\label{resolvent}
\begin{aligned}
\Trho(\lambda,k) &= \frac{1}{D(\lambda,k)} \Trho^0(\lambda,k),\qquad \Trho^0(\lambda,k) :=\int \frac{\hgamma_0(k-p,p)}{\lambda -2  ik \cdot  p + i|k|^2}\;dp
\end{aligned}
\end{equation}
where 
\begin{equation}\label{def-Dlambda}
\begin{aligned}
D(\lambda,k)
&:= 1 + i\hw(k)\int \frac{\mu(|p|^2) - \mu(|k-p|^2)}{\lambda -2ik \cdot p + i|k|^2} \;dp .
\end{aligned}
\end{equation}
\end{lemma}

Observe that $\Trho^0(\lambda,k)$ defined as in \eqref{resolvent} is exactly the Fourier-Laplace transform of the density $\rho^0(t,x)$ generated by the free Hartree dynamics $\gamma^0 (t)= e^{it\Delta}\gamma_0 e^{-it\Delta}$. Thus, the formulation \eqref{resolvent} asserts that the linearized Hartree density can be computed in terms of the free Hartree density, upon convoluting with the integral kernel of the symbol $\frac{1}{D(\lambda,k)}$. Note that by definition \eqref{def-Dlambda}, we may write for $\Re \lambda > 0$, 
\begin{equation}\label{eq:dlambdak}
	\begin{aligned}
		D(\lambda, k) 
		&= 1 +i \widehat{w}(k) \int_{\RR^{d}} \left( \frac{\mu(|p|^2)}{\lambda -2  ik \cdot  p + i|k|^2} - \frac{\mu(|p|^2)}{\lambda -2  ik \cdot  p - i|k|^2} \right) dp  \\
		&= 1 + i \widehat{w}(k) \int_{\RR^{d}} \int_{0}^{\infty}
		( e^{-(\lambda -2  ik \cdot  p + i|k|^2)t} - e^{-(\lambda -2  ik \cdot  p - i|k|^2)t}) \mu(|p|^2) \, dt dp \\
		&= 1 + i \widehat{w}(k)\int_{0}^{\infty} e^{-\lambda t} (e^{-i|k|^2 t} - e^{i|k|^2 t}) \int_{\RR^d} e^{2itk \cdot p} \mu(|p|^2) dp dt \\
		&= 1 + 2 \widehat{w}(k) \int_{0}^{\infty} e^{-\lambda t} \sin (t|k|^2)\widehat{\varphi}(2t|k|) \, dt,
	\end{aligned}
\end{equation}
which is the spacetime symbol derived in \cite{Lewin2}. Here, $\widehat{\varphi}(|k|)$ denotes the Fourier transform of the marginal function $\varphi(u) = \int_{\RR^{d-1}}\mu(u^2 + |w|^2)\; dw$.

\begin{proof}[Proof of Lemma \ref{lem-resolvent}]
	Let $\gamma(t,x,y)$ and $\mu(x,y)$ be the integral kernel of the density matrix operators $\gamma(t)$ and $\mu =\mu(-\Delta)$, respectively. We get
	\begin{align*}
		i \partial_{t} \gamma(t,x,y) 
		= (-\Delta_{x} + \Delta_{y})\gamma(x,y)
			+ (V(x) - V(y)) \mu(x,y)
	\end{align*}
	where $V = w \star_x \rho$. We take the Fourier transform with respect to $x$ and $y$, with dual variables $k$ and $p$, yielding 
	\begin{align*}
		i \partial_{t} \widehat{\gamma}(t,k,p) 
		= (|k|^2 - |p|^2)\widehat{\gamma}(t,k,p)
			+ \int \widehat{V}(l) \left( \widehat{\mu}(k-l, p) - \widehat{\mu}(k, p-l)  \right) dl.
	\end{align*}
	Observe that
	\begin{align*}
		\mu(x,y) = \int e^{i(x-y) \cdot k'} \mu(|k'|^2) dk', \,
	\end{align*}
which is well-defined, thanks to the decay property of $\mu(|k|^2)$. 
Hence, 
	\begin{align*}
		\widehat{\mu}(k, p) 
		&=\iiint e^{-ik \cdot x- i p \cdot y} e^{i(x-y) \cdot k'} \mu(|k'|^2) dk' dx dy \\
		&=\int \delta_{k=k'}\delta_{p=-k} \, \mu(|k'|^2) dk' =\mu(|k|^2) \delta_{p=-k}.
	\end{align*}
	Therefore, together with the fact that $\widehat{V}(t,k) = \widehat{w}(k)\widehat{\rho}(t,k)$, we get
	\begin{equation*}
		i \partial_{t} \widehat{\gamma}(t,k,p) 
		= (|k|^2 - |p|^2)\widehat{\gamma}(t,k,p)
		+ \widehat{w}(p +k) \widehat{\rho}(t, p + k) \left( \mu(|p|^2) - \mu(|k|^2) \right).
	\end{equation*}
	Taking the Laplace transform with respect to $t$ with the dual variable $\lambda$, we obtain
	\begin{equation}\label{eq:gammatilde}
		(\lambda + i(|k|^2 -|p|^2)) \widetilde{\gamma}(\lambda, k, p)
		= \widehat{\gamma_{0}}(k ,p) - i \widehat{w}(p +k) \widetilde{\rho}(\lambda,p+k) \left( \mu(|p|^2) - \mu(|k|^2) \right)
.	\end{equation}
	Note that the Fourier transform of the density is given by
	\begin{align*}
		\widehat{\rho}(t,k)
		&=  \iiint e^{ix\cdot (k'+p-k)}\widehat{\gamma}(t, k', p) \, dx dk' dp \\
		&= \iint \delta_{k=k'+p}\widehat{\gamma}(t, k', p) \, dk' dp 
		\\&=  \int \widehat{\gamma}(t, k- p , p) \, dp.
	\end{align*}
	From \eqref{eq:gammatilde}, we have
	\begin{align*}
		(\lambda + i(|k-p|^2 -|p|^2)) \widetilde{\gamma}(\lambda,k-p,p)
		= \widehat{\gamma_{0}}(k-p, p) - i \widehat{w}(k) \widetilde{\rho}(\lambda,k) \left( \mu(|p|^2) - \mu(|k-p|^2) \right).
	\end{align*}
	Dividing both sides by $\lambda +i(|k-p|^2 - |p|^2)$ and integrating over $p$, we get
	\begin{align*}
		\widetilde{\rho}(\lambda,k) 
		&=\int \frac{\widehat{\gamma_{0}}(k-p, p)}{\lambda -2ik \cdot p + i|k|^2} d p
		-  i\widehat{w}(k) \widetilde{\rho}(\lambda,k) \int \frac{\mu(|p|^2) - \mu(|k-p|^2)}{\lambda -2ik \cdot p + i|k|^2} d p,
		 	\end{align*}
leading to \eqref{resolvent}. 
\end{proof}

\begin{lemma}\label{lem-varphi}
Let $\mu(|p|^2)$ be the equilibrium as described in Section \ref{sec-mu}. Set 
\begin{equation}\label{def-marginals}
\varphi(u) =\int_{\mathbb{R}^{d-1}} \mu (u^2 + |w|^2) \, dw.
\end{equation}
Then, there hold 
\begin{itemize}
	\item $\varphi$ is of class $C^{N_{0}}$ where $N_{0} = n_{0} + \frac{d-1}{2} \ge d+7$.
	\item $\varphi$ is an even function on $\mathbb{R}$ and $\varphi'(u) <0$ for $0<u<\Upsilon$. In addition, $|\varphi'(u)| \lesssim |u|$. 
	\item When $\Upsilon = \infty$,
	\begin{equation}\label{decay-varphi}
	| \varphi(u)| \lesssim \langle u \rangle^{-N_{1}}, \qquad \forall u \in \mathbb{R}
	\end{equation}
	where $N_{1} = 2n_{1} - d +1 \ge {2N_{0} +2}$.
	\item When $\Upsilon < \infty$,
	\begin{equation}\label{decay-varphi2}
	\lim_{u \rightarrow \Upsilon} \frac{\varphi(u)}{(\Upsilon - u)^{N_{1}}}
	\end{equation}
	exists and is positive, where $N_{1} = n_{1} + \frac{d-1}{2} \ge {2}$.
\end{itemize}

\end{lemma}
\begin{proof}
By definition, we may set $r = |w|$ and write 
 $$
 \begin{aligned}
 \varphi(u) &=c_d\int_0^\infty\mu (u^2 + r^2) r^{d-2}\; dr
 \end{aligned}$$
for the constant $c_d$ being the measure of the unit ball in $\RR^{d-1}$. Therefore, for $d\ge 4$,  
we compute  $$
 \begin{aligned}
 \varphi'(u) &=c_d\int_0^\infty 2u \mu' (u^2 + r^2) r^{d-2}\; dr 
 = c_d\int_0^\infty u \frac{d}{dr}\mu(u^2 + r^2) r^{d-3}\; dr
 \\& 
 = - c_d(d-3)\int_0^\infty u \mu(u^2 + r^2) r^{d-4}\; dr, 
 \end{aligned}$$
 while for $d=3$, the above calculation yields $ \varphi'(u) = - c_3u \mu(u^2).$ This proves that $\varphi'(u) \le 0$ for all $u\ge 0$ in all dimensions $d\ge 3$. For $d=2$, $\varphi'\le 0$, thanks to the additional assumption that $\mu'\le 0$. Finally, the strict monotonicity of $\varphi(u)$ follows from the assumption that $\mu(u^2) >0$ on $\{0< u<\Upsilon\}$, while the remaining properties are obtained directly from those of $\mu(\cdot)$, see Section \ref{sec-mu}. 
\end{proof}

\begin{lemma}\label{lem-Dlambda} Let $D(\lambda,k)$ be defined as in \eqref{def-Dlambda}. Then, for each $k\not =0$ and $\lambda \in \CC$, we can write 
\begin{equation}\label{def-Dlambda1}
\begin{aligned}
D(\lambda,k) 
= 1 + \frac{\hw(k)}{2|k|} \left[ \cH \left(\frac{-i \lambda +|k|^2}{2|k|} \right) - \cH \left( \frac{-i\lambda -|k|^2}{2|k|} \right)   \right] ,
\end{aligned}
\end{equation}
where
\begin{equation}\label{def-cH}
 \cH(z) 
= \int_{\RR} \frac{\varphi(u)}{z-u} \, du,
\end{equation}
for $\varphi(u)$ defined as in \eqref{def-marginals}. In particular, $D(\lambda,k)$ is analytic in $\Re\lambda>0$, and 
\begin{equation}\label{unibd-cH} \left| \partial_z^n\cH \left( \frac{-i \lambda \pm |k|^2}{2|k|} \right) \right| \lesssim 1,
\end{equation}
uniformly for $k\in \mathbb{R}^d$, $\Re \lambda \ge 0$, and $0\le n< N_0 -1$.  
\end{lemma}

\begin{proof} 
In view of \eqref{def-Dlambda}, for $k\not =0$, we introduce the change of variables  
\begin{equation}\label{change-vu}
u = \frac{k\cdot p}{|k|} , \qquad w = p -   \frac{(k\cdot p)k}{|k|^2},
\end{equation}
for $p\in \RR^d$, with the Jacobian equal to one. Note that $|p|^2 = u^2 + |w|^2$. We then write 
\[
\begin{aligned}
D(\lambda,k)
&= 1 + i\hw(k)\int_{\RR^d} \left( \frac{1}{\lambda -2  ik \cdot  p + i|k|^2} - \frac{1}{\lambda -2  ik \cdot  p - i|k|^2} \right) \mu(|p|^2)\; dp 
\\&= 1 + \frac{\hw(k)}{2|k|}\int_{\RR^d} \left( \frac{1}{\frac{-i\lambda + |k|^2}{2|k|} - \frac{k \cdot  p}{|k|}} - \frac{1}{\frac{-i\lambda - |k|^2}{2|k|} - \frac{k \cdot  p}{|k|} } \right) \mu(|p|^2)\; dp 
\\&= 1 + \frac{\hw(k)}{2|k|}\int_{\RR} \left( \frac{1}{\frac{-i\lambda + |k|^2}{2|k|} - u} - \frac{1}{\frac{-i\lambda - |k|^2}{2|k|} -u } \right) \left(\int_{w \in k^{\perp}} \mu (u^2 +|w|^2) \, dw \right) du .
\end{aligned}
\]
For each $u \in \RR$, setting 
$$ \cH(z) 
= \int_{\RR} \frac{\varphi(u)}{z-u} \, du, \qquad	\varphi(u) =\int_{w \in k^{\perp}}\mu (u^2 + |w|^2) \, dw,
$$
we thus obtain  \eqref{def-Dlambda1}. Observe that since $\mu(\cdot)$ is radial, the integral over the hyperplane $k^\perp$ is equal to that on $\RR^{d-1}$, and $\varphi(u)$ is thus independent of $k\not =0$, as defined in \eqref{def-marginals}.   
It remains to study the function $\cH(z)$, which is the Hilbert transform of $\varphi$. By definition, $\cH(z)$ is well-defined and analytic in $\{ \Im z < 0\}$, since $\varphi(u)$ decays rapidly to zero (in particular, $\varphi \in L^1$). Moreover, we may write 
\begin{equation}\label{redef-cH}
\cH(z) 
= \int_{\mathbb{R}} \frac{\varphi(u)}{z-u} \, du
= i\int_{0}^{\infty} e^{-iz t} \int_{\mathbb{R}} e^{iut} \varphi(u) \, dudt
= i \int_{0}^{\infty}  e^{-izt} \widehat{\varphi}(t) \, dt,
\end{equation}
where $\widehat{\varphi}$ denotes the Fourier transform of $\varphi$ (recalling that $\varphi(u)$ is even in $u$). Integrating by parts in $u$ and using the regularity of $\varphi$, we obtain 
\begin{equation}\label{bounds-Nt}
|\partial_t^n \widehat{\varphi}(t)|\le C_n\langle t \rangle^{-N_0} ,
\end{equation}
for any $0 \le n< N_1$. This gives 
\begin{equation}\label{bd-cH1}
|\cH(z)| 
\le  C_0 \int_0^\infty e^{\Im z t}  \langle t \rangle^{-N_0} dt \lesssim 1,
\end{equation}
for any $\Im z \le 0$. Here
	\[
		\mathcal{H}(\tilde{\tau}) := \lim_{\tilde{\gamma} \rightarrow 0^{+}} \mathcal{H}(-i \tilde{\gamma} + \tilde{\tau})
	\]
for $|\tilde{\tau}| < \Upsilon$. Similarly
\[
	\partial_{z}^{n}  \mathcal{H}(z) = i \int_{0}^{\infty} e^{-izt} (-it)^{n} \widehat{\varphi}(t) \, dt 
\]
so that
\[
	|\partial_{z}^{n} \mathcal{H}(z)| \le C_{0} \int_{0}^{\infty} e^{\Im z  t} \langle t \rangle^{-N_{0} +n } \, dt \lesssim 1
\]
for $\Im z \le 0 $ and $0 \le n < N_{0}-1$. The lemma follows. 
\end{proof}

\subsection{Spectral stability}\label{sec-spectral}

In this section, we prove that there are no solutions in the right half plane $\{\Re \lambda>0\}$. Namely, we obtain the following.

\begin{proposition}\label{prop-nogrowth} Let $\mu(|p|^2)$ be the equilibrium as described in Section \ref{sec-mu}. Then, for any $\delta_{0} >0$ and $k \in \mathbb{R}^{d} \setminus \{0 \}$, there is a positive constant $c_{0}$ such that
	\begin{equation}\label{specs-D}
		\inf_{\Re \lambda \ge \delta_{0}} | D(\lambda, k )| \ge c_{0}.
	\end{equation}
In particular, $c_0$ may depend on $k$ for $|k|\le 1$, but is independent of $k$ for $|k|\ge 1$.
\end{proposition}

\begin{proof} 
Recall from \eqref{eq:dlambdak} that 
$$
		D(\lambda, k) 
= 1 + 2 \widehat{w}(k) \int_{0}^{\infty} e^{-\lambda t} \sin (t|k|^2)\widehat{\varphi}(2t|k|) \, dt.
$$
Therefore, recalling \eqref{bounds-Nt}, for $\Re \lambda \ge 0$, we have 
$$
|D(\lambda, k) - 1| \lesssim \widehat{w}(k) \int_{0}^{\infty} \langle kt \rangle^{-N_0} \, dt  \lesssim |k|^{-1}\widehat{w}(k),
$$
which in particular tends to zero as $|k|\to \infty$, uniformly in $\Re\lambda\ge 0$. This proves that there are positive constant $K_0,c_0$ so that 
	\begin{equation}\label{specs-Dlargek}
		\inf_{|k|\ge K_0}\inf_{\Re \lambda \ge 0} | D(\lambda, k )| \ge c_{0}.
	\end{equation}
In particular, this gives \eqref{specs-D} for large $|k|$. On the other hand, for $\lambda \not =0$, write $e^{-\lambda t} = -\lambda^{-1}\partial_te^{-\lambda t}$ and integrate the integral by parts in $t$, yielding 
$$
D(\lambda,k)-1 =  \frac{2 \widehat{w}(k)}{\lambda} \int_{0}^{\infty} e^{-\lambda t} \Big( |k|^2\cos(t|k|^2)\widehat{\varphi}(2t|k|) + 2|k|\sin (t|k|^2)\partial_t \widehat{\varphi}(2t|k|) \Big)\, dt,
$$
noting there are no boundary terms, since $\widehat{\varphi}(t)$ decays rapidly in $t$. Using again \eqref{bounds-Nt}, we bound 
$$
|D(\lambda,k)-1| \lesssim  \frac{\widehat{w}(k)}{|\lambda|} \int_{0}^{\infty} (1+|k|)|k| \langle kt\rangle^{-N_0} \, dt \lesssim \frac{(1+|k|)\widehat{w}(k)}{|\lambda|}.
$$
Therefore, for each $k \not =0$, $D(\lambda, k) \rightarrow 1$, when $|\lambda| \rightarrow \infty$ with $\Re \lambda \ge 0$, proving \eqref{specs-D} for large $\lambda$. Thus, we focus on bounded $k$ and bounded $\lambda$. By compactness, for any $k\in \mathbb{R}^d \setminus\{ 0 \}$, it suffices to prove that $D(\lambda,k) \neq 0$ for $\Re \lambda >0$. Equivalently, from \eqref{def-Dlambda1} with $z = - \frac{i\lambda + |k|^2}{2|k|}$, we show that the identity 
\begin{equation}\label{eq:posrealpt}
	\cH (z) - \cH ( z + |k|)  =  \frac{2|k|}{\hat{w}(k)}
\end{equation}
for some complex $z$ with $\Im z < 0$ would lead to a contradiction.
Indeed, by definition \eqref{def-cH}, we first note that
\[
	\begin{aligned}
	\cH(z) 
	&=  \frac{1}{2\pi} \int_{\RR} \frac{\varphi(u)}{z-u} \, du 
	 =  \frac{1}{2\pi}\int_{\mathbb{R}} \frac{(\Re z-u) \varphi(u)}{|z-u|^2} \, du - \frac{i}{2\pi} \Im z \int_{\RR} \frac{\varphi(u)}{|z-u|^2} \, du .
	\end{aligned}
\]
The imaginary part of \eqref{eq:posrealpt} reads $\Im \cH(z)  - \Im \cH(z+|k|) =0$, which implies 
\[
	\int_{\mathbb{R}} \frac{\varphi(u)}{|z-u|^2}  \, du -  \int_{\mathbb{R}}  \frac{\varphi(u)}{|z+|k|-u|^2}  \, du =0,
\]
since $\Im (z+|k|) = \Im z <0$. Writing $\lambda = \gamma + i \tau$ and $z = -\frac{i\lambda}{2|k|} - \frac{|k|}{2}$, we have 
\begin{align*}
	\int_{\mathbb{R}} \frac{\varphi(u)}{ \frac{\gamma^2}{4|k|^2} + | \frac{\tau}{2|k|} - \frac{|k|}{2} -u|^2 } \, du
	- \int_{\mathbb{R}} \frac{\varphi(u)}{ \frac{\gamma^2}{4|k|^2} + | \frac{\tau}{2|k|} + \frac{|k|}{2} -u |^2 } \, du=0.
\end{align*}
Using the change of variable $u'= -u$ in the first integral, we get
\begin{align*}
	\int_{\mathbb{R}} \frac{\varphi(u)}{ \frac{\gamma^2}{4|k|^2} + | -\frac{\tau}{2|k|} + \frac{|k|}{2} -u|^2 } \, du
	- \int_{\mathbb{R}} \frac{\varphi(u)}{ \frac{\gamma^2}{4|k|^2} + | \frac{\tau}{2|k|} + \frac{|k|}{2} -u |^2 } \, du=0.
\end{align*}
Introducing the change of variable $u' = \frac{|k|}{2} - u$ and then combining two integrals, we get
\begin{align*}
	 \tau \int_{\mathbb{R}} \frac{ u \varphi(\frac{|k|}{2} +u )}{g(\gamma, \tau, k, u)} \, du = 0
\end{align*}
where
	\begin{align*}
		g(\gamma, \tau, k, u) = \left(\frac{\gamma^2}{4|k|^2} + \left| -\frac{\tau}{2|k|} +u \right|^2\right)\left(\frac{\gamma^2}{4|k|^2} + \left| \frac{\tau}{2|k|} + u \right|^2 \right)
	\end{align*}
is the denominator which is even in $u$, and strictly positive for $k \in \mathbb{R}^{d} \setminus \{0 \}$, and $\gamma >0$. We conclude that $\tau =0$ because $\varphi$ is even and $\varphi'(u) <0$ for $0< u< \Upsilon$ so that
\begin{align*}
	 \int_{\mathbb{R}} \frac{ u \varphi(\frac{|k|}{2} +u)}{g(\gamma, \tau, k, u)} \, du
	 = \int_{0}^{\infty} \frac{u ( \varphi( \frac{|k|}{2} - u) - \varphi( \frac{|k|}{2} + u))}{g(\gamma, \tau, k ,u )} \, du >0
\end{align*}
On the other hand, the real part of \eqref{eq:posrealpt} is reduced to 
\[
	-\int_{\mathbb{R}} \frac{u \varphi(u)}{|z-u|^2} \, du - \int_{\mathbb{R}} \frac{(|k|-u) \varphi(u)}{|z+|k|-u|^2} \, du 
	=  \frac{2|k|}{\hat{w}(k)}.
\]
since $\Re (z+|k|) = \Re z + |k|$. When $\tau =0$, we obtain that 
\[
	-\int_{\mathbb{R}} \frac{u \varphi(u)}{ \frac{\gamma^2}{4|k|^2} + |   -\frac{|k|}{2} -u|^2} \, du - \int_{\mathbb{R}} \frac{(|k|-u)\varphi(u)}{ \frac{\gamma^2}{4|k|^2} + |   \frac{|k|}{2} -u|^2} \, du 
	=  \frac{2|k|}{\hat{w}(k)}.
\]
Using the change of variable $u'=-u$ in the second integral, it is simplied to
\[
	- \int_{\mathbb{R}} \frac{(|k| +2u) \varphi(u)}{\frac{\gamma^2}{4|k|^2} + |   \frac{|k|}{2} +u|^2}  \, du =  \frac{2|k|}{\hat{w}(k)}
\]
Since $\varphi$ is even, the left hand side becomes
\begin{align*}
	- \int_{\mathbb{R}} \frac{(|k| +2u) \varphi(u)}{\frac{\gamma^2}{4|k|^2} + |   \frac{|k|}{2} +u|^2} \, du
	&= - \int_{\mathbb{R}} \frac{2u \varphi(u - \frac{|k|}{2})}{\frac{\gamma^2}{4|k|^2} + u^2} \, du 
	= \int_{0}^{\infty} \frac{2u ( \varphi( u + \frac{|k|}{2}) - \varphi( u - \frac{|k|}{2} ))}{\frac{\gamma^2}{4|k|^2} + u^2} \, du,
\end{align*}
which is strictly negative, since $\varphi'(u) < 0$ for $0 < u< \Upsilon$. This yields a contradiction, which ends the proof of Proposition \ref{prop-nogrowth}.
\end{proof}

\subsection{No embedded eigenvalues}\label{sec-emlambda}

In this section, we show that pure imaginary solutions $\lambda = i \tau$ of $D(\lambda, k)=0$ do not exist for $|\tau| < 2k \Upsilon + |k|^2$. In particular, when $\Upsilon = \infty$, there are no pure imaginary solutions for $D(\lambda, k ) =0$.
Precisely, we obtain the following. 

\begin{proposition}\label{prop:nonexistence} Let $\mu(|p|^2)$ be the equilibrium as described in Section \ref{sec-mu}, and $\Upsilon$ be as in \eqref{def-Upsilon}. For $k \in \mathbb{R}^d \setminus \{ 0 \}$, the dispersion relation $D(\lambda,k)=0$ does not admit solutions of the form $\lambda = i\tau$ with $|\tau| < 2|k|\Upsilon + |k|^2$. Moreover, for any compact subset $U$ of $\{ |\tilde \tau| < 2\Upsilon \}$, there exists a positive constant $c_{U}$, independent of $k$,
so that
	\begin{equation}\label{eq:lowbd}
		|D(i \tilde\tau|k|, k)| \ge c_{U} \left( 1 + |\hw(k)|\right)
	\end{equation}
	uniformly in $\tilde \tau \in U$. 
\end{proposition}

\begin{proof} Fix $k\not =0$, and let $\lambda = (\tilde \gamma + i\tilde \tau) |k|$ with $|\tilde \tau| < 2\Upsilon + |k|$. Recalling \eqref{def-Dlambda1}, we compute 
$$
\begin{aligned}
D((\tilde \gamma + i\tilde \tau) |k|,k) 
= 1 + \frac{\hw(k)}{2|k|} \left[ \cH ( \frac{-i \tilde\gamma + \tilde\tau+|k|}{2} )  - \cH( \frac{-i\tilde\gamma + \tilde\tau-|k|}{2})  \right] ,
\end{aligned}
$$
in which by definition \eqref{def-cH}, 
\[
\cH ( \frac{-i\tilde\gamma + \tilde\tau\pm |k|}{2} ) 
= \int_{\{|u| < \Upsilon\}} \frac{\varphi(u)}{\frac{-i\tilde\gamma + \tilde\tau\pm |k|}{2} - u} \;du.
\]
Using the Plemelj's formula, we obtain
\begin{equation}\label{Plemelj}
 \cH ( \frac{ \tilde\tau\pm |k|}{2} ) := \lim_{\tilde\gamma \to 0^+}\cH ( \frac{-i\tilde\gamma + \tilde\tau\pm |k|}{2} ) 
=  PV \int_{\{|u|<\Upsilon\}} \frac{\varphi(u)}{\frac{\tilde{\tau} \pm |k|}{2} - u } du  +i\pi  \varphi(\frac{\tilde{\tau} \pm |k|}{2}) ,
\end{equation}
for $|\tilde \tau| < 2\Upsilon + |k|$, where $PV$ denotes the Cauchy principal value. Note that in case when the point $\frac{\tilde{\tau} \pm |k|}{2}$ does not belong to the support of $\varphi(u)$, the PV integral becomes the regular integration. Combining, we thus obtain 
\begin{equation}\label{form-Ditau2}
\begin{aligned}	
D(i\tilde \tau |k|,k) 
	= 1 &+ \frac{\hat{w}(k)}{2|k|} 
		\left[ PV \int_{\{|u|<\Upsilon\}} \frac{\varphi(u)}{\frac{\tilde\tau + |k|}{2} -u } du 
			- PV \int_{\{|u|<\Upsilon\}} \frac{\varphi(u)}{\frac{\tilde\tau - |k|}{2} - u} du \right]\\[3pt]
		& + \frac{i\pi\hat{w}(k)}{2|k|} 
		\left[ \varphi \left( \frac{\tilde\tau + |k|}{2} \right) - \varphi \left(\frac{\tilde\tau - |k|}{2} \right) \right].
\end{aligned}
\end{equation}
Therefore, we may obtain a lower bound on $D(i\tilde\tau|k|, k)$ by its imaginary part, giving 
$$
\begin{aligned}
|D(i\tilde \tau |k|,k) |
&\ge \frac{\pi\hat{w}(k)}{2|k|} 
\Big|\varphi \left( \frac{\tilde\tau + |k|}{2} \right) - \varphi \left(\frac{\tilde\tau - |k|}{2} \right) \Big|
\ge \frac{\pi\hat{w}(k)}{2|k|} 
\Big|\int_{\frac{\tilde\tau - |k|}{2} }^{\frac{\tilde\tau + |k|}{2} } \varphi'(x)\; dx \Big|.
\end{aligned}$$
Recall from Lemma \ref{lem-varphi} that $\varphi(u)$ is even in $u$ and strictly decreasing for $0 < u < \Upsilon$ (and by definition of $\Upsilon$, $\varphi(u) =0$ for $|u|\ge \Upsilon$). 
Therefore, the above never vanishes for $0<|\tilde \tau| <2\Upsilon + |k|$, since the integration domain 
$$(\frac{\tilde\tau - |k|}{2},\frac{\tilde\tau + |k|}{2} ) \cap (-\Upsilon, \Upsilon) $$
is non empty and non symmetric (except for $\tilde \tau=0$). In addition, for $|\tilde \tau| <2\Upsilon$, 
\begin{equation}\label{low-Dtau}
\begin{aligned}
|D(i\tilde \tau |k|,k)|
\ge  \frac{\pi\hat{w}(k)}{4} |\varphi'(\frac{\tilde\tau}{2})|
\end{aligned}\end{equation}
which again never vanishes except for $\tilde \tau=0$, giving \eqref{eq:lowbd} on $U\setminus(-\delta_0,\delta_0)$ for any positive constant $\delta_0$. On the other hand, at $\tilde \tau =0$, we note that $D(0,k)$ is real-valued and computed by 
\begin{align*}
	D(0,k)
	&= 1 + \frac{\hat{w}(k)}{2|k|} 
	\left[PV \int_{\{|u|<\Upsilon\}} \frac{\varphi(u)}{\frac{|k|}{2} - u} du 
	- PV \int_{\{|u|<\Upsilon\}} \frac{\varphi(u)}{-\frac{|k|}{2}-u} du \right].
\end{align*}
Consider first the case when $|k|\ge 2\Upsilon$, for which $PV$ integrals become the usual integration, giving 
\begin{align*}
	D(0,k)
	&= 1 + \frac{\hat{w}(k)}{2} 
	\int_{\{|u|<\Upsilon\}} \frac{\varphi(u)}{\frac{|k|^2}{4} - u^2} du \ge 1 .
\end{align*}
 Note that since $|k|\ge 2\Upsilon$, we have $|\hw(k)| \lesssim 1$, and so $D(0,k)\gtrsim 1 + |\hw(k)|$ in this case. It remains to study the case when $|k|\le 2\Upsilon$. By definition, we compute 
\begin{align*}
	PV \int_{\{ |u| < \Upsilon \}} \frac{\varphi(u)}{\frac{|k|}{2} - u} \, du 
	&= \lim_{ \epsilon \rightarrow 0^{+}} \left[ \int_{-\Upsilon}^{\frac{|k|}{2} - \epsilon} \frac{\varphi(u)}{\frac{|k|}{2} - u} \, du + \int_{\frac{|k|}{2} + \epsilon}^{\Upsilon} \frac{\varphi(u)}{\frac{|k|}{2} - u} \, du  \right] \\[3pt]
	&= \lim_{ \epsilon \rightarrow 0^{+}} \left[ \int_{\epsilon}^{\Upsilon + \frac{|k|}{2}} \frac{\varphi( u- \frac{|k|}{2} )}{u} \, du - \int_{\epsilon}^{\Upsilon - \frac{|k|}{2}} \frac{\varphi(u+\frac{|k|}{2} )}{u}\, du  \right] 
	\\[3pt]
	&= \lim_{\epsilon \rightarrow 0^{+}} 
		\int_{\epsilon}^{\Upsilon - \frac{|k|}{2}} \frac{\varphi(u -\frac{|k|}{2} ) - \varphi( u +\frac{|k|}{2})}{u} \, du + \int_{\Upsilon - \frac{|k|}{2}}^{\Upsilon + \frac{|k|}{2}} \frac{\varphi( u-\frac{|k|}{2} )}{u} \, du 
\\
	&= 		\int_0^{\Upsilon - \frac{|k|}{2}} \frac{\varphi(u -\frac{|k|}{2} ) - \varphi( u +\frac{|k|}{2})}{u} \, du + \int_{\Upsilon - \frac{|k|}{2}}^{\Upsilon + \frac{|k|}{2}} \frac{\varphi( u-\frac{|k|}{2} )}{u} \, du
		,
\end{align*}
upon recalling from Lemma \ref{lem-varphi} that $\varphi'(u)/u$ is regular. Recall that $\varphi(u)$ is even in $u$ and $\varphi'(u) < 0$ for $0 < u < \Upsilon$. In particular, the integrands in the above integrals are non-negative. Therefore, we obtain 
\begin{align*}
	D(0,k)
&= 1 + \frac{\hat{w}(k)}{|k|}  PV \int_{\{|u|<\Upsilon\}} \frac{\varphi(u)}{ \frac{|k|}{2} - u} du
\ge 1 + \frac{\hat{w}(k)}{|k|}  \int_{0}^{\Upsilon - \frac{|k|}{2}} \frac{\varphi(u -\frac{|k|}{2} ) - \varphi( u +\frac{|k|}{2})}{u} \, du
	\end{align*}
which in particular proves $D(0,k)\ge 1$, and hence $D(0,k) \gtrsim 1 + |\hw(k)|$ for $|k|$ being bounded away from zero on which $|\hw(k)|\lesssim 1$. For $|k|\ll1$, the above yields 
\begin{align*}
	D(0,k)
&\ge 1 + \frac{\hat{w}(k)}{2}  \int_{0}^{\Upsilon - \frac{|k|}{2}} \frac{|\varphi'(u)|}{u} \;du.
	\end{align*}
By Lemma \ref{lem-varphi}, $\varphi'(u)/u$ is regular, and so the integral term is strictly positive, yielding $D(0,k) \gtrsim 1 + |\hw(k)|$. Combining with \eqref{low-Dtau}, we obtain the uniform lower bound \eqref{eq:lowbd} as stated.    
\end{proof}

\subsection{Oscillatory modes}\label{sec-oscmode}

In this section, we prove the existence of pure imaginary solutions to the dispersion relation $D(\lambda,k)=0$ for $\lambda = i\tau$, necessarily for $|\tau|\ge 2|k|\Upsilon + |k|^2$ and $\Upsilon< \infty$, thanks to Proposition \ref{prop:nonexistence}. Namely, we obtain the following.

\begin{theorem}\label{theo-LangmuirE} 
Let $\mu(|p|^2)$ be the equilibrium as described in Section \ref{sec-mu}, with a finite constant density $\rho_\mu = \int_{\RR^d} \mu(|p|^2)\; dp$. Then, there is a non-negative constant $\kappa_0$ defined through 
\begin{equation}\label{def-k0-rel}\kappa_0^2 = \frac{1}{2} \int_{|u|<\Upsilon} \frac{\varphi(u)}{(\Upsilon-u)(\Upsilon+\kappa_0-u)} \, du \end{equation}
so that for any $0\le |k| \le \kappa_0$, there are exactly two zeros $\lambda_\pm= \pm i \tau_*(k)$ of the dispersion relation $ D(\lambda,k) = 0$ on the imaginary axis $\{\Re \lambda =0\}$, where $\tau_*(k)$ is sufficiently regular, strictly increasing in $|k|$, with 
\begin{equation}\label{tau-a0}\tau_{*}(0)= \sqrt{2\rho_\mu}
, \qquad \tau_*(\kappa_0) = 2\kappa_0\Upsilon + \kappa_0^2.
\end{equation}
In particular, there are positive constants $c_0$ and $C_0$ such that 
\begin{equation}\label{lowerbound-taustarDE} 
 c_0 |k|\le \tau'_*(k) \le C_0|k|, \qquad  c_0 \le \tau''_*(k) \le C_0,
\end{equation}
for $0\le |k|\le \kappa_0$.
\end{theorem}

\begin{remark} 
In the case when $\Upsilon =\infty$, Theorem \ref{theo-LangmuirE} in fact holds with $\kappa_0 =0$ so that there are only trivial solutions $\lambda_\pm(0) = \pm i \tau_0$ with $\tau_0^2 = 2\rho_\mu$. 
\end{remark}

\begin{proof} In view of Section \ref{sec-emlambda}, it is enough to consider the case when $\Upsilon<\infty$ and $\lambda = i\tau$, with $|\tau|\ge 2|k| \Upsilon + |k|^2$. In this case, $| \frac{-\tau \pm |k|^2}{2|k|} | > \Upsilon$ so we can compute from \eqref{def-Dlambda1} that
\begin{equation}
\begin{aligned}\label{eq:Displargetau}
	D(i \tau, k)
	&=1 +  \frac{\hat{w}(k)}{2|k|}
	\left[  \int_{ \{|u| < \Upsilon \}} \frac{\varphi(u)}{\frac{\tau +|k|^2}{2|k|} - u} \, du - \int_{ \{|u| < \Upsilon \}} \frac{\varphi(u)}{\frac{\tau -|k|^2}{2|k|} - u} \, du\right] \\[3pt]
	&= 1 - \frac{\hat{w}(k)}{2}
	 \int_{ \{|u| < \Upsilon \}} \frac{\varphi(u)}{ \left( \frac{\tau -|k|^2}{2|k|} - u \right) \left( \frac{\tau +|k|^2}{2|k|} - u \right)} \, du. 
\end{aligned}
\end{equation}
Observe that $D(i \tau, k)$ is real-valued and even in $\tau$, since $\varphi(u)$ is even in $u$. Therefore, it suffices to consider the case when $\tau\ge 2|k| \Upsilon + |k|^2$.  
It follows that
\begin{equation}\label{eq:defPhi}
	\Phi(|k|): = D( i(2|k|\Upsilon +|k|^2),k)
	= 1- \frac{\hat{w}(k)}{2} \int_{|u|<\Upsilon} \frac{\varphi(u)}{(\Upsilon-u)(\Upsilon+|k|-u)} \, du .
\end{equation}
Then $\Phi$ is a real-valued function on $[0, \infty]$, with $\Phi(0) = -\infty$ and $\Phi(\infty) =1$, recalling that $\hw(k) = |k|^{-2}$. Also, we have
\[
	\Phi'(|k|) 
	=  -\frac{\hw'(k)}{2} \int_{ \{ |u| < \Upsilon \}} \frac{\varphi(u)}{(\Upsilon-u)(\Upsilon+|k|-u)} \, du 
	+ \frac{\hw(k)}{2} \int_{ \{|u|<\Upsilon \}} \frac{\varphi(u)}{(\Upsilon-u)(\Upsilon+|k|-u)^2} \, du
\]
is strictly positive so that $\Phi$ is a bijection from $[0, \infty]$ onto $[-\infty, 1]$. Therefore, there exists the unique positive number $\kappa_{0}$ satisfying 
\begin{equation}\label{existence-k0}
\Phi(\kappa_{0}) =0.
\end{equation}
In addition, $\Phi(k) <0 $ for $|k|<\kappa_0$ and $\Phi(k)>0$ for $|k|>\kappa_0$. On the other hand, for $\tau> 2|k| \Upsilon + |k|^2$, we compute 
\begin{align*}
	\partial_{\tau} D(i \tau, k)
	&= \frac{\hat{w}(k)}{2|k|}
	\int_{\{|u| < \Upsilon \}} \frac{ ( \frac{\tau}{2|k|} - u )\varphi(u)}{\left[ (  \frac{\tau}{2|k|} - u )^2 - \frac{|k|^2}{4} \right]^2} \, du,
\end{align*}
which is strictly positive for $\tau > 2|k| \Upsilon + |k|^2$. As a consequence, $D(i \tau, k)$ is a bijection from $[2|k|\Upsilon + |k|^2, \infty]$ onto $[\Phi(|k|) , 1]$.

Now, in view of \eqref{existence-k0}, if $|k| > \kappa_{0}$, then the dispersion relation has no pure imaginary zeros since $\Phi(|k|) >0$. On the other hand, if $0 < |k| < \kappa_{0}$, then $\Phi(|k|) < 0$ so that there exists an unique radial function $\tau_* (k) \in [2|k| \Upsilon +|k|^2, \infty]$ satisfying $D(i \tau_*(k), k) =0$. Since $D(i \tau, k)$ is even in $\tau$, we obtain exactly two solutions $\lambda_{\pm} = \pm i \tau_* (k)$ of $D(\lambda, k)=0$, for $0 \le |k| \le \kappa_{0}$. Observe from \eqref{eq:Displargetau} that $\tau_*(k)$ solves
\begin{equation}\label{eq:taustar}
	\int_{\{|u| < \Upsilon \}} \frac{ \varphi(u)}{ (  \frac{\tau_* (k) }{2} - |k|u )^2 - \frac{|k|^4}{4}} \, du =  2,
\end{equation}
upon recalling that $\hw(k) = |k|^{-2}$. Note that by construction, \eqref{tau-a0} follows. 

Next, we study the dispersive property of $\tau_*(k)$. Indeed, differentiating both sides of \eqref{eq:taustar} with respect to $|k|$, we have
\begin{align*}
	\int_{\{|u| < \Upsilon \}} \frac{ \left( (\frac{\tau_* (k)}{2} -|k|u) (\tau_* ' (k) - 2u) - |k|^3\right)\varphi(u)}{\left( (\frac{\tau_* (k) }{2} - |k|u )^2 - \frac{|k|^4}{4} \right)^2} \, du = 0
\end{align*}
so that
\begin{equation}\label{eq:taustarprime}
	\tau_* '(k) \int_{\{|u| < \Upsilon \}} \frac{  (\frac{\tau_* (k)}{2} -|k|u)\varphi(u)}{\left( (\frac{\tau_* (k) }{2} - |k|u )^2 - \frac{|k|^4}{4} \right)^2} \, du 
	=   \int_{\{|u| < \Upsilon \}} \frac{ (|k|^3 + \tau_* (k) u -2 |k| u^2)\varphi(u)}{ \left( (\frac{\tau_* (k) }{2} - |k|u )^2 - \frac{|k|^4}{4} \right)^2} \, du. 
\end{equation}
For $n  =0, \,1$ and $m \ge 0$, define
\[
	I_{n,m}(k) = \int_{\{ |u| < \Upsilon \}} \frac{  (\frac{\tau_* (k)}{2} -|k|u)^n u^m \varphi(u)}{ \left( (\frac{\tau_* (k) }{2} - |k|u )^2 - \frac{|k|^4}{4} \right)^2} \, du,
\]
for $0 \le |k| \le \kappa_{0}$. Note that $I_{n,m}$ is radial in $k$. We claim that for $l \ge 0$, $n=0, \,1$, and $0 \le |k| \le \kappa_{0}$, there hold
\begin{equation}\label{eq:In}
	c_{0} \le I_{n, 2l}(k) \le C_{0}, \qquad c_{0}|k| \le  I_{n, 2l+1} (k) \le C_{0}|k|.
\end{equation}
for some positive constants $c_{0}$ and $C_{0}$, depending on $l$.

Recall that $\tau_{\ast} (k) \ge 2|k| \Upsilon + |k|^2$ so that $\frac{\tau_{\ast}(k)}{2} -|k|u > \frac{|k|^2}{2}$ for $|u| <\Upsilon$. Therefore, if $m=2l$, then $u^{2l} \varphi(u) \ge 0$ so that $I_{n, 2l}(k) \ge 0$. Recalling from \eqref{tau-a0}, we obtain the claim \eqref{eq:In} for $I_{n, 2l}(k) $. On the other hand, if $m= 2l+1$ and $n=0$, then we use the fact that $\varphi$ is even to obtain
\begin{align*}
	I_{0, 2l+1}(|k|) 
	&= \int_{0}^{\Upsilon}  
	\left[
	\frac{  u^{2l+1} \varphi(u)}{ \left( (\frac{\tau_* (k) }{2} - |k|u )^2 - \frac{|k|^4}{4} \right)^2} 
	- \frac{ u^{2l+1} \varphi(u)}{\left( (\frac{\tau_* (k) }{2} + |k|u )^2 - \frac{|k|^4}{4} \right)^2} 
	\right]  \, du \\[3pt]
	&= |k| \int_{0}^{\Upsilon} 
	\left[
	\frac{\tau_{*}(k)(\tau_{*}(k)^2 +4|k|^2 u^2 -|k|^4) u^{2l+2} \varphi(u)}{\left( (\frac{\tau_* (k) }{2} - |k|u )^2 - \frac{|k|^4}{4} \right)^2 \left( (\frac{\tau_* (k) }{2} + |k|u )^2 - \frac{|k|^4}{4} \right)^2}
	\right] \, du
\end{align*}
Recall that for $0 \le u < \Upsilon$, we get $\tau_{*} > 2|k|u + |k|^2$ so that the integral is strictly positive. The continuity of $I_{0, 2l+1}$ on $\{ 0 \le |k| \le \kappa_{0} \}$ again implies the claim. Similarly, when $m=2l+1$ and $n=1$, we compute that
\begin{align*}
	&I_{1, 2l+1}(|k|) 
	= \int_{0}^{\Upsilon}  
		\left[
		\frac{ (\frac{\tau_* (k)}{2} -|k|u) u^{2l+1} \varphi(u)}{ \left( (\frac{\tau_* (k) }{2} - |k|u )^2 - \frac{|k|^4}{4} \right)^2} 
		- \frac{(\frac{\tau_* (k)}{2} +|k|u) u^{2l+1} \varphi(u)}{\left( (\frac{\tau_* (k) }{2} + |k|u )^2 - \frac{|k|^4}{4} \right)^2} 
		\right]  \, du \\[3pt]
	&= \frac{|k|}{8} \int_{0}^{\Upsilon} 
		\left[
			\frac{(-8 \tau_{*}(k)^2 |k|^2 u^2 - 2 \tau_{*}(k)^2 |k|^4 + 3 \tau_{*}(k)^4 - 16 |k|^4 u^4 + 8|k|^6 u^2 - |k|^8 )u^{2l+2} \varphi(u)}{\left( (\frac{\tau_* (k) }{2} - |k|u )^2 - \frac{|k|^4}{4} \right)^2 \left( (\frac{\tau_* (k) }{2} + |k|u )^2 - \frac{|k|^4}{4} \right)^2}
		\right] \, du.
\end{align*}
Using $\tau_{*}(k)^2 > 4 |k|^2 u^2 + |k|^4$ for $0 \le u \le \Upsilon$, it follows that
\begin{multline*}
	-8 \tau_{*}(k)^2 |k|^2 u^2 - 2 \tau_{*}(k)^2 |k|^4 + 3 \tau_{*}(k)^4 - 16 |k|^4 u^4 + 8|k|^6 u^2 - |k|^8\\
	= 2 \tau_{*}(k)^2 \left(\tau_{*}(k)^2 -  4 |k|^2 u^2 -|k|^4 \right) + \tau_{*}(k)^4 - (4|k|^2 u^2 + |k|^4 )^2 > 0
\end{multline*}
which concludes the proof of \eqref{eq:In}.

Next, from \eqref{eq:taustarprime}, we have 
\[
	\tau_{*}'(k) = \frac{|k|^3 I_{0, 0} + 2 I_{1,1}}{I_{1,0}}
\]
which together with \eqref{eq:In} gives
$ 	c_{0} |k| \le \tau_{*}'(k) \le C_{0} |k|
$ for some positive constants $c_{0}$ and $C_{0}$ independent of $k$. 
Next, we prove that $\tau_{*} ''(k) \ge c_{0}$ for $0 \le |k| \le \kappa_{0}$, where $c_{0}$ is some positive constant. Note that
\begin{multline*}
	\frac{\partial^2}{\partial |k|^2} \frac{1}{(  \frac{\tau_* }{2} - |k|u )^2 - \frac{|k|^4}{4}}
	= \frac{1}{((  \frac{\tau_* }{2} - |k|u )^2 - \frac{|k|^4}{4})^3} 
	\bigg[ 2\left( (\frac{\tau_*}{2} - |k|u) (\tau_{*}' - 2u) - |k|^3 \right)^2 \\
	- \left( (\frac{\tau_{*}}{2} - |k|u)^2 - \frac{|k|^4}{4}  \right) 
	\left( \frac{1}{2} (\tau_{*}' - 2u)^2 + (\frac{\tau_*}{2} - |k|u) \tau_{*}'' - 3|k|^2 \right)
	\bigg].
\end{multline*}
Differentiating \eqref{eq:taustar} with respect to $|k|$ twice,
\begin{multline}\label{eq:taustardoubleprime}
	I_{1,0}(k) \, \tau_{*}'' 
	= \int_{\{|u| < \Upsilon \}} 
	\frac{\varphi(u)}{\left( (\frac{\tau_* }{2} - |k|u )^2 - \frac{|k|^4}{4} \right)^3} \bigg[
	2\left( (\frac{\tau_*}{2} - |k|u) (\tau_{*}' - 2u) - |k|^3 \right)^2 \\
	- \left( (\frac{\tau_{*}}{2} - |k|u)^2 - \frac{|k|^4}{4}  \right) 
	\left( \frac{1}{2} (\tau_{*}' - 2u)^2 - 3|k|^2 \right) 
	\bigg] \, du .
\end{multline}
The right hand side of \eqref{eq:taustardoubleprime} is continuous on $\{ 0 \le |k| \le \kappa_{0}\}$, so it suffices to show that the right hand side is strictly positive on $\{ 0 \le |k| \le \kappa_{0} \}$. Let 
\[
A = \frac{\tau_{*}}{2} - |k|u, \qquad B = \tau_{*}' - 2u.
\]
Then the numerator of the integrand in the right hand side of \eqref{eq:taustardoubleprime} is
\[
\varphi(u) \left[ \frac{5}{4}|k|^6 + \frac{3}{2} A^2 B^2 - 4AB |k|^3 + 3|k|^2 A^2 + \frac{1}{8}|k|^4 B^2 \right].
\]
Using $A=\tau_{*}/2 - |k| u > |k|^2 /2$, we compute that
\begin{align*}
	\frac{5}{4}|k|^6 + \frac{3}{2} A^2 B^2 - 4AB |k|^3 + 3|k|^2 A^2 + \frac{1}{8}|k|^4 B^2
	&\ge \frac{3}{2} |k|^6 + \frac{3}{2} A^2 B^2 - 4 AB|k|^3  + 2|k|^2 A^2 + \frac{1}{8} |k|^4 B^2 \\
	&= \frac{3}{2}( AB - |k|^3)^2 + \frac{1}{8}|k|^2 ( |k|B -4A)^2 \ge 0.
\end{align*}
The term in the last line is zero if and only if
\[
AB= |k|^3 \quad \textrm{and} \quad |k| B = 4A,	
\]
which implies $A=\tau_{*}/2 - |k|u = |k|^2 /2$. This contradicts to $A > |k|^2 /2$, so the integral on the right hand of \eqref{eq:taustardoubleprime} is strictly positive for $0 \le |k| \le \kappa_0$. The lower bound on $\tau_{*}''$ thus follows. 
This completes the proof of Theorem \ref{theo-LangmuirE}.
\end{proof}

\begin{remark}\label{rem-genw}
The existence of plasmons (i.e. oscillatory modes) in Theorem \ref{theo-LangmuirE} holds for more general long-range repulsive interaction potentials. Indeed, the threshold $\kappa_0$ exists through the relation  
$$
\frac{\hat{w}(\kappa_0)}{2} \int_{|u|<\Upsilon} \frac{\varphi(u)}{(\Upsilon-u)(\Upsilon+\kappa_0-u)} \, du =1
$$
as long as $\hw(k)\ge 0$, $\hat{w}(0) = \infty$, $\lim_{|k|\to \infty}\hat{w}(k) <\infty$, and $\hw'(k) \le 0$. Next, the dispersion relation function $D(i \tau, k)$ remains a bijection from $[2|k|\Upsilon + |k|^2, \infty]$ onto $[\Phi(|k|) , 1]$, since $\hw(k)\ge 0$. This gives the existence of oscillatory modes, since $\Phi(|k|)<0$ for $0\le |k|<\kappa_0$. 
\end{remark}

\subsection{Landau's damping}

Theorem \ref{theo-LangmuirE} yields that the dispersion relation $D(\lambda, k) = 0$ has pure imaginary solutions $\lambda_\pm(k) = \pm i \tau_*(k)$ for $|k|\le \kappa_0$, while such solutions do not exist for $|k|>\kappa_0$. Note that, by definition, $\kappa_0 =0$ when $\Upsilon =\infty$. In this section, we study how these curves $\lambda_\pm(k) $ leave the imaginary axis as $|k|>\kappa_0$, and in particular compute the Landau's damping rate. In view of Proposition \ref{prop-nogrowth} and Theorem \ref{theo-LangmuirE}, we have $\Re \lambda_\pm(k) <0$ for $|k|>\kappa_0$.

Precisely, we obtain the following.

\begin{theorem}\label{theo-Landau} Let $\mu(|p|^2)$ be the equilibrium as described in Section \ref{sec-mu}, $\Upsilon$ be as in \eqref{def-Upsilon}, and $\kappa_0$ be defined through \eqref{def-k0-rel}. Then, for any $0<|k| - \kappa_0\ll1$, there are exactly two zeros $\lambda_\pm(k)$ of the dispersion relation $ D(\lambda,k) = 0$ that are sufficiently regular in $|k|$, with $\lambda_\pm(\kappa_0) = \pm i \tau_*(\kappa_0)$, and $\tau_*(k) = \mp \Im \lambda_\pm(k)$ satisfy \eqref{lowerbound-taustarDE} for $\kappa_{0} \le |k| \le \kappa_{0} +\delta_{0}$. In addition, the followings hold:

\begin{itemize}

\item If $\Upsilon = \infty$, then for $|k|\ll 1$, we have 
\begin{equation}\label{Landau}
\Re \lambda_\pm(k) = {\pi} [u^2 \varphi'(u)]_{u = \nu_* (k)} ( 1 + \mathcal{O}(|k|))
\end{equation}
where $\nu_*(k) = \frac{\tau_0 }{2|k|}$ is the phase velocity, with $\tau_0 = \tau_{*}(0)$.

\item If $\Upsilon <\infty$, then for $0< |k| - \kappa_0\ll1$, we have  
\begin{equation}\label{Landau-cpt}
\Re \lambda_\pm(k) = {- \frac{\pi}{2 \widetilde{\kappa}_{1}(\kappa_{0}) |k|}} \left[ \varphi(u - \frac{|k|}{2}) - \varphi(u + \frac{|k|}{2}) \right]_{u = \nu_*(k)} (1 + \cO(|k| - \kappa_0)),
\end{equation}
where $\nu_*(k) =  \Upsilon + \frac{|k|}{2} - \frac{2\kappa_{0} - \widetilde{\kappa}_{0}' (\kappa_{0})}{\widetilde{\kappa}_{1}(\kappa_{0})}(|k| - \kappa_{0})$ is the phase velocity, where $\widetilde{\kappa_{j}}(k)$ are defined in \eqref{def-tildekappaj}. 
Since $2\kappa_{0} - \widetilde{\kappa}_{0}' (\kappa_{0}) >0$, the phase velocity $\nu_*(k) < \Upsilon + \frac{|k|}{2}$ and so $\Re \lambda_\pm(k) $ is indeed strictly negative. 

\end{itemize}

\end{theorem}

\begin{proof}
Recall from \eqref{def-Dlambda1} the dispersion relation 
\begin{equation}\label{reexp-DR}
	\frac{2|k|}{\hat{w}(k)} =  \cH \left( -\frac{i\lambda +|k|^2}{2|k|} \right) - \cH \left(- \frac{i \lambda -|k|^2}{2|k|} \right)
\end{equation}
where $\cH(z)$ is defined as in \eqref{def-cH}, which is analytic in $\Im z < 0$. 
In view of Proposition \ref{prop-nogrowth} and Theorem \ref{theo-LangmuirE}, the dispersion relation \eqref{reexp-DR} has no solutions satisfying $\Re \lambda \ge 0$ for $0< |k| -\kappa_{0} \ll 1$. 

\subsubsection*{Case 1: $\Upsilon=\infty$}

When $\Upsilon = \infty$, we have $\kappa_0 =0$, and so we consider the case when $|k|\ll 1$. Using the geometric series of $\frac{1}{1-y}$, we get
\[
\frac{1}{z-u \pm \frac{|k|}{2}} = \frac{1}{z}\frac{1}{1- \frac{u \mp \frac{|k|}{2}}{z}} 
= \frac{1}{z}\sum_{j=0}^{2m}  \frac{(u \mp \frac{|k|}{2})^j}{z^j}  + \frac{(u \mp \frac{|k|}{2})^{2m+1}}{z^{2m+1}(z-u \pm \frac{|k|}{2})}
\]
for any $m \ge 0$. Then
\begin{align*}
	&\frac{1}{z-u - \frac{|k|}{2}} - \frac{1}{z-u +\frac{|k|}{2}} \\
	&= \sum_{j=0}^{2m} \frac{1}{z^{j+1}} \left[\left(u +\frac{|k|}{2}\right)^j - \left(u - \frac{|k|}{2}
	\right)^j \right]
		+ \frac{1}{z^{2m+1}} \left[ \frac{(u+\frac{|k|}{2})^{2m+1}}{z-u-\frac{|k|}{2}} - \frac{(u-\frac{|k|}{2})^{2m+1}}{z-u+\frac{|k|}{2}}\right] \\
	&= \sum_{j=0}^{2m} \frac{1}{z^{j+1}} \sum_{l=0}^{j} (1-(-1)^{j-l}) \binom{j}{l} u^{l} \frac{|k|^{j-l}}{2^{j-l}}
	+ \frac{1}{z^{2m+1}} \left[ \frac{(u+\frac{|k|}{2})^{2m+1}}{z-u-\frac{|k|}{2}} - \frac{(u-\frac{|k|}{2})^{2m+1}}{z-u+\frac{|k|}{2}}\right].
\end{align*}
Using the fact that $\varphi(u)$ is even in $u$, we get 
\begin{equation}
\begin{aligned}\label{exp-cH}
\int_{\RR} \left[  \frac{\varphi(u)}{z- u - \frac{|k|}{2}} - \frac{\varphi(u)}{z - u + \frac{|k|}{2}} \right] \, du
&= \sum_{j=0}^{m-1} \frac{1}{z^{2j+2}} \sum_{l=0}^{j} 2 \binom{2j+1}{2l} \frac{|k|^{2j-2l+1}}{2^{2j-2l+1}} \int_{\mathbb{R}} u^{2l} \varphi(u) \, du \\
& \qquad	+ \frac{1}{z^{2m+1}} \int_{\mathbb{R}} \left[ \frac{(u+\frac{|k|}{2})^{2m+1}}{z-u-\frac{|k|}{2}} - \frac{(u-\frac{|k|}{2})^{2m+1}}{z-u+\frac{|k|}{2}}\right] \varphi(u)\, du
\end{aligned}
\end{equation}
for $0\le m < \frac{N_1}{2}$, where $N_1$ is defined as in \eqref{decay-varphi}. 
For $z = -i \lambda/ (2|k|)$, the dispersion relation \eqref{reexp-DR} becomes
\begin{equation}
\begin{aligned}\label{disp1}
\lambda^2 
=  \sum_{j=0}^{m-1} (-1)^{j+1} \frac{|k|^{2j}}{\lambda^{2j}} \widetilde{\tau}_{j} (k) + (-1)^m \frac{|k|^{2m-2}}{\lambda^{2m-1}}\mathcal{R}_{1}(\lambda, k) 
\end{aligned}
\end{equation}
where
\begin{equation}
	\widetilde{\tau}_{j}(k):=  \sum_{l=0}^{j} 2^{2l+1} |k|^{2j-2l} \binom{2j+1}{2l} \int_{\RR} u^{2l} \varphi(u)\, du ,
\end{equation}
and
\begin{equation}\label{eq:remainder0}
 \mathcal{R}_{1}(\lambda, k)
 :={2^{2m}} i\int_{\RR} \left[ \frac{(u+\frac{|k|}{2})^{2m+1}}{-\frac{i\lambda}{2|k|}-u-\frac{|k|}{2}} - \frac{(u-\frac{|k|}{2})^{2m+1}}{-\frac{i\lambda}{2|k|}-u+\frac{|k|}{2}}\right] \varphi(u)\, du.
\end{equation}
for $\Re \lambda >0$. Note that $\widetilde{\tau}_{j}(k)$ is positive for all $j$ and $k$. Recall that the solutions do not exist for $\Re \lambda >0$, see Proposition \ref{prop-nogrowth} and Theorem \ref{theo-LangmuirE}. We use the classical Whitney's extension theorem \cite{Whi} to extend $\mathcal{R}_{1}(\lambda, k)$ to small neighborhoods of $\lambda = \pm i \tau_{0}$, with $\tau_0 = \tau_{*}(0)$, as a sufficiently smooth function in $\lambda$, viewing $\lambda$ as an element in $\mathbb{R}^2$. 
Following \eqref{bd-cH1}, we get
$|\mathcal{R}_{1}(\lambda, k)| \lesssim |k|$. The implicit function theorem implies the existence of the zeros $\lambda_\pm(k)$ of dispersion relation \eqref{disp1} for $|k|\ll 1$, noting that for $\lambda = i\tau$, we have 
\[
\sum_{j=0}^{{m-1}} (-1)^{j+1} \frac{|k|^{2j}}{\lambda^{2j}} \widetilde{\tau}_{j} (k) = -\tau_{0}^2 + \mathcal{O}(|k|^2)
\]
This yields $\lambda_\pm(k) = \pm i\tau_0 + \cO(|k|^2)$ for $|k|\ll1$. By induction, 
we obtain 
\begin{equation}\label{exp-lambdapm} 
\lambda_\pm(k) = \pm i \sum_{j=0}^{m-1} a_{j} |k|^{2j} + \mathcal{O}(|k|^{2m-1})
\end{equation}
for $|k|\ll 1$, where $a_{j}$ are some nonnegative coefficients. These can be computed using $\widetilde{\tau}_{j}(k)$, for instance, $a_{0} = \tau_{0}$, {$a_{1} = \frac{12}{ \tau_{0}^3} \int u^2 \varphi(u) \, du$.} In particular, taking {$m = \lfloor \frac{N_1-1}{2} \rfloor$}, we have 
\begin{equation}\label{upper-Relambda}
|\Re \lambda_\pm(k)| \lesssim {|k|^{2\lfloor \frac{N_{1}-1}{2} \rfloor - 1}}
\end{equation}
for $|k|\ll 1$. Let us study further the real part of the dispersion relation $\lambda_\pm(k)$. Taking $m=1$ in \eqref{disp1}, we obtain that the curves $\lambda_\pm(k)$ solve
\[
	\lambda^2 = - \tau_0^2 - \frac{1}{\lambda} \mathcal{R}_{1}(\lambda, k)
\]
where
\[
	\mathcal{R}_{1}(\lambda,k) =  {4i} \int_{\RR} \left[ \frac{(u+\frac{|k|}{2})^{3}}{-\frac{i\lambda}{2|k|}-u-\frac{|k|}{2}} - \frac{(u-\frac{|k|}{2})^{3}}{-\frac{i\lambda}{2|k|}-u+\frac{|k|}{2}}\right] \varphi(u)\, du.
\]
for $\Re \lambda >0$.
Here we have omitted $\pm$ for simplicity. 
Write $\lambda = \gamma + i\tau$. Note that $\gamma < 0$. Taking the imaginary part yields
\[
2 \gamma \tau 
=  -  \frac{\gamma}{|\lambda|^2} \Im \mathcal{R}_{1}(\lambda, k) 
 +\frac{\tau}{|\lambda|^2} \Re \mathcal{R}_{1}(\lambda, k) 
\]
which implies 
\[
\left[ 2\tau + \frac{1}{|\lambda|^2} \Im \mathcal{R}_{1}(\lambda, k) \right] \gamma
  = \frac{\tau}{|\lambda|^2} \Re \mathcal{R}_{1}(\lambda, k)
\]
Note that we have $|\mathcal{R}_{1}(\lambda, k)| \lesssim 1$, $\tau = \pm (\tau_0 + \cO(|k|^2))$ and {$\gamma = \cO(|k|^{2\lfloor \frac{N_{1}-1}{2} \rfloor - 1})$.} Then we get
\begin{equation}
\begin{aligned}\label{est-gamma0}
	 \left( 2\tau_{0} + \mathcal{O}(|k|) \right) \gamma 
	 =   \left(\frac{1}{\tau_{0}} + \mathcal{O}(|k|) \right) \Re \mathcal{R}_{1}(\lambda, k).
\end{aligned}
\end{equation}
We use Plemelj's analysis to compute $\Re \mathcal{R}_{1} (\lambda, k)$. Using the continuity of $\mathcal{R}_{1}(\lambda, k)$, we may evaluate it using the representation in $\{ \Re \lambda > 0 \}$. For $\Re \lambda >0$, we have
\[
	\Re \mathcal{R}_{1}(\lambda, k) = - {4} \Im \int_{\RR} \left[ \frac{(u+\frac{|k|}{2})^{3}}{-\frac{i\lambda}{2|k|}-u-\frac{|k|}{2}} - \frac{(u-\frac{|k|}{2})^{3}}{-\frac{i\lambda}{2|k|}-u+\frac{|k|}{2}}\right] \varphi(u)\, du.
\]
Write $\widetilde{\gamma} = \gamma/(2|k|)$ and $\widetilde{\tau} = \tau/(2|k|)$. Since {$\gamma = \mathcal{O}(|k|^{2\lfloor \frac{N_{1}-1}{2} \rfloor - 1})$}, we obtain that $\widetilde{\gamma}$ approaches to $0^+$ as $|k|\rightarrow 0$. Therefore, we can compute $\Re \mathcal{R}_{1}(\lambda, k)$ following the Plemelj's analysis.
We first write 
\[
\begin{aligned}
	\Im \int_{\RR} \frac{(u \pm \frac{|k|}{2})^{3} \varphi(u)}{- i \tilde{\gamma} + \tilde{\tau} -u \mp \frac{|k|}{2}}   \, du 
	= \widetilde{\gamma} \int_{\RR} \frac{(u \pm \frac{|k|}{2})^{3} \varphi(u)}{\tilde{\gamma}^2 + ( \tilde{\tau} \mp \frac{|k|}{2} - u)^2} \, du 
	= \widetilde{\gamma}  \int_{\RR} \frac{ u^3 \varphi(u \mp \frac{|k|}{2})}{\tilde{\gamma}^2  + ( \tilde{\tau}  - u)^2} \, du.
\end{aligned}
\]
Using $\int_{\RR} a(a^2+x^2)^{-1} dx =\pi$ for $a>0$, we write
\[
	\widetilde{\gamma}  \int_{\RR} \frac{ u^3 \varphi(u \mp \frac{|k|}{2})}{\widetilde{\gamma}^2  + ( \widetilde{\tau}  - u)^2} \, du
	 =  \pi \widetilde{\tau}^3 \varphi \left( \widetilde{\tau} \mp \frac{|k|}{2}  \right) + \int_{\RR} \frac{\widetilde{\gamma}}{\widetilde{\gamma}^2  + ( u - \widetilde{\tau})^2} \left[ u^3 \varphi\left( u \mp \frac{|k|}{2} \right) - \widetilde{\tau}^3 \varphi \left( \widetilde{\tau} \mp \frac{|k|}{2} \right) \right] \, du .
\]
When $|u|\le \frac{\tau_0}{4|k|}$, note that $|u - \widetilde{\tau}|\ge \frac{\tau_0}{4|k|}$ for sufficiently small $|k|$. Since $u^3 \varphi(u \mp |k|/2)$ is uniformly bounded for $u$ and $k$, 
\[
	\widetilde{\gamma} \left| \int_{ \{ |u| \le \frac{\tau_{0}}{4|k|} \}} \frac{u^3 \varphi(u \mp \frac{|k|}{2}) - \widetilde{\tau}^3 \varphi(\widetilde{\tau} \mp \frac{|k|}{2})}{\widetilde{\gamma}^2 + ( u - \widetilde{\tau})^2} \, du\right|
	\le \frac{|\gamma|}{2|k|} \frac{16|k|^2}{\tau_{0}^2} C = |\gamma| \mathcal{O}(|k|).
\]
For the integral over $|u|\ge \frac{\tau_0}{4|k|}$, we obtain that
\[
\begin{aligned}
\widetilde{\gamma}   \int_{ \{ |u| \ge \frac{\tau_{0}}{4|k|} \}} \frac{u^3 \varphi(u \mp \frac{|k|}{2}) - \widetilde{\tau}^3 \varphi(\widetilde{\tau} \mp \frac{|k|}{2})}{\widetilde{\gamma}^2 + ( u - \widetilde{\tau})^2} \, du
= |\gamma| \mathcal{O}(|k|)
\end{aligned}
\]
from the decay of $\varphi$ and its derivatives. To sum up,
\[
	\Re \mathcal{R}_{1}(\lambda, k) = -{\frac{\pi\tau^3}{2|k|^3}} \left[ \varphi \left(\frac{\tau}{2|k|} - \frac{|k|}{2} \right) -\varphi \left(\frac{\tau}{2|k|} + \frac{|k|}{2} \right) \right] + |\gamma| \mathcal{O}(|k|)
\]
and $|\gamma| \mathcal{O}(|k|)$ can be put on the left hand side of \eqref{est-gamma0}.
Therefore we get from \eqref{est-gamma0} that
\[
	\left( 2\tau_{0} + \mathcal{O}(|k|) \right) \gamma 
	= -{\frac{\pi\tau^3}{2|k|^3}} \left[ \varphi \left(\frac{\tau}{2|k|} - \frac{|k|}{2} \right) -\varphi \left(\frac{\tau}{2|k|} + \frac{|k|}{2} \right) \right] \left( \frac{1}{\tau_{0}} + \mathcal{O}(|k|) \right)
\]
For simplicity, we consider the case when $\tau >0$. Since $\tau = \tau_{0} + \mathcal{O}(|k|^2)$, we have
\begin{align*}
	\left( \tau_{0} + \mathcal{O}(|k|) \right) \gamma 
	&=  -{\frac{\pi\tau^2}{4|k|^3}}
	\left[ \varphi \left( \frac{\tau}{2|k|} - \frac{|k|}{2} \right) - \varphi \left( \frac{\tau}{2|k|} + \frac{|k|}{2} \right)  \right] \left(1 + \mathcal{O}(|k|) \right) \\
	&=  {\frac{\pi\tau^2}{4|k|^2}}\varphi' \left( \frac{\tau}{2|k|} \right) \left(1 + \mathcal{O}(|k|) \right).
\end{align*}
This proves \eqref{Landau}, noting that $\tau = \tau_0 + \cO(|k|^2)$ and $\nu_* = \tau_*/(2|k|)$. The case when $\tau <0$ follows similarly.

\subsubsection*{Case 2: $\Upsilon<\infty$}

Consider the case when $\Upsilon < \infty$. Recall that $\kappa_0>0$ and $\tau_*(\kappa_0) = 2 \kappa_{0} \Upsilon + \kappa_{0}^2$. We study the dispersion relation $D(\lambda,k)=0$ for $|k| \to \kappa_0^{+}$ and $\lambda \to \pm i \tau_*(\kappa_0)$. For simplicity, we focus on the case when $\lambda \sim + i\tau_*(\kappa_0)$. The other case can be done similarly. Observe that $z_{+}(k)= - i \lambda_{+}(k) / (2|k|) \sim \Upsilon + \kappa_{0}/2$. We use the geometric series of $\frac{1}{1-y}$ as in the previous case to write
\begin{align*}
	&\frac{1}{z-u-\frac{|k|}{2}} -\frac{1}{z-u{+}\frac{|k|}{2}}
	\\&= \frac{1}{\Upsilon -u} \frac{1}{1 - \frac{\Upsilon + \frac{|k|}{2} - z}{\Upsilon - u}} - \frac{1}{\Upsilon + |k|- u} \frac{1}{1 - \frac{\Upsilon + \frac{|k|}{2} -z}{\Upsilon +|k| - u}} \\[3pt]
	&= \frac{1}{\Upsilon - u} \sum_{j=0}^{{m}} \frac{(\Upsilon + \frac{|k|}{2} -z )^{j}}{(\Upsilon-u)^{j}} + \frac{(\Upsilon+ \frac{|k|}{2} - z)^{m+1}}{(\Upsilon-u)^{m+1} (z- \frac{|k|}{2} - u)} \\[3pt]
	& \hspace{1in}	-\frac{1}{\Upsilon +|k|- u} \sum_{j=0}^{{m}} \frac{(\Upsilon + \frac{|k|}{2} -z )^{j}}{(\Upsilon +|k| -u)^{j}} + \frac{(\Upsilon+ \frac{|k|}{2} - z)^{m+1}}{(\Upsilon+|k|-u)^{m+1} (z+ \frac{|k|}{2} - u)} \\
	&= \sum_{j=0}^{{m}} \frac{(\Upsilon + \frac{|k|}{2} - z)^{j}}{(\Upsilon-u)^{j+1} (\Upsilon +|k|- u)^{j+1}} \sum_{l=1}^{j+1} \binom{j+1}{l} (\Upsilon-u)^{j+1 - l} |k|^{l} \\[3pt]
	& \hspace{1in} + \frac{(\Upsilon+ \frac{|k|}{2} - z)^{m+1}}{(\Upsilon-u)^{m+1} (z- \frac{|k|}{2} - u)} + \frac{(\Upsilon+ \frac{|k|}{2} - z)^{m+1}}{(\Upsilon+|k|-u)^{m+1} (z+ \frac{|k|}{2} - u)}.
\end{align*}
Then we get
\begin{equation}
\begin{aligned}\label{exp-cH2}
	&\int_{\RR} \left[  \frac{\varphi(u)}{z- u - \frac{|k|}{2}} - \frac{\varphi(u)}{z - u + \frac{|k|}{2}} \right] \, du \\
	&= \sum_{j=0}^{m} (\Upsilon + \frac{|k|}{2} - z)^{j} \sum_{l=1}^{j+1} \binom{j+1}{l} |k|^{l} \int_{\RR} \frac{\varphi(u)}{(\Upsilon-u)^{l} (\Upsilon +|k| - u)^{j+1}}  \, du \\
	& \quad + ( \Upsilon+ \frac{|k|}{2} - z)^{m+1} \int_{\RR} \left[ \frac{\varphi(u)}{(\Upsilon - u)^{m+1} (z - \frac{|k|}{2} - u)} - \frac{\varphi(u)}{(\Upsilon + |k| -u)^{m+1}(z + \frac{|k|}{2} - u)} \right] \, du
\end{aligned}
\end{equation}
for $0\le m < N_1$, where $N_1$ is defined as in \eqref{decay-varphi2}. The dispersion relation \eqref{reexp-DR} becomes
\begin{equation}
\begin{aligned}\label{disp1-cpt}
	 |k|^2  
	 = \sum_{j=0}^{m} (\Upsilon + \frac{|k|}{2} - z)^{j} \widetilde{\kappa}_{j} (k) 
	   + ( \Upsilon+ \frac{|k|}{2} - z)^{m+1} \mathcal{R}_{2}(z,k)
\end{aligned}
\end{equation}
where
\begin{equation}\label{def-tildekappaj}
\widetilde{\kappa}_{j} (k) 
:= {\frac{1}{2}} \sum_{l=1}^{j+1} \binom{j+1}{l}  |k|^{l-1} \int_{\RR} \frac{\varphi(u)}{(\Upsilon-u)^{l} (\Upsilon +|k| - u)^{j+1}}  \, du 
\end{equation}
and
\[
\mathcal{R}_{2}(z, k)
:=
{\frac{1}{2|k|}}\int_{\RR} \left[ \frac{\varphi(u)}{(\Upsilon - u)^{m+1} (z - \frac{|k|}{2} - u)} - \frac{\varphi(u)}{(\Upsilon + |k| -u)^{m+1}(z + \frac{|k|}{2} - u)} \right] \, du
\]
for $\Im z <0$. Note that $\widetilde{\kappa}_{j}(k)$ is sufficiently smooth and positive for all $j$ and $k$. The dispersion relation does not have zeros for $\Im z \le 0$ and $0 < |k| -\kappa_{0} \ll 1$, so we need to extend $\mathcal{R}_{2}(z, k)$ smoothly to $\Im z > 0$, noting that $z= - i\lambda / (2|k|)$. We would like to continue $\mathcal{R}_{2}(z,k)$ as a {sufficiently smooth} function in $z$, where we view $z$ as an element in $\mathbb{R}^2$, past the slit $(-\Upsilon - \frac{|k|}{2}, \Upsilon + \frac{|k|}{2})$.
Following \eqref{bd-cH1}, we get $|\mathcal{R}_{2}(z,k)| \lesssim 1$. The implicit function theorem implies the existence of $z_{+}(k)$ satisfying the dispersion relation \eqref{disp1-cpt} for $0 < |k| - \kappa_{0} \ll 1$, noting that $\widetilde{\kappa}_{j} (k)$ are positive for all $j$.
Also note that $z_{+}(k) = - i \lambda_{+}(k)/(2|k|)$. When $|k| > \kappa_{0}$, the zeros of the dispersion relation do not exist for $\Re \lambda \ge 0$ so that $\Im z_{+}(k) >0$. Moreover, we claim that
\begin{equation}
\begin{aligned}\label{eq:ReIm}
	\left| 
		\Re z_{+}(k) - \Upsilon - \frac{|k|}{2} 
		-\frac{2\kappa_{0} - \widetilde{\kappa}_{0}' (\kappa_{0})}{\widetilde{\kappa}_{1}(\kappa_{0})}(|k| - \kappa_{0})
	\right| &\lesssim (|k| -\kappa_{0})^2 , \\
	|\Im z_{+}(k) | &\lesssim (|k| - \kappa_{0})^{N_{1} -1}.
\end{aligned}
\end{equation}
In fact, the estimate of $\Im z_{+}(k)$ can be obtained from \eqref{disp1-cpt} as in the previous case. On the other hand, one can evaluate from \eqref{disp1-cpt} that $\widetilde{z}_{+}(k) := \Upsilon + \frac{|k|}{2} - z_{+}(k)$ satisfies $\widetilde{z}_{+}(\kappa_{0}) =0 $ and $2\kappa_{0} = \widetilde{\kappa}_{0}' (\kappa_{0}) + \widetilde{z}_{+}' (\kappa_{0}) \widetilde{\kappa}_{1} (\kappa_{0})$.
This implies that $\widetilde{z}_{+}'(\kappa_{0})$ is real. Using the Taylor's theorem, we get \eqref{eq:ReIm}. Note that
	\[
		\widetilde{\kappa}_{0}'(\kappa_{0}) = - \frac{1}{2} \int \frac{\varphi(u)}{(\Upsilon - u)(\Upsilon + \kappa_{0} - u)^2} \, du <0
	\]
so that
	\[
		\frac{2\kappa_{0} - \widetilde{\kappa}_{0}' (\kappa_{0})}{\widetilde{\kappa}_{1}(\kappa_{0})} >0.
	\]
Let us study further the real part of the dispersion relation $\lambda_{+}(k) = 2i z_{+}(k) |k|$. Taking $m=1$, the dispersion relation \eqref{disp1-cpt} becomes 
\begin{equation}\label{eq:disp-cpt}
	|k|^2 
	= \widetilde{\kappa}_{0}(k) 
		+ (\Upsilon + \frac{|k|}{2} -z ) \widetilde{\kappa}_{1}(k) 
		+ (\Upsilon + \frac{|k|}{2} -z )^2 \mathcal{R}_{2}(z,k).
\end{equation}

Write $z = \tau -i \gamma$ with $\gamma <0$. Taking the imaginary part of \eqref{eq:disp-cpt}, we get
\[
	0= \gamma \widetilde{\kappa}_{1}(k)
	+ \left[ (\Upsilon + \frac{|k|}{2} - \tau)^2 - \gamma^2 \right] \Im \mathcal{R}_{2}(z,k)
	+ 2 \gamma(\Upsilon + \frac{|k|}{2} - \tau) \Re \mathcal{R}_{2} (z,k),
\]
which implies
\[
	- \left[ \widetilde{\kappa}_{1}(k) + 2 (\Upsilon + \frac{|k|}{2} - \tau) \Re \mathcal{R}_{2} (z,k) \right] \gamma
	=  \left[ (\Upsilon + \frac{|k|}{2} - \tau)^2 - \gamma^2 \right] \Im \mathcal{R}_{2}(z,k).
\]
Recall that as $|k| \rightarrow \kappa_{0}^{+}$, we have $\tau \rightarrow \Upsilon + \frac{|k|}{2}$ and $\gamma = \mathcal{O}((|k| - \kappa_{0})^{N_{1} -1})$. Hence $\gamma^2 \Im \mathcal{R}_{2}(z,k)$ can be absorbed into the left hand side. Also $\widetilde{\kappa}_{1}'(k)$ is continuous. Hence for $0 < |k| -\kappa_{0} \ll 1$,
\begin{equation}\label{est-gamma0-cpt}
	- \left[ \widetilde{\kappa}_{1}(\kappa_{0}) + \mathcal{O}((|k| - \kappa_{0})) \right] \gamma
	=  (\Upsilon + \frac{|k|}{2} - \tau)^2 \Im \mathcal{R}_{2}(z,k).
\end{equation}

Let us study the term on the right hand side. Using the continuity of $\mathcal{R}_{2}(\lambda, k)$, we may evaluate it using the representation in $\{ \Im z < 0 \}$. We have
\[
\mathcal{R}_{2}(z,k) = {\frac{1}{2|k|}} \int_{\RR} \left[ \frac{\varphi(u)}{(\Upsilon - u)^{2} (z - \frac{|k|}{2} - u)} - \frac{\varphi(u)}{(\Upsilon + |k| -u)^{2}(z + \frac{|k|}{2} - u)} \right] \, du.
\]
for $\Im z < 0$. Write $z = \tau- i \gamma$ with $\gamma >0$. Observe that
\begin{multline*}
	\Im \int_{\RR} \left[ \frac{\varphi(u)}{(\Upsilon - u)^{2} (z - \frac{|k|}{2} - u)} - \frac{\varphi(u)}{(\Upsilon + |k| -u)^{2}(z + \frac{|k|}{2} - u)} \right] \, du \\
	= \int_{\RR} \frac{\gamma}{\gamma^2 + (\tau- \frac{|k|}{2} - u)^2} \frac{\varphi(u)}{(\Upsilon-u)^2} \, du 
		-\int_{\RR} \frac{\gamma}{\gamma^2 + (\tau + \frac{|k|}{2} - u)^2} \frac{\varphi(u)}{(\Upsilon +|k| -u)^2} \, du 
\end{multline*}
As in the previous case, we apply Plemelj's analysis to compute this term. Note that $\gamma \lesssim \mathcal{O}((|k| - \kappa_{0})^{N_{1} -1})$, and $\varphi$ and its derivatives decay rapidly. Since $\int_{\RR} a(a^2+x^2)^{-1} dx =\pi$ for $a>0$, we get that $\frac{\gamma}{\gamma^2 + ( \tau \pm \frac{|k|}{2} - u)^2}$ can be approximated by the Dirac delta function at $u= \tau \pm \frac{|k|}{2}$. Therefore
\begin{multline*}
	\Im \int_{\RR} \left[ \frac{\varphi(u)}{(\Upsilon - u)^{2} (z - \frac{|k|}{2} - u)} - \frac{\varphi(u)}{(\Upsilon + |k| -u)^{2}(z + \frac{|k|}{2} - u)} \right] \, du\\
	= \frac{\pi}{(\Upsilon +\frac{|k|}{2} - \tau)^2} \left[ \varphi(\tau - \frac{|k|}{2}) - \varphi(\tau + \frac{|k|}{2}) \right] + |\gamma| \mathcal{O}(|k|-\kappa_{0}),
\end{multline*}
which yields
\[
	(\Upsilon + \frac{|k|}{2} - \tau)^2 \Im \mathcal{R}_{2}(z,k) = {\frac{\pi}{2|k|}} \left[ \varphi(\tau - \frac{|k|}{2}) - \varphi(\tau + \frac{|k|}{2}) \right] + |\gamma| \mathcal{O}(|k|-\kappa_{0})
\]
Putting this into \eqref{est-gamma0-cpt}, we obtain that
\[
	- \left[ \widetilde{\kappa}_{1}(\kappa_{0}) + \mathcal{O}((|k| - \kappa_{0})) \right] \gamma
	=  {\frac{\pi}{2|k|}}\left[ \varphi(\tau - \frac{|k|}{2}) - \varphi(\tau + \frac{|k|}{2}) \right] + \mathcal{O}(|k| - \kappa_{0})
\]
Using \eqref{eq:ReIm}, we get \eqref{Landau-cpt}, noting that $\lambda_{+}(k) = 2i z_{+}(k) |k|$ and $\nu_* = \tau$. This completes the proof of Theorem \ref{theo-Landau}. 
\end{proof}

\subsection{Proof of Theorem \ref{theo-mainLandau} }

Finally, collecting the results established in the previous sections, Theorem \ref{theo-LangmuirE} and Theorem \ref{theo-Landau}, gives Theorem  \ref{theo-mainLandau}, while the lower bounds \eqref{Penrose-intro} follow from a combination of Propositions \ref{prop-nogrowth} and \ref{prop:nonexistence}.   

\section{Green function}
\subsection{Green function in Fourier space}
In this section, we bound the Green function in Fourier space. Recall from \eqref{resolvent} that
\[
	\Trho(\lambda,k) = \widetilde{G}(\lambda,k) \Trho^0(\lambda,k)
\]
where
\[
	\widetilde{G}(\lambda,k) := \frac{1}{D(\lambda,k)}.
\]
We define the Green function in the Fourier space via the inverse Laplace transform
\[
	\widehat{G}_{k}(t) = \frac{1}{2\pi i} \int_{\{ \Re \lambda = \gamma_{0} \}} e^{\lambda t } \widetilde{G}(\lambda,k) \, d \lambda
\]
where $\gamma_{0}$ is a sufficiently large positive constant so that it is well defined. We then get
\[
	\widehat{\rho}_{k}(t) = \int_{0}^{t} \widehat{G}_{k}(t-s) \widehat{\rho}_{k}^{0}(s) \, ds.
\]
In this section, we estimate $\widehat{G}_{k}(t)$. From \eqref{eq:dlambdak}, we obtain for $\Re \lambda >0$ that
\begin{equation}\label{eq:displap}
\begin{aligned}
	D(\lambda, k) 
	= 1 + 2 \widehat{w}(k) \int_{0}^{\infty} e^{-\lambda t} \sin (t|k|^2)\widehat{\varphi}(2t|k|) \, dt 
	=: 1 + \widehat{w}(k) m_{f}(\lambda, k).
\end{aligned}
\end{equation}
   From the decay of $\varphi$, the dispersion relation $D(\lambda, k)$ is continuous up to the imaginary axis, that is, the limit of  $D(\gamma + i \tau, k)$ as $\gamma \rightarrow 0^{+}$ is well-defined for all $\tau \in \RR$. 
Also we can write
\[
	\widetilde{G}(\lambda,k) = 1 - \frac{\widehat{w}(k) m_{f}(\lambda, k)}{1 + \widehat{w}(k) m_{f}(\lambda, k)}.
\]

We first obtain the following. 

\begin{proposition}\label{prop:stable} 
Let $\kappa_0$ be defined as in \eqref{def-k0-rel}, and let $\delta_0$ be any fixed small constant. For $|k| \ge \kappa_{0} + \delta_0$, we can write
	\[
		\widehat{G}_{k}(t) = \delta(t) + \widehat{G}_{k}^{r}(t),
	\]
	where $\delta(t)$ is the Dirac delta function, and 
	\begin{equation}\label{eq:Gr-highfreq}
		|\widehat{G}_{k}^{r}(t)| \le  C_{N_{0}}|k|^{-1} \langle kt \rangle^{-N_{0}+3}
	\end{equation}
	for all $t\ge 0$ and for some positive constant $C_{N_{0}}$ depending only on $N_{0}$.
\end{proposition}

\begin{proof}
	Let
	\[
			\widetilde{G}^r_k(\lambda) = \frac{\widehat{w}(k) m_{f}(\lambda, k)}{1 + \widehat{w}(k) m_{f}(\lambda, k)}.
	\]
	We first estimate $m_{f}(\lambda, k)$. For $\Re \lambda \ge 0$, we have
	\begin{align}
		\label{eq:mfbdd}
		\begin{split}
			|m_{f}(\lambda, k)|
			&\le C_{0} \int_{0}^{\infty} \langle kt \rangle^{-N_{0}} \, dt \le C_{0} |k|^{-1}
		\end{split}
	\end{align}
	for some constant $C_{0}$ independent of $\lambda$ and $k$. On the other hand, integrating by parts yields that for $\lambda = i \tau$ with $\tau \neq |k|$,
	\begin{align*}
		(|k|^2 + \tau^2) m_{f}(i\tau, k)
		&= -2\int_{0}^{\infty}  (|k|^2 - \partial_{t}^2)(e^{-i \tau t}) \sin (t|k|^2) \widehat{\varphi}(2t|k|) dt \\
		&= -2  \int_{0}^{\infty} e^{-i \tau t} (|k|^2 - \partial_{t}^2) \left( \sin (t|k|^2) \widehat{\varphi}(2t|k|) \right) dt
		+ 2 |k|^2 \widehat{\varphi}(0).
	\end{align*}
	Here the boundary terms at infinity vanish since for all $j \ge 0$,
	\[
		\lim_{t \rightarrow \infty} \left| e^{-i \tau t } \partial_{t}^{j} \left( \sin (t|k|^2) \widehat{\varphi}(2t|k|) \right) \right| \lesssim \lim_{t \rightarrow \infty} \langle kt \rangle^{-N_{0}} =0.
	\]
	We also used
	\[
		\partial_{t}^{j} \left( \sin(t|k|^2) \widehat{\varphi}(2t|k|) \right)(0)
		=\begin{cases}
			0 & \textrm{if } j=0 \\
			|k|^2 \widehat{\varphi}(0) & \textrm{if } j=1.
		\end{cases}
	\]
	Hence for $\Re \lambda =0$,
	\begin{align}
		\label{eq:mf-verticalline}
		\begin{split}
			|m_{f}(\lambda, k)| 
			\le C_{0} \int_{0}^{\infty} \frac{t|k|^4 + |k|^3 }{  |k|^2 + |\Im \lambda|^2}  \langle kt \rangle^{-N_{0}} dt 
			+ \frac{C_{0} |k|^2 \widehat{\varphi}(0)}{ |k|^2 + |\Im \lambda|^2 } 
			\le \frac{C_{0} |k|^2 }{|k|^2 + |\Im \lambda|^2 }
		\end{split}
	\end{align}
	where $C_{0}$ is independent of $\lambda$ and $k$. One can repeat the same argument using $\partial^{n}_{\lambda} e^{-\lambda t} = (-t)^{n} e^{-\lambda t}$ to obtain the analogues of \eqref{eq:mfbdd} and \eqref{eq:mf-verticalline} for $\partial_{\lambda}^{n} m_{f}(\lambda, k)$ for $0 \le n< N_{0} -2$. Namely, 
	\[
		|\partial_{\lambda}^{n} m_{f}(\lambda, k)| \le C_{0} \int_{0}^{\infty} t^n \langle kt \rangle^{-N_{0}} \, dt \le C_{0} |k|^{-n -1}
	\]
	for $\Re \lambda \ge 0$ and
	\begin{align}
		\label{eq:mfderiv-verticalline}
		\begin{split}
			|\partial_{\lambda}^{n} m_{f}(\lambda, k)| 
			\le C_{0}\int_{0}^{\infty} \frac{t|k|^4 + |k|^3 }{  |k|^2 + |\Im \lambda|^2}  t^{n} \langle kt \rangle^{-N_{0}} dt 
			\le \frac{C_{0} |k|^{-n+2}}{|k|^2 + |\Im \lambda|^2 }
		\end{split}
	\end{align}
	for $\Re \lambda =0$ where $C_{0}$ is a constant independent of $\lambda$, $k$ and $N$.
	
	In view of \eqref{eq:mf-verticalline}, there exists a large constant $M$ such that
		\[
			|D(\lambda, k)| \ge \frac{1}{2}
		\]
	for all $\lambda$ satisfying $\Re \lambda = 0$ with $|\Im \lambda| \ge M$ or $|k| \ge M$. When $|\Im \lambda | \le M$ and $1 \lesssim |k| - \kappa_{0} \le M - \kappa_{0}$, we use the fact that for $|k| > \kappa_{0}$, the dispersion relation $D(\lambda, k)$ does not vanish on $\{ \Re \lambda \ge 0 \}$, see Proposition \ref{prop-nogrowth}, Sections \ref{sec-emlambda} and \ref{sec-oscmode}. 
	Hence there is a positive constant $c_{0}$ satisfying
		\[
			|D(\lambda, k)| \ge c_{0}.
		\]
	for $\Re \lambda \ge 0$, $|\Im \lambda | \le M$ and $1 \lesssim |k| - \kappa_{0} \le M - \kappa_{0}$. 
	Using \eqref{eq:mf-verticalline}, there exists a positive constant $C_{0}$, independent of $\lambda$ and $k$, so that 
	\begin{align}
		\label{eq:Gtilde-largek}
		|\widetilde{G}_{k}^{r} (\lambda)| \le \frac{C_{0}}{|k|^2 + |\Im \lambda|^2}
	\end{align}
	for $|k| - \kappa_{0} \gtrsim 1$ and $\Re \lambda = 0$. By the same argument, using \eqref{eq:mfderiv-verticalline}, there exists a positive constant $C_{N_{0}}$ independent of $\lambda$, $k$ and $N$ satisfying
	\begin{equation}
		\label{eq:Gtildederiv-largek}
		|\partial_{\lambda}^{n} \widetilde{G}_{k}^{r} (\lambda)| \le \left| \sum_{j=0}^{n} \binom{n}{j} \partial_{\lambda}^{j} (1-D(\lambda, k)) \partial_{\lambda}^{n-j} (D(\lambda, k)^{-1}) \right| \le \frac{C_{N_{0}} |k|^{-n}}{|k|^2 + |\Im \lambda|^2}
	\end{equation}
	for $0 \le N < N_{0} -2$, $|k| - \kappa_{0} \gtrsim 1$ and $\Re \lambda =0$.
	
	Now we take the inverse Laplace transform of $\widetilde{G}_{k}^{r}(\lambda)$. Note that $\widetilde{G}_{k}^{r}(\lambda)$ is analytic on $ \Re \lambda > 0 $ and continuous up to $ \Re \lambda =0$ from the right. Therefore, writing $\lambda = i \tau$, integrating by parts yields
	\begin{align*}
		\widehat{G}_{k}^{r}(t) 
		= \frac{1}{2\pi i } \int_{\{ \Re \lambda = 0 \}} e^{\lambda t} \widetilde{G}_{k}^{r}(\lambda) d \lambda = \frac{1}{2\pi} \int_{\RR} e^{i \tau t} \widetilde{G}_{k}^{r} (i \tau) d \tau 
		=\frac{(-1)^{n}}{2\pi t^{n}} \int_{\RR}  e^{i\tau t} \partial_{\lambda}^{N} \widetilde{G}_{k}^{r}(i \tau) d \tau
	\end{align*}
	for all $0 \le n < N_{0} -2$. Choose $n= N_{0} -3$. Using \eqref{eq:Gtildederiv-largek} and $\int_{\mathbb{R}} (x^2 +|k|^2)^{-1} dx = \pi |k|^{-1}$, we get
	\[
		|\widehat{G}_{k}^{r}(t)| \le \frac{C_{N_{0}}}{|t|^{N_{0} -3}} \left|\int_{\RR} \frac{|k|^{-N_{0}+3}}{|k|^2 + \tau^2} \, d\tau \right|\le C_{N_{0}} |k|^{-1} |kt|^{-N_{0} +3}.
	\]
	for some $C_{N_{0}}$ independent of $t$ and $k$.
\end{proof}

Next, we obtain the following. 

\begin{proposition}\label{prop:GreenFourier1}
Let $\kappa_0$ be defined as in \eqref{def-k0-rel}, and let $\delta_0$ be sufficiently small. For $|k| \le \kappa_{0} + \delta_0$, the Green function can be written as
	\begin{equation}
		\widehat{G}_{k}(t) = \delta(t) + \sum_{\pm} \widehat{G}^{osc}_{k, \pm}(t) + \widehat{G}^{r}_{k}(t)
	\end{equation}
	for all $t \in \mathbb{R}$ where
	\[
		\widehat{G}^{osc}_{k, \pm}(t) = e^{\lambda_{\pm}(k)t} a_{\pm}(k).
	\]
	Here
	\begin{equation}
		a_{\pm}(k) = \pm \frac{i \tau_{0}}{2} + \mathcal{O}(|k|^2)
	\end{equation}
	is a sufficiently smooth function supported in $\{|k| \le \kappa_{0} + \delta_{0} \}$ and 
	\begin{equation}\label{eq:Gr-lowfreq}
		\left||k|^{|\alpha|} \partial_{k}^{\alpha} \widehat{G}^{r}_{k}(t) \right| \le C |k|^{3} \langle kt \rangle^{-N_{0} +3 + |\alpha|}
	\end{equation}
	for some constant $C$ and for all $0 \le |\alpha| < N_{0}-2$.
\end{proposition}

\begin{proof}
Fix $0 < |k| \le \kappa_{0} + \delta_{0}$. There are poles of $\widetilde{G}(\lambda,k)$ exactly at $\lambda_{\pm}(k)$. Since $\varphi$ is merely $C^{N_{0}}$, it is not allowed to continue $\widetilde{G}(\lambda,k)$ analytically past the imaginary axis. But $\widetilde{G}(\lambda,k)$ is analytic in $\{ \Re \lambda >0 \}$ so that we can shift the contour $\{ \Re \lambda = \gamma_{0} \}$ to the imaginary axis as
\begin{equation}\label{eq:Gkhat}
	\widehat{G}_{k}(t) = \frac{1}{2\pi i} \int_{\Gamma} e^{\lambda t} \widetilde{G}(\lambda,k) d\lambda.
\end{equation}
To obtain the decay in time, we integrate by parts many times on $\{ \Re \lambda = \gamma_{0} \}$ and then take $\gamma_{0} \rightarrow 0^{+}$ so that for $n \ge 0$,
\begin{equation}\label{eq:Gkhat2}
	\widehat{G}_{k}(t) = \frac{(-1)^{n}}{2\pi i (|k|t)^{n}} \int_{\Gamma}  e^{\lambda t} |k|^{n} \partial_{\lambda}^{n} \widetilde{G}(\lambda,k) \, d \lambda.
\end{equation}
We decompose the contour into
	\[
		\Gamma = \Gamma_{1} \cup \Gamma_{2} \cup \mathcal{C}_{\pm}
	\]
where
\begin{equation*}
	\begin{aligned}
		&\Gamma_{1} = \{ \lambda = i \tau, \quad |\tau \pm \tau_* (k) | \ge |k| , \quad |\tau| > 2|k| \Upsilon_* + |k|^2 \}, \\
		&\Gamma_{1} = \{ \lambda = i \tau, \quad |\tau \pm \tau_* (k) | \ge |k| , \quad |\tau| \le 2|k| \Upsilon_* + |k|^2 \}, \\
		&\mathcal{C}_{\pm} = \{ \Re \lambda \ge 0 , \quad |\lambda \mp i \tau_* (k) = |k| \},
	\end{aligned}
\end{equation*}
where
	\begin{equation}
		\Upsilon_{\ast}
		=
		\begin{cases}
			\Upsilon, \qquad &\textrm{if } \Upsilon < \infty \\
			M, &\textrm{if } \Upsilon = \infty
		\end{cases}
	\end{equation}
for some sufficiently large $M >0$, which may depend on $k$.

Recall from \eqref{def-Dlambda1} that the dispersion relation is
\[
	D(\lambda,k) 
	= 1 - \frac{\hw(k)}{2|k|} \left[ \cH \left( -\frac{i\lambda +|k|^2}{2|k|} \right) - \cH \left(- \frac{i \lambda -|k|^2}{2|k|} \right)  \right].
\]
Using \eqref{disp1} with $m=2$, we get
\begin{equation}
	\begin{aligned}
		D(\lambda, k) = 1 + \frac{\widetilde{\tau}_{0}(k)}{\lambda^2} - \frac{|k|^2}{\lambda^4} \widetilde{\tau}_{1}(k)  - \frac{1}{\lambda^4} \mathcal{R}_{1} (\lambda,k)
	\end{aligned}
\end{equation}
where
\begin{equation}
\widetilde{\tau}_{0}(k)=2 \int_{\RR} \varphi(u)\, du, \qquad \widetilde{\tau}_{1}(k) =2|k|^2 \int_{\RR} \varphi(u) \, du +24\int_{\RR} u^2 \varphi(u) \, du,
\end{equation}
and
\begin{equation}
\mathcal{R}_{1}(\lambda, k)
:=
8|k|\int_{\RR} \left[ \frac{(u+\frac{|k|}{2})^{4}}{-\frac{i\lambda}{2|k|}-u-\frac{|k|}{2}} - \frac{(u-\frac{|k|}{2})^{4}}{-\frac{i\lambda}{2|k|}-u+\frac{|k|}{2}}\right] \varphi(u)\, du.
\end{equation}
for $\Re \lambda >0$.
Introducing
\[
	\tau_{0}^2 : = 2\int_{\RR} \varphi(u)\, du , \qquad \tau_{1}^2 = 24 \int_{\RR} u^2 \varphi(u) \, du,
\]
we have $\widetilde{\tau}_{0}(k) = \tau_{0}^2$ and $\widetilde{\tau}_{1}(k) = |k|^2 \tau_{0}^2 + \tau_{1}^2$. Therefore the dispersion relation can be written as
\begin{equation}\label{eq:dispersion}
	D(\lambda, k) = 1 + \frac{\tau_{0}^2}{\lambda^2} - \frac{|k|^2 \tau_{1}^2}{\lambda^4} - \frac{1}{\lambda^4 } \mathcal{R}(\lambda,k)
\end{equation}
where
\begin{equation}\label{eq:remainderexpression}
	\mathcal{R}(\lambda, k)
	:=
	|k|^4 \tau_{0}^2 + 8|k| \int_{\RR} \left[ \frac{(u+\frac{|k|}{2})^{4}}{-\frac{i\lambda}{2|k|}-u-\frac{|k|}{2}} - \frac{(u-\frac{|k|}{2})^{4}}{-\frac{i\lambda}{2|k|}-u+\frac{|k|}{2}}\right] \varphi(u)\, du.
\end{equation}
Note that when $|k| \gtrsim 1$, we have $\delta_{0} \le |k| \le \kappa_{0} + \delta_{0}$ so that $|k|$ is bounded.

\subsubsection*{Estimate of $\mathcal{R}(\lambda, k)$}
We first prove that
\begin{equation}\label{eq:remainder}
	|\partial_{\lambda}^{n} \mathcal{R}(\lambda, k)| 
	\lesssim |k|^{2-n}
\end{equation}
and
\begin{equation}\label{eq:remainder2}
	|\partial_{\lambda}^{n} \mathcal{R}(\lambda, k)| 
	\lesssim \begin{cases}
		|k|^4 ( 1 +|\lambda|^{-2}), \quad &\textrm{if } n =0 \\
		|k|^4 |\lambda|^{-2-n}, &\textrm{if } 1 \le n < N_{0} - 2,
	\end{cases}
\end{equation}
where all the estimates hold for $0 \le n < N_{0} -2$ and uniformly in $\Re \lambda \ge 0$. \\[5pt]
\noindent\textbf{Case 1: $\Upsilon = \infty$} \\
In fact, following \eqref{eq:dlambdak}, we obtain that
\begin{equation}\label{eq:Ndefinition}
	\begin{aligned}
		 &\int_{\RR} 
		 	\left[ \frac{(u+\frac{|k|}{2})^{4}}{-\frac{i\lambda}{2|k|}-u-\frac{|k|}{2}} - \frac{(u-\frac{|k|}{2})^{4}}{-\frac{i\lambda}{2|k|}-u+\frac{|k|}{2}}\right] \varphi(u)\, du \\[3pt]
		 &= 2i|k| \int_{\RR}
		 	\left[
		 		\frac{(u+ \frac{|k|}{2})^4}{\lambda - 2i|k|u - i|k|^2} - \frac{(u- \frac{|k|}{2})^4}{\lambda - 2i|k|u + i|k|^2}
		 	\right] \varphi(u) \, du \\[3pt]
		 &= 2i|k| \int_{0}^{\infty} e^{-\lambda t} \int_{\RR} e^{2i|k|ut}
		 	\left[
		 		e^{i|k|^2 t} ( u+ \frac{|k|}{2})^4 - e^{-i|k|^2 t} (u - \frac{|k|}{2})^4 
		 	\right] \varphi(u) \, du \, dt \\
		 &=: 2i|k|\int_{0}^{\infty} e^{-\lambda t} N(t) \, dt.
	\end{aligned}
\end{equation}
We claim that
\begin{equation}\label{eq:Nestimate}
	|N(t)| \lesssim |k| \langle kt \rangle^{-N_{0} +1}.
\end{equation}
When $|k| \gtrsim 1$, we write
\[
	N(t) = \int_{\RR} e^{2i|k|ut} u^4 \left[\varphi(u - \frac{|k|}{2}) - \varphi(u + \frac{|k|}{2})  \right] \, du.
\]
Using the regularity of $\varphi$, we get
\[
	|N(t)| \lesssim \langle kt \rangle^{-N_{0}} \lesssim |k| \langle kt \rangle^{-N_{0} +1}.
\]
When $|k| \ll 1$, we write
\begin{equation*}
	\begin{aligned}
		N(t) &= \int_{\RR} e^{2i|k|ut}
		\left[
		e^{i|k|^2 t} ( u+ \frac{|k|}{2})^4 - e^{-i|k|^2 t} (u - \frac{|k|}{2})^4 
		\right] \varphi(u) \, du .
	\end{aligned}
\end{equation*}
Using the regularity of $\varphi$ and noting that the integrand is of order $|k|$, we get
\[
	|N(t)| \lesssim |k| \langle kt \rangle^{-N_{0}+1}.
\]
Next, we estimate $\partial_{\lambda}^{n} \mathcal{R}(\lambda, k)$ using the decay of $N(\cdot)$. Namely, integrating by parts yields
\[
	|\partial_{\lambda}^{n} \mathcal{R}(\lambda, k) | 
	\lesssim |k|^4 + |k|^2 \int_{0}^{\infty} e^{-\Re \lambda t} t^n | N(t)| \, dt 
	\lesssim |k|^{2-n}
\]
for $0 \le n < N_{0} -2$.
Note that this also proves
\[
|\partial_{\lambda}^{n} \mathcal{R}(\lambda, k)| 
\lesssim \begin{cases}
	|k|^4 ( 1 +|\lambda|^{-2}), \quad &\textrm{if } n =0 \\
	|k|^4 |\lambda|^{-2-n}, &\textrm{if } 1 \le n < N_{0} - 2.
\end{cases}
\]
for $|\lambda| \lesssim |k|$. The same estimate follows in the case when $|k| \lesssim |\lambda|$ by considering higher order expansions in \eqref{disp1}. For instance, taking $m=3$, we get
\[
	\mathcal{R}(\lambda, k)
	=
	|k|^4 \tau_{0}^2 - \frac{|k|^4}{\lambda^2} \widetilde{\tau}_{2}(k)
	{- \frac{32|k|^3}{ \lambda^2}} \int_{\RR} \left[ \frac{(u+\frac{|k|}{2})^{6}}{-\frac{i\lambda}{2|k|}-u-\frac{|k|}{2}} - \frac{(u-\frac{|k|}{2})^{6}}{-\frac{i\lambda}{2|k|}-u+\frac{|k|}{2}}\right] \varphi(u)\, du.
\]
so that
\[
	|\mathcal{R}(\lambda, k)| \lesssim |k|^4 + \frac{|k|^4}{|\lambda|^2} + \frac{|k|^5}{|\lambda|^2} \int_{0}^{\infty}  |k| \langle kt \rangle^{-N_{0}+1} \, dt \lesssim |k|^4 ( 1 + |\lambda|^{-2}).
\] 
The corresponding estimates for derivatives in $\lambda$ can be obtained by a similar argument, provided that $N_{1} \ge 2N_{0}+2$.\\[5pt]
\noindent\textbf{Case 2: $\Upsilon < \infty$}\\
The estimates of $\mathcal{R}(\lambda, k)$ can be derived in the exactly same way. However, $\varphi(u)$ is compactly supported, so we do not need to consider the decay rate of the integrand at the infinity. Hence the condition $N_{1} \ge 2N_{0} +2$ is not necessary in this case.

\subsubsection*{Decomposition of $\widetilde{G}(\lambda,k)$}
We decompose $\widetilde{G}(\lambda,k)$ as follows.
\begin{equation*}
	\begin{aligned}
		\widetilde{G}(\lambda,k)
		&= 1 + \frac{1-D(\lambda,k)}{D(\lambda,k )} = 1 + \frac{ - \lambda^2 \tau_{0}^2 + |k|^2 \tau_{1}^2 + \mathcal{R}(\lambda, k)}{\lambda^4 + \lambda^2 \tau_{0}^2 -  |k|^2 \tau_{1}^2 - \mathcal{R}(\lambda, k)} \\[3pt]
		&= 1 + \widetilde{G}_{k,0}(\lambda) + \widetilde{G}_{k,1}(\lambda)
	\end{aligned}
\end{equation*}
where
\begin{equation*}
	\begin{aligned}
		\widetilde{G}_{k,0}(\lambda)
		&:=  \frac{ - \lambda^2 \tau_{0}^2 +|k|^2 \tau_{1}^2 }{\lambda^4 + \lambda^2 \tau_{0}^2 -|k|^2 \tau_{1}^2 }, \\[3pt]
		\widetilde{G}_{k,1}(\lambda)
		&:= \frac{\lambda^4 \mathcal{R}(\lambda,k)}{(\lambda^4 + \lambda^2 \tau_{0}^2 - |k|^2 \tau_{1}^2 - \mathcal{R}(\lambda, k)) (\lambda^4 + \lambda^2 \tau_{0}^2 - |k|^2 \tau_{1}^2 ) }.
	\end{aligned}
\end{equation*}
We claim that for $0 < |k| \le \kappa_{0} + \delta_{0}$, there is some $0 < \theta_{0} < \tau_{1}$ satisfying
\begin{align}
		\label{eq:est-Gkhat}
		&|\widehat{G}_{k,0}(t) - \sum_{\pm} e^{\lambda_{\pm, 0} t} \operatorname{Res} (\widetilde{G}_{k,0}(\lambda_{\pm, 0})) | \lesssim |k|^3 e^{-\theta_{0}|kt|}, \\
		\label{eq:est-Gkhat2}
		&|\widehat{G}_{k,1}(t) - \sum_{\pm} e^{\lambda_{\pm} t} \operatorname{Res} (\widetilde{G}_{k}(\lambda_{\pm})) + \sum_{\pm} e^{\lambda_{\pm,0}t} \operatorname{Res} (\widetilde{G}_{k,0}(\lambda_{\pm, 0})) | \lesssim |k|^3 \langle kt \rangle^{-N_{0} +3}.
\end{align}
These estimates conclude the proof of the proposition for $\alpha =0$ since $\sum_{\pm} e^{\lambda_{\pm, 0}t } \operatorname{Res} (\widetilde{G}_{k,0}(\lambda_{\pm, 0}))$ is cancelled and $\sum_{\pm} e^{\lambda_{\pm} t}\operatorname{Res} (\widetilde{G}_{k}(\lambda_{\pm}))$ gives us the oscillatory term.

\subsubsection*{Estimate of $\widehat{G}_{k,0}(t)$}
Recall that
\[
	\widetilde{G}_{k,0}(\lambda)
	=  \frac{ - \lambda^2 \tau_{0}^2 + |k|^2 \tau_{1}^2 }{\lambda^4 + \lambda^2 \tau_{0}^2 - |k|^2 \tau_{1}^2 },
\]
which is meromorphic in $\mathbb{C}$. The polynomial $\lambda^4 + \lambda^2 \tau_{0}^2 - |k|^2 \tau_{1}^2$ has four distinct roots, given by
\begin{equation*}
	\begin{aligned}
		\lambda_{\pm, 0} = \pm i \left( \frac{\tau_{0}^2 + \sqrt{\tau_{0}^4 + 4 \tau_{1}^2 |k|^2 }}{2} \right)^{1/2}, \quad
		\mu_{\pm, 0} = \pm \left( \frac{ - \tau_{0}^2 + \sqrt{\tau_{0}^4 +4 \tau_{1}^2 |k|^2}}{2} \right)^{1/2}
	\end{aligned}
\end{equation*}
Note that
	\begin{equation}\label{eq:roots}
		|\lambda_{\pm, 0}(k) \mp i \tau_* (k)| \lesssim |k|^2, \qquad | \mu_{\pm, 0}(k) \mp \tau_{1}|k| | \lesssim |k|^2.
	\end{equation}
Only one root $\mu_{+, 0 }$, which is located on the positive real axis, is on the right of $\Gamma$, so the residue theorem yields
\begin{equation}\label{eq:Gk0}
	\begin{aligned}
		\widehat{G}_{k,0}(t) 
		&= \frac{1}{2\pi i} \int_{\Gamma} e^{\lambda t} \widetilde{G}_{k,0}(\lambda) \, d \lambda \\
		&= \sum_{\pm} e^{\lambda_{\pm, 0} t} \operatorname{Res} (\widetilde{G}_{k,0}(\lambda_{\pm, 0})) + e^{\mu_{-, 0 } t} \operatorname{Res} (\widetilde{G}_{k,0}(\mu_{-,0}))
		+ \frac{1}{2\pi i} \lim_{\gamma_{0} \rightarrow -\infty} \int_{\{ \Re \lambda = \gamma_{0} \}} e^{\lambda t } \widetilde{G}_{k,0}(\lambda) \, d\lambda \\
		&= \sum_{\pm} e^{\lambda_{\pm, 0} t} \operatorname{Res} (\widetilde{G}_{k,0}(\lambda_{\pm, 0})) + e^{\mu_{-, 0 } t} \operatorname{Res} (\widetilde{G}_{k,0}(\mu_{-,0})).
	\end{aligned}
\end{equation}
where the limit vanishes as $\gamma_{0} \rightarrow -\infty$ since the integrand decays exponentially. We compute that
\begin{equation}\label{eq:Gk02}
	\begin{aligned}
		e^{\mu_{-, 0 } t} \operatorname{Res} (\widetilde{G}_{k,0}(\mu_{-,0}))
		&= \frac{\mu_{-, 0}^4 e^{\mu_{-,0}t}}{(\mu_{-,0} - \mu_{+,0})(\mu_{-,0} - \lambda_{+,0})(\mu_{-,0} - \lambda_{-,0}) } \\
		&= \frac{\mu_{-, 0}^3 e^{\mu_{-,0}t}}{2(\mu_{-,0} - \lambda_{+,0})(\mu_{-,0} - \lambda_{-,0}) }.
	\end{aligned}
\end{equation}
where the last equality follows from $\mu_{+,0} = -\mu_{-,0}$. Using \eqref{eq:roots}, we have
\[
 	|e^{\mu_{-, 0 } t} \operatorname{Res} (\widetilde{G}_{k,0}(\mu_{-,0}))| \lesssim |k|^3 e^{-\theta_{0} |kt|}
\]
for some $0 < \theta_{0} < \tau_{1}$, which implies \eqref{eq:est-Gkhat}.

\subsubsection*{Estimate of $\widehat{G}_{k,1}(t)$ on $\Gamma_{1}$}
Recall that
\[
\widetilde{G}_{k,1}(\lambda)
= \frac{\lambda^4 \mathcal{R}(\lambda,k)}{(\lambda^4 + \lambda^2 \tau_{0}^2 - |k|^2 \tau_{1}^2 - \mathcal{R}(\lambda, k)) (\lambda^4 + \lambda^2 \tau_{0}^2 - |k|^2 \tau_{1}^2 ) }.
\]
On $\Gamma_1$, we have $\lambda = i \tau$ with $|\tau| > 2|k| \Upsilon_{*} + |k|^2$ and $|\tau \pm \tau_*| \ge |k|$. We first study the lower bound of the denominator.\\[5pt]
\noindent\textbf{Case 1: $\Upsilon = \infty$}\\
When $\Upsilon = \infty$, we have $\kappa_{0} =0$ so that $0 <|k| \ll 1$. Fix $|k|$ and then take $\Upsilon_* = M$ sufficiently large so that
	\[
		|\tau^4 - \tau^2 \tau_{0}^2 -|k|^2 \tau_{1}^2 | \gtrsim 1 + |\tau|^4.
	\]
for $|\tau|> 2|k|M + |k|^2$. The same lower bound also hold for $\tau^4 - \tau^2 \tau_{0}^2 -|k|^2 \tau_{1}^2 - \mathcal{R}(i \tau, k)$ since $|k| \ll 1 \lesssim |\tau|$ so that $|\mathcal{R}(i\tau, k)| \lesssim |k|^4 |\tau|^{-2} \ll |k|^2$.\\[5pt]
\noindent\textbf{Case 2: $\Upsilon < \infty$}\\
When $\Upsilon < \infty$, the lower bound for $\tau^4 - \tau^2 \tau_{0}^2 - |k|^2 \tau_{1}^2$ is as follows.
\begin{equation}\label{eq:lowerbound}
	|\tau^4 - \tau^2 \tau_{0}^2 - |k|^2 \tau_{1}^2|
	\gtrsim
	\begin{cases}
		1 + |\tau|^4,  					&\textrm{if} \quad |\tau| \ge \dfrac{3\tau_{*}}{2} \\[5pt]
		|\tau \mp \tau_* | +|k|, \qquad &\textrm{if} \quad \dfrac{\tau_*}{2} \le \pm \tau \le \dfrac{3\tau_*}{2} \\[5pt]
		|\tau|^2 +|k|^2, 				&\textrm{if} \quad 2|k|\Upsilon + |k|^2 < |\tau| \le \dfrac{\tau_*}{2}
	\end{cases}
\end{equation}
Indeed, note that
\[
	\tau^4 - \tau^2 \tau_{0}^2 - |k|^2 \tau_{1}^2
	= (\tau +i\lambda_{+,0})(\tau +i\lambda_{-,0})(\tau +i\mu_{+,0})(\tau +i\mu_{-,0})
\]
has four distinct zeros $- i\lambda_{\pm, 0} \sim \pm \tau_{0}$ and $-i \mu_{\pm, 0} \sim \mp i \tau_{1} |k|$. Also $|\tau_* - \tau_0| \lesssim |k|^2$. Therefore, if $|\tau| \ge 3\tau_* / 2$, then $|\tau^4 - \tau^2 \tau_{0}^2 -|k|^2 \tau_{1}^2| \gtrsim |\tau|^2 |\tau^2 - \tau_{*}^2| \gtrsim 1 +|\tau|^4$. If $\tau_* / 2 \le \tau \le 3\tau_* / 2$, then $|\tau^4 - \tau^2 \tau_{0}^2 -|k|^2 \tau_{1}^2| \gtrsim |\tau + i \lambda_{+, 0} | \gtrsim |\tau- \tau_*| + |k|$, and similarly for the $-$ case. If $2|k|\Upsilon + |k|^2 < |\tau| \le \tau_{*} /2$, then $|\tau^4 - \tau^2 \tau_{0}^2 -|k|^2 \tau_{1}^2| \gtrsim |(\tau + i \mu_{+,0})(\tau + i \mu_{-,0})| \gtrsim |\tau|^2 + |k|^2$. 

Next, we claim that the same lower bound \eqref{eq:lowerbound} also hold for $\tau^4 - \tau^2 \tau_{0}^2 -|k|^2 \tau_{1}^2 - \mathcal{R}(i\tau, k)$. Observe that
\begin{equation}
	\begin{aligned}
		\mathcal{R}(i\tau, k)
		&= |k|^4 \tau_{0}^2 
		+ {8|k|} \int_{\{ |u| < \Upsilon \}}
		\left[ \frac{(u+ \frac{|k|}{2})^{4}}{\frac{\tau}{2|k|} - u - \frac{|k|}{2}}
		-\frac{(u- \frac{|k|}{2})^{4}}{\frac{\tau}{2|k|} - u + \frac{|k|}{2}} \right] \varphi(u) \, du \\
		&= |k|^4 \tau_{0}^2 + {8|k|} \int_{\{ |u| < \Upsilon \}} \frac{u^4}{\frac{\tau}{2|k|} - u} \left[ \varphi(u- \frac{|k|}{2}) - \varphi(u + \frac{|k|}{2}) \right] \, du \\
		&= |k|^4 \tau_{0}^2 + {16|k|} \int_{0}^{\Upsilon} \frac{u^5}{\frac{\tau^2}{4|k|^2} - u^2} \left[ \varphi(u- \frac{|k|}{2}) - \varphi(u + \frac{|k|}{2}) \right] \, du
	\end{aligned}
\end{equation}
is strictly positive since $\varphi$ is even, strictly decreasing in $u$, and $|\tau| > 2|k| \Upsilon +|k|^2$.

When $|k| \gtrsim 1$, note that $0 < \mathcal{R}(i\tau, k) \le C |k|^2$ for some positive constant $C$. Since $|\tau \mp \tau_*| \ge |k| \gtrsim 1$, we get
\[
	|\tau^4 - \tau^2 \tau_{0}^2 - |k|^2 \tau_{1}^2 - \mathcal{R}(i\tau, k)| \ge \min \left\{ |\tau^4 - \tau^2 \tau_{0}^2 - |k|^2 \tau_{1}^2|, |\tau^4 - \tau^2 \tau_{0}^2 - |k|^2 (\tau_{1}^2 +C)| \right\}
\]
which implies \eqref{eq:lowerbound} for $\tau^4 - \tau^2 \tau_{0}^2 -|k|^2 \tau_{1}^2 - \mathcal{R}(i\tau, k)$.

When $|k| \ll 1$ and $|k| \ll |\tau|$, we have $|\mathcal{R}(\lambda, k)| \lesssim |k|^4 |\lambda|^{-2} \ll |k|^2$ from \eqref{eq:remainder2}, so we may view $\mathcal{R}(\lambda, k)$ as a remainder. The lower bound \eqref{eq:lowerbound} holds for $\tau^4 - \tau^2 \tau_{0}^2 -|k|^2 \tau_{1}^2 - \mathcal{R}(i\tau, k)$.

When $|k| \ll 1$ and $|k| \gtrsim |\tau|$, we instead estimate $\tau^4 - \tau^2 \tau_{0}^2 -|k|^2 \tau_{1}^2 - \mathcal{R}(i\tau, k)$ directly, using the sign of $\mathcal{R}(i\tau, k)$.  Since $|\tau| \lesssim |k| \ll 1$, we have $|\tau| \le \tau_{0} /2$. Hence
\[
	|\tau^4 - \tau^2 \tau_{0}^2 - |k|^2 \tau_{1}^2  - \mathcal{R}(i\tau, k)| \gtrsim \frac{3}{4} \tau^2 \tau_{0}^2 + |k|^2 \tau_{1}^2 + \mathcal{R}(i\tau, k) \gtrsim |\tau|^2 + |k|^2.
\]

\subsubsection*{Estimate of $\widehat{G}_{k,1}(t)$ on $\Gamma_{2}$}
In this case, $\mathcal{R}(i\tau, k)$ is of order $|k|^2$, so we cannot view it as a remainder term. Instead, we  can use a lower bound on $D(i \tau, k)$ since $|\tau| \le 2|k| \Upsilon_{*} + |k|^2$. Indeed, we claim that 
\begin{equation}\label{eq:lowbd2}
	|D(i \tau, k)| \gtrsim 1 + \frac{1}{|k|^2}
\end{equation}
for $|\tau|\le 2|k|\Upsilon_* + |k|^2$ and $|\tau \pm \tau_*| \ge |k|$.

When $\Upsilon = \infty$, we obtain \eqref{eq:lowbd2} by applying \eqref{eq:lowbd} since $\{|\tilde{\tau}| \le 2M + \delta_{0} \}$ is compact in $\mathbb{R}$.
When $\Upsilon < \infty$, note that $\kappa_{0} >0$ and $\Phi(\kappa_{0}) =0$, see \eqref{eq:defPhi}. 
We obtain from \eqref{eq:lowbd} that $|D(i\tilde{\tau}|k|, k) |\gtrsim 1+|k|^{-2}$ for $|\tilde{\tau}| \le 2\Upsilon$ and $|k| \le \kappa_{0} + \delta_{0}$. For $2\Upsilon \le |\tilde{\tau}| < 2\Upsilon + |k| - \varepsilon$, recall rom \eqref{form-Ditau2} that
\[
	|D(i\tilde \tau |k|,k) |
	\ge \frac{\pi\hat{w}(k)}{2|k|} 
	\Big|\varphi \left( \frac{\tilde\tau + |k|}{2} \right) - \varphi \left(\frac{\tilde\tau - |k|}{2} \right) \Big|.
\]
Since exactly one of $\frac{\tilde{\tau} \pm |k|}{2}$ is contained in the support of $\varphi$, we obtain that $|D(i\tilde{\tau}|k|, k)| \gtrsim \hat{w}(k)|k|^{-1}$ for $2\Upsilon \le |\tilde{\tau}| \le 2\Upsilon + |k| - \varepsilon$ for any positive $\varepsilon$.
At $|\tilde{\tau}| = \Upsilon + |k|/2$, the definition of $\kappa_{0}$ tells us that $\Phi(|k|) = D(\pm i(2|k|\Upsilon +|k|^2), k) \neq 0$ for $|k| \le \kappa_{0} /2$. Hence we get \eqref{eq:lowbd2} for $|k| \le \kappa_{0} /2$.
For $\kappa_{0} /2 < |k| \le \kappa_{0} + \delta_{0}$, we have $|D(i\tau, k)| \gtrsim 1$ since $|\tau \pm \tau_*| \ge |k| > \kappa_{0} /2$ and $\tau = \pm \tau_*$ are the only zeros of $D(i \tau, k)$. This concludes the proof of the claim \eqref{eq:lowbd2}. 

Recall that
	\[
		\widetilde{G}_{k,1} (i \tau) = \frac{\mathcal{R}(i\tau, k)}{(\tau^4 - \tau^2 \tau_{0}^2 -|k|^2 \tau_{1}^2 ) D(i\tau, k)}.
	\]
Note from \eqref{eq:remainder2} that $|\mathcal{R}(i\tau, k)| \lesssim |k|^2$. Since $|\tau^4 - \tau^2 \tau_{0}^2 -|k|^2 \tau_{1}^2| \gtrsim |\tau|^2 +|k|^2$, we get
	\[
		|\widetilde{G}_{k,1}(i\tau) \lesssim \frac{|k|^2}{|\tau|^2 +|k|^2|} \cdot \frac{|k|^3}{1+|k|^3} \lesssim |k|^3.
	\]
This implies
	\[
		\left| \int_{\Gamma_2} e^{\lambda t} \widetilde{G}_{k,1} (\lambda) \, d \lambda \right| \lesssim |k|^4.
	\]

Similarly, we bound $|k|^{n} \partial_{\lambda}^{n} \widetilde{G}_{k,1}(\lambda)$. Observe that
	\[
		\left| |k| \partial_{\lambda} \left( \frac{1}{\lambda^4 + \lambda^2 \tau_0^2 -|k|^2 \tau_{1}^2} \right)_{\lambda = i\tau} \right|
		\lesssim \frac{|k| |\tau|}{(|\tau|^2 +|k|^2)^2} \lesssim \frac{1}{|\tau|^2 +|k|^2},
	\]
and similarly $|k|^{n}\partial_{\lambda}^{n}$ derivatives of $(\lambda^{4} +\lambda^2 \tau_{0}^2 - |k|^2 \tau_{1}^2)^{-1}$, evaluated at $\lambda = i\tau$, obeys the same upper bound as that of $(\tau^4 - \tau^2 \tau_{0}^2 - |k|^2 \tau_{1}^2)^{-1}$. In addition, $|k|^n \partial_{\lambda}^n \mathcal{R}(\lambda,k)$ satisfy the same upper bound as that of $\mathcal{R}(\lambda, k)$, see \eqref{eq:remainder} and \eqref{eq:remainder2}. Finally, in view of \eqref{eq:displap}, $|k|^n \partial_{\lambda}^{n} \mathcal{R}(\lambda, k)$ satisfy the same bound as that of $1-D(\lambda, k)$ so that
	\[
		\left| |k| \partial_{\lambda} \left( \frac{1}{D(\lambda, k)} \right)_{\lambda = i \tau} \right|
		= \frac{1}{|D(i\tau, k)|} \cdot \frac{|k| |\partial_{\lambda} D(i\tau, k)|}{|D(i\tau, k)|}
		\lesssim \frac{|k|^3}{1+|k|^3},
	\]
and similarly for $|k|^n \partial_{\lambda}^n$ derivatives of $D(\lambda, k)^{-1}$. To sum up, we get
\[
\left| \int_{\Gamma_2} e^{\lambda t} |k|^n \partial_{\lambda}^{n} \widetilde{G}_{k,1} (\lambda) \, d \lambda \right| \lesssim |k|^4.
\]
\subsubsection*{Estimate of $\widehat{G}_{k,1}(t)$ on $\mathcal{C}_{\pm}$}
Recall that
\[
	\mathcal{C}_{\pm} = \{ \Re \lambda \ge 0 , \quad |\lambda \mp i \tau_* (k) |= |k| \}.
\]
We need to deal with the poles $\lambda_{\pm}(k)$ of $\widetilde{G}_{k}(t)$. We first note that 
\begin{equation}\label{eq:remainderonC}
|\partial_{\lambda}^{n} \mathcal{R}(\lambda, k)| \lesssim |k|^4
\end{equation}
uniformly in $\lambda \in \mathcal{C}_{\pm}$ for $0 \le n < N_{0}-2$.
Indeed, if $|k| \gtrsim 1$, then $|\partial_{\lambda}^{n} \mathcal{R}(\lambda, k)| \lesssim |k|^{2-n} \lesssim |k|^4$ from \eqref{eq:remainder}.
If $|k| \ll 1$, then $|\lambda| \ge \tau_* (k) - |k| \ge \tau_* (k) / 2 \gtrsim 1$ so the claim follows from \eqref{eq:remainder2}.

We may write
\[
	\widetilde{G}_{k,1}(\lambda) = \frac{\lambda^4}{\lambda^4 + \lambda^2 \tau_{0}^2 - |k|^2 \tau_{1}^2  - \mathcal{R}(\lambda, k)} - \frac{\lambda^4}{\lambda^4 + \lambda^2 \tau_{0}^2 - |k|^2 \tau_{1}^2}.
\]

Let $\lambda_{\pm}(k)$, $\lambda_{\pm, 0}(k)$ be the nonreal poles of $\widetilde{G}(\lambda,k)$, $\widetilde{G}_{k, 0}(\lambda)$, respectively. Then 
\begin{equation}\label{eq:poles}
	|\lambda_{\pm}(k) - \lambda_{\pm, 0}(k)| \lesssim |k|^4
\end{equation}
because
\begin{align*}
	|\mathcal{R}(\lambda_{\pm, 0} , k)|
	= |(\lambda_{\pm, 0}^4 + \lambda_{\pm, 0}^2 \tau_{0}^2 -|k|^2 \tau_{1}^2 ) - (\lambda_{\pm}^4 + \lambda_{\pm}^2 \tau_{0}^2 - |k|^2 \tau_{1}^2)| 
	\gtrsim |\lambda_{\pm} - \lambda_{\pm,0}|.
\end{align*}
and \eqref{eq:remainderonC}.
For $|\lambda \mp i \tau_* | \le |k|$ with $\Re \lambda \ge 0$, we may write
\begin{align*}
	\lambda^4 + \lambda^2 \tau_{0}^2 -|k|^2 \tau_{1}^2 
	&=
	\sum_{n=1}^{N_{0}-3} a_{\pm, n, 0}(k)(\lambda - \lambda_{\pm, 0}(k))^{n} + \mathcal{R}_{\pm, 0}(\lambda, k), \\
	\lambda^4 + \lambda^2 \tau_{0}^2 -|k|^2 \tau_{1}^2  - \mathcal{R}(\lambda,k)
	&=
	\sum_{n=1}^{N_{0}-3} a_{\pm, n}(k) (\lambda - \lambda_{\pm}(k))^{n} + \mathcal{R}_{\pm}(\lambda, k).
\end{align*}
Using \eqref{eq:remainderonC} and \eqref{eq:poles}, we get
\[
	|a_{\pm, n, 0} (k) - a_{\pm, n}(k) | \lesssim |k|^4, \qquad |\partial^{n} \mathcal{R}_{\pm, 0} (\lambda, k) - \partial_{\lambda}^{n} \mathcal{R}_{\pm}(\lambda, k)| \lesssim |k|^4
\]
for $0 \le n < N_{0}- 2$.

By the definition of $\lambda_{\pm,0 }$ and $\lambda_{\pm}$, the coefficients $a_{\pm, 1,0}$ and $a_{\pm, 1}$ do not vanish. Hence
\begin{equation}
	\begin{aligned}\label{eq:bn0}
	\frac{\lambda^4}{\lambda^4 + \lambda^2 \tau_{0}^2 -|k|^2 \tau_{1}^2 }
	&=
	\sum_{n=1}^{N_{0}-3} b_{\pm, n, 0}(k) (\lambda - \lambda_{\pm, 0}(k))^{n-1} + \widetilde{\mathcal{R}}_{\pm, 0}(\lambda, k), \\
	\frac{\lambda^4}{\lambda^4 + \lambda^2 \tau_{0}^2 -|k|^2 \tau_{1}^2  - \mathcal{R}(\lambda,k)}
	&=
	\sum_{n=1}^{N_{0}-3} b_{\pm, n}(k) (\lambda - \lambda_{\pm}(k))^{n-1} + \widetilde{\mathcal{R}}_{\pm}(\lambda, k),
	\end{aligned}
\end{equation}
for $|\lambda \mp i \tau_* | \le |k|$ with $\Re \lambda \ge 0$. Here $b_{\pm, n, 0}$ and $b_{\pm, n}$ can be computed in terms of $a_{\pm,n,0}$ and $a_{\pm, n}$. Similarly, one can show that
\begin{equation}\label{eq:est-bn0R}
|b_{\pm, n, 0} (k) - b_{\pm, n}(k) | \lesssim |k|^4, \qquad |\partial_\lambda^{n} \widetilde{\mathcal{R}}_{\pm, 0} (\lambda, k) - \partial_{\lambda}^{n} \widetilde{\mathcal{R}}_{\pm}(\lambda, k)| \lesssim |k|^4
\end{equation}
for $0 \le n < N_{0}-2$. Observe that
\begin{equation}\label{eq:b00}
	b_{\pm, 0, 0}(k) = \operatorname{Res} (\widetilde{G}_{k,0} (\lambda_{\pm,0}) ), \qquad b_{\pm,0}(k) = \operatorname{Res} (\widetilde{G}_{k}(\lambda_{\pm})).
\end{equation}
Using \eqref{eq:bn0}, we write
\[
	\widetilde{G}_{k,1}(\lambda) 
	= \sum_{n=1}^{N_{0}-3} \left[ \frac{b_{\pm,n}(k)}{(\lambda - \lambda_{\pm}(k))^{-n+1}} - \frac{b_{\pm,n, 0}(k)}{(\lambda - \lambda_{\pm, 0}(k))^{-n+1}} \right] + \widetilde{\mathcal{R}}_{\pm}(\lambda, k) - \widetilde{\mathcal{R}}_{\pm,0}(\lambda, k).
\]
Observe that the terms inside the summation are meromorphic in $\mathbb{C}$ so that we may deform the contour to the left half plane in $\lambda$. Applying the residue theorem,
\begin{multline}
	\frac{1}{2\pi i} \int_{\mathcal{C}_{\pm}} e^{\lambda t} \widetilde{G}_{k,1}(\lambda) \, d\lambda \\
	=
	\sum_{\pm} e^{\lambda_{\pm} t} b_{\pm,0}(k)  - \sum_{\pm} e^{\lambda_{\pm,0} t} b_{\pm, 0, 0}(k) 
	+ \frac{1}{2\pi i} \sum_{n=0}^{N_{0}-3} \int_{\mathcal{C}_{\pm}^{*}} e^{\lambda t} \left[ \frac{b_{\pm,n}(k)}{(\lambda - \lambda_{\pm}(k))^{-n+1}} - \frac{b_{\pm,n, 0}(k)}{(\lambda - \lambda_{\pm, 0}(k))^{-n+1}} \right] \, d \lambda \\
	+ \frac{1}{2\pi i} \int_{| \tau \mp \tau_* | \le |k| } e^{i \tau t} \left[\widetilde{\mathcal{R}}_{\pm}(\lambda, k) - \widetilde{\mathcal{R}}_{\pm,0}(\lambda, k)  \right] \, d \tau,
\end{multline}
where $\mathcal{C}_{\pm}^{*} = \{ \Re \lambda \le 0 , \quad |\lambda \mp i \tau_* (k)| = |k| \}$.

As the remainder satisfies \eqref{eq:est-bn0R}, the integral containing the remainder is estimated by
	\[
		\frac{1}{2\pi} \left| \int_{| \tau \mp \tau_* | \le |k| } e^{i \tau t} \left[\widetilde{\mathcal{R}}_{\pm}(\lambda, k) - \widetilde{\mathcal{R}}_{\pm,0}(\lambda, k)  \right] \, d \tau \right| \lesssim |k|^5.
	\]
Next, we consider the integral over $\mathcal{C}_{\pm}^{*}$. On the semicircle, we have $|\lambda| \lesssim 1$ and
	\begin{equation}\label{eq:lambdasemicircle}
		|\lambda - \lambda_{\pm}(k)| \ge \frac{|k|}{2}, \qquad |\lambda - \lambda_{\pm, 0}(k)| \ge \frac{|k|}{2}.
	\end{equation}
Also note that $b_{\pm, n}$ and $b_{\pm, n, 0}$ are bounded as $|k| \le \kappa_{0} + \delta_{0}$. 
Now for $n=0$, we get
	\begin{align}
		\left| \frac{b_{\pm, 0}(k)}{\lambda - \lambda_{\pm}(k)} - \frac{b_{\pm, 0, 0}(k)}{\lambda - \lambda_{\pm, 0}(k)} \right|
		\lesssim \frac{|k|^4}{|\lambda - \lambda_{\pm}(k)| |\lambda - \lambda_{\pm,0}(k)|}
		\lesssim |k|^2.
	\end{align}
Similarly,
	\begin{equation}\label{eq:poly}
		\sum_{n=1}^{N_{0}-3}\left| \frac{b_{\pm,n}(k)}{(\lambda - \lambda_{\pm}(k))^{-n+1}} - \frac{b_{\pm,n, 0}(k)}{(\lambda - \lambda_{\pm, 0}(k))^{-n+1}} \right| \lesssim |k|^4. 
	\end{equation}
Therefore
	\[
		\left| \frac{1}{2\pi i} \sum_{n=0}^{N_{0}-3} \int_{\mathcal{C}_{\pm}^{*}} e^{\lambda t} \left[ \frac{b_{\pm,n}(k)}{(\lambda - \lambda_{\pm}(k))^{-n+1}} - \frac{b_{\pm,n, 0}(k)}{(\lambda - \lambda_{\pm, 0}(k))^{-n+1}} \right] \, d \lambda \right|
		\lesssim \int_{\mathcal{C}_{\pm}^{*}} |k|^2 \, d \lambda 
		\lesssim |k|^3.
	\]
	
Using a similar argument, one can estimate for $|k|^{n}\partial_{\lambda}^{n} \widetilde{G}_{k, 1}(\lambda)$. Indeed,
\begin{equation}
	\partial_{\lambda}^{n} \widetilde{G}_{k,1}(\lambda) 
	= \sum_{j=1}^{N_{0}-3} \frac{d^{n}}{d\lambda^n} \left[ \frac{b_{\pm,j}(k)}{(\lambda - \lambda_{\pm}(k))^{-j+1}} - \frac{b_{\pm,j, 0}(k)}{(\lambda - \lambda_{\pm, 0}(k))^{-j+1}} \right] + \partial_{\lambda}^{n} \widetilde{\mathcal{R}}_{\pm}(\lambda, k) - \partial_{\lambda}^{n} \widetilde{\mathcal{R}}_{\pm,0}(\lambda, k)
\end{equation}
for $0 \le n < N_{0} - 2$. From \eqref{eq:est-bn0R}, we have
\begin{equation}
\frac{1}{2\pi} \left| \int_{| \tau \mp \tau_* | \le |k| } e^{i \tau t} |k|^{n} \left[\partial_{\lambda}^{n} \widetilde{\mathcal{R}}_{\pm}(\lambda, k) - \partial_{\lambda}^{n}\widetilde{\mathcal{R}}_{\pm,0}(\lambda, k)  \right] \, d \tau \right| \lesssim |k|^{n+5}.
\end{equation}
Next, integrating by parts, we obtain from the residue theorem that
\[
	\frac{(-1)^{n}}{2\pi i t^n} \int_{\mathcal{C}_{\pm}}  e^{\lambda t} \frac{d^n}{d \lambda^n} 
		\left[
			\sum_{j=1}^{N_{0}-3} \frac{b_{\pm,j}(k)}{(\lambda - \lambda_{\pm}(k))^{-j+1}}
		 \right] \, d \lambda
	= e^{\lambda_{\pm}t} b_{\pm, 0}(k),
\]
and
\[
	\frac{(-1)^{n}}{2\pi i t^n} \int_{\mathcal{C}_{\pm}}  e^{\lambda t} \frac{d^n}{d \lambda^n} 
	\left[
	\sum_{j=1}^{N_{0}-3} \frac{b_{\pm,j,0}(k)}{(\lambda - \lambda_{\pm,0}(k))^{-j+1}}
	\right] \, d \lambda
	= e^{\lambda_{\pm,0}t} b_{\pm, 0, 0}(k).
\]
It remains to show
\begin{equation}\label{eq:semicirc-higher}
	\left| \frac{1}{2\pi i} \int_{\mathcal{C}_{\pm}^{*}} e^{\lambda t}|k|^n \frac{d^{n}}{d\lambda^n} \left[ \frac{b_{\pm,j}(k)}{(\lambda - \lambda_{\pm}(k))^{-j+1}} - \frac{b_{\pm,j, 0}(k)}{(\lambda - \lambda_{\pm, 0}(k))^{-j+1}} \right]  \, d\lambda \right|
	\lesssim |k|^3
\end{equation}
for $0 \le j , n < N_{0} -2$. When $j=0$, we get
\begin{align*}
	\left| \frac{d^{n}}{d\lambda^n} \left[ \frac{b_{\pm,0}(k)}{\lambda - \lambda_{\pm}(k)} - \frac{b_{\pm,0, 0}(k)}{\lambda - \lambda_{\pm, 0}(k)} \right] \right|
	&\lesssim \left| \frac{b_{\pm,0}(k)}{(\lambda - \lambda_{\pm}(k))^{n+1}} - \frac{b_{\pm,0, 0}(k)}{(\lambda - \lambda_{\pm, 0}(k))^{n+1}}  \right| \\[4pt]
	&\lesssim \frac{|b_{\pm,0}(k) - b_{\pm,0,0}(k)|}{|\lambda - \lambda_{\pm}(k)|^{n+1}} + \frac{|\lambda_{\pm}(k) - \lambda_{\pm,0}(k)|}{|\lambda - \widetilde{\lambda}_{\pm}(k)|^{n+2}}
\end{align*}
for $\lambda \in \mathcal{C}_{\pm}^{*}$, where $\widetilde{\lambda}_{\pm}(k)$ is some number on the line segment connecting $\lambda_{\pm}(k)$ and $\lambda_{\pm,0}(k)$. Using \eqref{eq:est-bn0R} and \eqref{eq:lambdasemicircle}, we get
\begin{align*}
	\left| \frac{d^{n}}{d\lambda^n} \left[ \frac{b_{\pm,0}(k)}{\lambda - \lambda_{\pm}(k)} - \frac{b_{\pm,0, 0}(k)}{\lambda - \lambda_{\pm, 0}(k)} \right] \right|
	\lesssim |k|^{-n+2},
\end{align*}
which implies \eqref{eq:semicirc-higher}.
The other terms are easier, see \eqref{eq:poly}. To sum up,
\[
	\frac{(-1)^{n}}{2\pi i (|k|t)^{n}} \int_{\mathcal{C}_{\pm}}  e^{\lambda t} |k|^{n} \partial_{\lambda}^{n} \widetilde{G}_{k,1}(\lambda) \, d \lambda
	= \sum_{\pm} e^{\lambda_{\pm} t} b_{\pm,0}(k) -\sum_{\pm} e^{\lambda_{\pm,0} t}b_{\pm,0,0}(k)+ \mathcal{O}(|k|^3 \langle kt \rangle^{-n} )
\]
for $0 \le n < N_{0} -2$. From \eqref{eq:Gkhat2} and \eqref{eq:b00}, we obtain \eqref{eq:est-Gkhat2}.

\subsubsection*{Residue of $\widetilde{G}(\lambda,k)$ at $\lambda_{\pm}(k)$}
Now we compute $\operatorname{Res} (\widetilde{G}_{k}(\lambda_{\pm}))$. We may write
\[
	\widetilde{G}(\lambda,k) = 1 + \sum_{\pm} \frac{a_{\pm}(\lambda, k)}{\lambda -\lambda_{\pm}(k)}
\]
where $a_{\pm}(\lambda,k)$ is some function which is holomorphic in $\{ \Re \lambda > 0 \}$.
Observe that
	\[
		a_{\pm}(k) :=  a_{\pm}(\lambda_{\pm}(k) , k) = \operatorname{Res} (\widetilde{G}_{k}(\lambda_{\pm}(k))) = \lim_{\lambda \rightarrow \lambda_{\pm}(k)} \frac{\lambda - \lambda_{\pm}(k)}{D(\lambda, k)}.
	\]
Using \eqref{eq:displap}, we obtain that $\partial_{\lambda}^{n} D(\lambda, k)$ is well-defined and finite on $\{ \Re \lambda \ge 0\}$ for $0 \le n < N_{0} -2$. It follows that $a_{\pm}(k)$ is of class $C^{N_{0} - 3}$ for $|k| \le \kappa_{0} + \delta_{0}$ by the smooth extension. In particular, we can compute from \eqref{eq:dispersion} that
	\[
		a_{\pm}(0) =  \frac{1}{\partial_{\lambda}D(\pm i \tau_{0} ,0) } = \pm \frac{i \tau_{0}}{2}.
	\]

\subsubsection*{Estimate of $|k|^{|\alpha|} \partial_{k}^{\alpha} \widehat{G}_{k}^{r}(t)$}
We estimate $|k|^{|\alpha|} \partial_{k}^{\alpha} \widehat{G}_{k}^{r}(t)$ for $ |\alpha| < N_{0} -2$. Namely, we prove that
\begin{align}
	\label{eq:est-Gkhat-kderiv}
	&\Big|  \partial_{k}^{\alpha} \widehat{G}_{k,0}(t) 
	- \sum_{\pm}  e^{\lambda_{\pm, 0} t} \operatorname{Res} (\partial_{k}^{\alpha} \widetilde{G}_{k,0}(\lambda); \lambda_{\pm, 0})  \Big| 
	\lesssim |k|^{3} \langle t \rangle^{|\alpha|} e^{-\theta_{0}|kt|}, \\
	\label{eq:est-Gkhat2-deriv}
	&\Big|  \partial_{k}^{\alpha}\widehat{G}_{k,1}(t) 
	- \sum_{\pm} e^{\lambda_{\pm} t} \operatorname{Res} (\partial_{k}^{\alpha} \widetilde{G}(\lambda,k); \lambda_{\pm})
	+ \sum_{\pm}  e^{\lambda_{\pm, 0} t} \operatorname{Res} (\partial_{k}^{\alpha}\widetilde{G}_{k,0}(\lambda); \lambda_{\pm, 0})   \Big|
	\lesssim |k|^{3-|\alpha|} \langle kt \rangle^{-N_{0} +3 + |\alpha|}.
\end{align}
for some $0 < \theta_{0} < \tau_{1}$. 
First, observe from \eqref{eq:Gk0} and \eqref{eq:Gk02} that
\[
	\left| \partial_{k}^{\alpha} \left(e^{\mu_{-, 0 } t} \operatorname{Res} (\widetilde{G}_{k,0}(\mu_{-,0})) \right) \right|
	\lesssim |k|^3 \langle t \rangle^{|\alpha|} e^{-\theta_{0} |kt|},
\]
which implies \eqref{eq:est-Gkhat-kderiv}.
Next, recall from \eqref{eq:remainderexpression} that
\[
	\mathcal{R}(\lambda,k) = |k|^4 \tau_{0}^2 + {16i|k|^2} \int_{0}^{\infty} e^{-\lambda t} N(t) dt 
\]
where $N(t)$ is given by \eqref{eq:Ndefinition} and \eqref{eq:Nestimate}. Hence
\[
	\left| |k|^{|\alpha|}\partial_{k}^{\alpha} \mathcal{R}(\lambda, k) \right|
	\lesssim |k|^4 + |k|^{|\alpha|} \int_{0}^{\infty} \left( |k|^{2-|\alpha|} + |k|^2 \langle t \rangle^{|\alpha|} \right) |k| \langle kt \rangle^{-N_{0} +1} \, dt
	\lesssim |k|^2
\]
for $|\alpha| < N_{0} -2$. Therefore $|k|^{|\alpha|} \partial_{k}^{\alpha} \mathcal{R}(\lambda, k)$ satisfy the same bounds as those of $\mathcal{R}(\lambda, k)$. One can repeat the same argument on $|k|^{|\alpha|} \partial_{k}^{\alpha} \partial_{\lambda}^{n} \widetilde{G}_{k, 1}(\lambda)$ to obtain \eqref{eq:est-Gkhat2-deriv}.
\end{proof}

\subsection{Green function in the physical space}

In this section, we bound the Green function in the physical space.
We will use Littlewood-Paley decomposition. That is, for any function $h$, we have
	\begin{equation}\label{eq:LP}
		h(x) = \sum_{q \in \mathbb{Z}} P_{q} h(x)	
	\end{equation}
where $P_{q}$ is the Littlewood-Paley projection on $[2^{q-1}, 2^{q+1}]$. Namely, $\widehat{P_{q}h}(k) = \widehat{h}(k) \varphi(k/2^{q})$ for some smooth cutoff function $\varphi \in [0,1]$ supported in $\{ \frac{1}{4} \le |k| \le 4 \}$ and equal to one in $\{ \frac{1}{2} \le |k| \le 2 \}$.
\begin{proposition}\label{prop-Greenx}
	Let $G(t,x)$ be the Green function of the linearized Hartree density \eqref{linHartree}. Then, there holds 
$$G(t,x)	
	= \delta(t, x) + \sum_{\pm} G^{osc}_{\pm}(t,x) + G^{r}(t, x),$$
	in which the oscillatory part satisfies
	\begin{equation}\label{Gosc-phys}
		\| G_{\pm}^{osc} \star_{x} f \|_{L^{p}_{x}} \lesssim t^{-d (1/2 - 1/p)} \| f \|_{L^{p'}_{x}}
	\end{equation}
	for any $p \in [2, \infty]$ satisfying $\frac{1}{p} + \frac{1}{p'} = 1$. 
The remainder part satisfies
	\begin{equation}
		\| \chi(i\partial_{x}) \partial_{x}^{n} G^{r}(t) \|_{L^{p}_{x}} \lesssim \langle t \rangle^{-3-d(1-1/p)-n},
	\end{equation}
	and
	\begin{equation}
		\| (1-\chi(i\partial_{x})) \partial_{x}^{n} G^{r}(t) \|_{L^{p}_{x}} \lesssim \langle t \rangle^{-N_{0}/2}
	\end{equation}
	for $0 \le n \le N_{0}/2$ and $p \in [1,\infty]$, where $\chi(k)$ is a smooth cutoff function supported in $\{ |k| \le \kappa_{0} + \delta_{0} \}$.
\end{proposition}
\begin{proof}
	Note that the oscillatory Green function is given by
	\[
		G_{\pm}^{osc} (t,x) =  \int_{\mathbb{R}^d} e^{\lambda_{\pm}(k) t + ik \cdot x } a_{\pm}(k) \, dk,
	\]
where $a_{\pm}(k)$ is sufficiently smooth and supported in $\{ |k| \le \kappa_{0} + \delta_{0} \}$. Recalling that $\Re \lambda_{\pm}(k) \le 0$, we obtain the $L^2$ estimate, namely,
\[
	\| G_{\pm}^{osc} \star_{x} f \|_{L^{2}_{x}} \lesssim  \| f \|_{L^{2}_{x}}.
\]
The dispersive estimate
\[
	\| G_{\pm}^{osc} \star_{x} f \|_{L^{\infty}_{x}} \lesssim  t^{-d/2}\| f \|_{L^{1}_{x}}
\]
follows from the stationary phase analysis, see the proof of \cite[Proposition 5.4]{HKNR3} or \cite[Lemma 3.1]{RS}, upon noting that the Hessian determinant $| \det D^2 \tau_{*} (k) |  = | \left( \frac{\tau_{*}'}{|k|} \right)^{d-1} \tau_{*}'' | \gtrsim 1$ by Theorem \ref{theo-LangmuirE}. The general estimates \eqref{Gosc-phys} follow from the interpolation.

Next we deal with $G^{r}(t)$. Recalling \eqref{eq:Gr-lowfreq}, we have
\[
	| \chi(i \partial_{x}) \partial_{x}^{n} G^{r}(t) |
	\lesssim \int_{\{|k| \lesssim 1 \}} |k|^{n} |\widehat{G}_{k}^{r}(t) | \, dk 
	\lesssim \int_{\{|k| \lesssim 1 \}} |k|^{n+3} \langle kt \rangle^{-N_{0} +3} \, dk \lesssim t^{-n-3-d},
\]
which gives the $L^{\infty}$ bound in the low frequency case. Similarly, using \eqref{eq:Gr-highfreq}, we have
\[
	| (1-\chi(i\partial_{x})) \partial_{x}^{n} G^{r}(t) |
	\lesssim \int_{\{|k| \gtrsim 1 \}} |k|^{n} |\widehat{G}_{k}^{r}(t) | \, dk
	\lesssim \int_{\{|k| \gtrsim 1 \}} |k|^{n-1} \langle kt \rangle^{-N_{0} +3} \, dk
	\lesssim \langle t \rangle^{-N_{0}/2}
\]
provided that $N_{0} \ge 6$. This gives the $L^{\infty}$ bound in the high frequency case. To obtain the $L^1$ bound, observe that
\begin{equation*}
	\begin{aligned}
		\| (1- \chi(i\partial_{x})) \partial_{x}^{n} G^{r}(t) \|_{L^{1}_{x}}^2
		&\lesssim \sum_{|\alpha| \le \lfloor d/2 \rfloor +1} \int_{\{ |k| \gtrsim 1 \}} |k|^{2n} |\partial_{k}^{\alpha} \widehat{G}_{k}^{r} (t) |^2 \, dk  \\
		&\lesssim \sum_{|\alpha| \le \lfloor d/2 \rfloor +1} \int_{\{ |k| \gtrsim 1 \}} |k|^{2n-2-2|\alpha|} \langle kt \rangle^{-2N_{0} +6 + 2 |\alpha|} \, dk 
		\lesssim \langle t \rangle^{-N_{0}}
	\end{aligned}
\end{equation*}
provided that $N_{0} \ge d +7$. For the low frequency part, we use the Littlewood-Paley decomposition \eqref{eq:LP}. Noting that $|k| \lesssim 1$, we can write
\[
	\chi(i \partial_{x}) G^{r}(t,x) = \sum_{q \le C} P_{q}[ G^{r}(t,x)]	
\]
for some constant $C$ where
\begin{equation*}
\begin{aligned}
	P_{q}[G^{r}(t,x)] 
	&= \int_{\{2^{q-2} \le |k| \le 2^{q+2}\}} e^{ik \cdot x } \widehat{G}^{r}(t,k) \varphi(k/2^q) dk\\
	&= 2^{d} q \int_{ \{ 1/4 \le |\widetilde{k}| \le 4 \}} e^{i \widetilde{k} \cdot 2^{q}x} \widehat{G}^{r}(t, 2^{q} \widetilde{k}) \varphi(\widetilde{k}) \, d \widetilde{k}
\end{aligned}
\end{equation*}
Observe that
\begin{equation*}
\begin{aligned}
		| \partial_{\widetilde{k}}^{\alpha} \widehat{G}^{r}(t, 2^{q}\widetilde{k}) |
		\lesssim 2^{q|\alpha|} | \partial_{k}^{\alpha} \widehat{G}^{r}(t, 2^{q}\widetilde{k}) |
		\lesssim 2^{3q} \langle 2^{q} t \rangle^{-N_{0} +3 + |\alpha|}
\end{aligned}
\end{equation*}
for $1/4 \le |\widetilde{k}| \le 4$ and $q \le C$. Integrating by parts, it follows that
\begin{equation*}
\begin{aligned}
	|P_{q} [G^{r}(t,x)] |
	&\lesssim 2^{dq} \left| \int e^{i \widetilde{k} \cdot 2^{q} x} \widehat{G}^{r} ( t, 2^{q} \widetilde{k}) \varphi(\widetilde{k}) \, d \widetilde{k} \right| \\
	&\lesssim 2^{dq} \int \langle 2^{q} x \rangle^{-|\alpha|} \left| \partial_{\widetilde{k}}^{\alpha} (\widehat{G}^{r}(t, 2^{q}\widetilde{k}) \varphi(\widetilde{k}) ) \right| \, d \widetilde{k} \\
	&\lesssim 2^{q(d+3)} \langle 2^{q} x \rangle^{-|\alpha|}  \langle 2^{q} t \rangle^{-N_{0} +3 + |\alpha|}.
\end{aligned}
\end{equation*}
Taking $|\alpha| = 4$, we get
\begin{equation*}
\begin{aligned}
	\| \chi(i\partial_{x}) \partial_{x}^{n} G^{r}(t) \|_{L^{1}_{x}}
	\lesssim \sum_{q \le C} 2^{nq} \| P_{q}[G^{r}(t)]\|_{L^{1}_{x}} 
	\lesssim \sum_{q \le C} 2^{n(q+3)} \langle 2^{q}t \rangle^{-N_{0} +d +4} \lesssim t^{-n-3},
\end{aligned}
\end{equation*}
provided that $N_{0} > d+5$.
To sum up,
\[
	\| \chi(i \partial_{x}) \partial_{x}^{n} G^{r}(t) \|_{L^{p}_{x}} \lesssim \langle t \rangle^{-3 -d (1- 1/p) - n }
\]
for $p \in [1, \infty]$.
\end{proof}

\section{Linear decay}

\subsection{Phase mixing for free gases}\label{sec-freeHartree}

In this section, we prove the following dispersion for the free Hartree dynamics. 

\begin{proposition}\label{prop-mixrho0} 
Let $\gamma_{0}$ be the initial matrix operator whose integral kernel $\gamma_0(x,y)$ satisfies 
$$ \sum_{|\alpha|+|\beta|\le N_3 }\| (1+|x|+|y|)^{N_2} \partial_x^\alpha \partial_y^\beta\gamma_0\|_{L^1_{x,y}} < \infty
$$
for some $N_2,N_3>d$. 
Then, for any $n,\alpha$ so that $n+|\alpha|< N_2 - d$, the density $\rho^0(t,x)$ of the free Hartree dynamics $\gamma^0(t) = e^{it\Delta} \gamma_{0} e^{-it \Delta}$ satisfies 
\begin{equation}\label{dispHLp}
		\| \partial_{t}^{n} \partial_{x}^{\alpha} \rho^{0}(t) \|_{L^{p}_{x}} \lesssim t^{-n-|\alpha| - d(1-1/p)} 
\end{equation}
for $p\in [1,\infty]$. 
\end{proposition}

\begin{proof}
Let $\gamma^{0}(t,x,y)$ be the integral kernel of $\gamma^0(t) = e^{it\Delta} \gamma_{0} e^{-it \Delta}$. It follows that the Fourier transform of $\gamma^0(t,x,y)$ can be computed by 
	\[
		\widehat{\gamma}^{0}(t,k,p) = e^{-it(|k|^2 -|p|^2)} \widehat{\gamma_{0}}(k,p),
	\]
while the associated density is 
	\begin{equation}\label{rho0-exp1}
		\begin{aligned}
			\widehat{\rho}^{0}(t,k) 
			&= \int \widehat{\gamma}^{0}(t,k-p,p)\, dp 
			= \int e^{-itk \cdot (k-2p)} \widehat{\gamma_{0}} (k-p, p) \, dp \\
			&= 2^{-d} \int e^{-it k \cdot v} \widehat{\gamma_{0}} \left( \frac{k+v}{2}, \, \frac{k-v}{2} \right) \, dv
		\end{aligned}
	\end{equation}
where we have substituted $v= k-2p$ in the last inequality. 
From the regularity of $\widehat{\gamma_{0}}$, it follows that
	\[
		|\widehat{\rho}^{0}(t,k)| \lesssim \langle kt \rangle^{-N_{2}} \sum_{|\alpha| + |\beta| \le N_{2}} \| (|k|+|p|)^{N_3} \partial_{k}^{\alpha} \partial_{p}^{\beta} \widehat{\gamma_{0}} \|_{L^{\infty}_{k,p}} ,
	\]
for any $N_2$ and for $N_3>d$. In addition, $t$-derivatives gain a factor of $|k|$, but lose a linear growth in $|v|\le |k+v| + |k-v|$, yielding  
	\[
		|\partial_t^n\widehat{\rho}^{0}(t,k)| \lesssim |k|^n\langle kt \rangle^{-N_{2}} \sum_{|\alpha| + |\beta| \le N_{2}} \| (|k| + |p|)^{n+N_3}\partial_{k}^{\alpha} \partial_{p}^{\beta} \widehat{\gamma_{0}} \|_{L^{\infty}_{k,p}} ,
	\]
for any $n\ge 0$. Thus, taking $N_2>d$, we thus have
	\[
		\begin{aligned}
		|  \rho^{0} (t, x) | &\lesssim \int  \langle kt \rangle^{-N_{2}} \sum_{|\alpha| + |\beta| \le N_{2}}  \| (|k|+|p|)^{N_3} \partial_{k}^{\alpha} \partial_{p}^{\beta} \widehat{\gamma_{0}} \|_{L^{\infty}_{k,p}}\, dk  
		\\&\lesssim t^{-d} \sum_{|\alpha| + |\beta| \le N_{2}}  \| (|k|+|p|)^{N_3} \partial_{k}^{\alpha} \partial_{p}^{\beta} \widehat{\gamma_{0}} \|_{L^{\infty}_{k,p}}.
\end{aligned}	\]
Similarly, for any $\alpha,n$ and for $N_2> d + n + |\alpha|$, we obtain
	\[
		\| \partial_{t}^{n} \partial_{x}^{\alpha} \rho^{0}(t) \|_{L^{\infty}_{x}} \lesssim t^{-d-n-|\alpha|} 
		\sum_{|\alpha| +|\beta| \le N_{2}} \| (|k| + |p|)^{n+N_3}\partial_{k}^{\alpha} \partial_{p}^{\beta} \widehat{\gamma_{0}} \|_{L^{\infty}_{k,p}},
	\]
	which proves \eqref{dispHLp} for $p=\infty$. Note that the right hand-side is finite by assumption and the fact that 
	$\| \widehat{g}\|_{L^{\infty}_{k,p}} \le \| g\|_{L^1_{x,y}}$	for any function $g(x,y)$. 
	
As for $L^1$ estimates, from \eqref{rho0-exp1}, we compute 
$$
		\begin{aligned}
			|k|\nabla_k\widehat{\rho}^{0}(t,k) 
			&= -2^{-d} i |k|t \int ve^{-it k \cdot v} \widehat{\gamma_{0}} \left( \frac{k+v}{2}, \, \frac{k-v}{2} \right) \, dv 
			\\& + 2^{-d-1} |k|\int e^{-it k \cdot v} \Big( \nabla_k\widehat{\gamma_{0}} +  \nabla_p\widehat{\gamma_{0}} \Big)\left( \frac{k+v}{2}, \, \frac{k-v}{2} \right) \, dv.
		\end{aligned}
$$
A similar calculation holds for $|k|^{|\alpha|}\partial_k^\alpha$, for any $|\alpha|\ge 0$. Thus, estimating the integrals similarly as done above, we obtain for $t\ge 1$, 
	\[
		|k|^{|\alpha|}|\partial_k^\alpha\widehat{\rho}^{0}(t,k)| \lesssim \langle kt \rangle^{-N_{2}+|\alpha|} \sum_{|\alpha| + |\beta| \le N_{2}}\| (|k| + |p|)^{N_3}\partial_{k}^{\alpha} \partial_{p}^{\beta} \widehat{\gamma_{0}} \|_{L^{\infty}_{k,p}},
	\]
from which we obtain the $L^1$ estimates, following the analysis done for $G^r(t,x)$ in the proof of Proposition \ref{prop-Greenx}.  Finally, the decay estimates \eqref{dispHLp} in $L^p$ norms follow by interpolation between those in $L^1$ and $L^\infty$ norms. 
\end{proof}

\subsection{Density representation}
From \eqref{resolvent}, we have $\rho(t,x) 
	= (G \star_{t,x} \rho^{0} )(t,x) $. Therefore, using Proposition \ref{prop-Greenx}, we may write \begin{equation}\label{densdecomp}
\begin{aligned}
	\rho(t,x) 
	&= \rho^{0}(t,x) + \sum_{\pm} G^{osc}_{\pm} \star_{t,x} \rho^{0} (t,x) + G^{r} \star_{t,x} \rho^{0}(t,x) 
\end{aligned}
\end{equation}
As it turns out that $\rho^{0}$ does not have enough decay to obtain the decay for the oscillatory part of the density through the spacetime convolution, we shall integrate by parts in time to make use of the extra decay from the derivatives of $\rho^0$, see Proposition \ref{prop-mixrho0}. Precisely, we have the following.
\begin{proposition}\label{prop-densityrep}
	Let $\rho(t,x)$ be the density of the linearized Hartree equations \eqref{linHartree}. It can be written as
		\begin{equation}
			\rho = \sum_{\pm} \rho^{osc}_{\pm} + \rho^{r}
		\end{equation}
	with
		\begin{equation}\label{eq:rep}
		\begin{aligned}
			\rho^{osc}_{\pm}(t,x) 
			&= G^{osc}_{\pm}(t) \star_{x}
				\left[
					\frac{1}{\lambda_{\pm}(i\partial_{x})} \rho^{0}(0) + \frac{1}{\lambda_{\pm}(i\partial_{x})^2} \partial_{t} \rho^{0}(0)
				\right]
				+ G^{osc}_{\pm} \star_{t,x} \frac{1}{\lambda_{\pm}(i\partial_{x})^2} \partial_{t}^2 \rho^{0} \\
			\rho^{r}(t,x)
			&= G^r \star_{t,x} \rho^{0} + \mathcal{P}_{2}(i\partial_{x}) \rho^{0} + \mathcal{P}_{4}(i \partial_{x}) \partial_{t} \rho^{0}
		\end{aligned}
		\end{equation}
	where $G^{osc}_{\pm}(t,x)$ and $G^{r}(t,x)$ are defined in Propositions \ref{prop:stable} and \ref{prop:GreenFourier1}, and $\lambda_{\pm}(k)$ in Theorems \ref{theo-LangmuirE} and \ref{theo-Landau}.
	Moreover, $\mathcal{P}_{2}(i\partial_{x})$ and $\mathcal{P}_{4}(i\partial_{x})$ Fourier multipliers, where $\mathcal{P}_{2}(k)$ and $\mathcal{P}_{4}(k)$ are sufficiently smooth and satisfy
	\begin{equation}\label{bd-P2P4}
		|\mathcal{P}_{2}(k)| + |\mathcal{P}_{4}(k)| \lesssim |k|^2 \langle k \rangle^{-2}, \qquad k \in \mathbb{R}^d.
	\end{equation}
	Furthermore, $\mathcal{P}_{4}(k)$ is compactly supported in $\{ |k| \le \kappa_{0} + \delta_{0} \}$.
\end{proposition}

\begin{proof}
	Since $\lambda_{\pm}(k)$ is supported in $\{ |k| \le \kappa_{0} + \delta_{0} \}$, the proposition in the case when $\{|k| \gtrsim 1\}$ follows from \eqref{densdecomp} .
	When $|k| \lesssim 1$, we integrate $G^{osc}_{\pm} \star_{t,x} \rho^{0}$ by parts in time to use the oscillation of $G^{osc}_{\pm}(t,x)$. Namely, we get
	\[
		G^{osc}_{\pm} \star_{t,x} \rho^{0}
		= \frac{1}{\lambda_{\pm}(i \partial_{x})} 
			\left[
				G^{osc}_{\pm}(t) \star_{x} \rho^0 (0) - G^{osc}_{\pm}(0) \star_{x} \rho^{0}(t) + G^{osc}_{\pm} \star_{t,x} \partial_{t} \rho^{0}
			\right]
	\]
	Using the fact that $\lambda_{\pm}(0) = \pm i \tau_{0}$ and $a_{\pm}(0) = \pm \frac{i \tau_{0}}{2}$, we get
	\[
		\sum_{\pm} \frac{1}{\lambda_{\pm}(i\partial_{x})} G^{osc}_{\pm}(0,x) 
		= \sum_{\pm} \int e^{ik\cdot x} \frac{a_{\pm}(k)}{\lambda_{\pm}(k)} \, dk 
		= \int_{\{|k| \lesssim 1\}} e^{ik \cdot x } (1 + \mathcal{O}(|k|^2)) \, dk.
	\]
	Hence
	\[
		\sum_{\pm} \frac{1}{\lambda_{\pm}(i \partial_{x})} G^{osc}_{\pm}(0) \star_{x} \rho^{0}(t) 
		= \int_{\{|k| \lesssim 1\}} e^{ik \cdot x } (1 + \mathcal{O}(|k|^2)) \widehat{\rho}^{0}(t,k) \, dk 
		= \widetilde{\mathcal{P}}_{3} (i\partial_{x}) \rho^{0} - \widetilde{\mathcal{P}}_{2} (i\partial_{x})\rho^{0}
	\]
	where $\widetilde{\mathcal{P}}_{3}$ denotes the smooth version of the Fourier multiplier $1_{\{ |k| \lesssim 1 \}}$ and $\widetilde{\mathcal{P}}_{2}$ denotes the Fourier multiplier which is sufficiently smooth, $\widetilde{\mathcal{P}}_{2}(k) = 0$ for $|k| \gtrsim 1$, and $|\widetilde{\mathcal{P}}_{2}(k)| \lesssim |k|^2$.
	Now
$$	\begin{aligned}
		\rho^{0} + \sum_{\pm} G^{osc}_{\pm} \star_{t,x} \rho^{0}
		&= \sum_{\pm} \frac{1}{\lambda_{\pm}(i\partial_{x})} G^{osc}_{\pm}(t) \star_{x} \rho^0 (0) + \sum_{\pm} \frac{1}{\lambda_{\pm}(i\partial_{x})}  G^{osc}_{\pm} \star_{t,x} \partial_{t} \rho^{0} \\
			&\quad+ \widetilde{\mathcal{P}}_{2} (i\partial_{x})\rho^{0} +(1- \widetilde{\mathcal{P}}_{3} (i\partial_{x})) \rho^{0} .
	\end{aligned}$$
	The time decay of $\partial_{t} \rho^{0}$ is still not enough to get the decay of $G^{osc}_{\pm} \star_{t,x} \rho^{0}$, so we integrate by parts in time once again to obtain
	\[
		G^{osc}_{t,x} \star_{t,x} \partial_{t} \rho^{0}
		= \frac{1}{\lambda_{\pm}(i \partial_{x})} 
			\left[
				G^{osc}_{\pm}(t) \star_{x} \partial_{t} \rho^0 (0) - G^{osc}_{\pm}(0) \star_{x} \partial_{t} \rho^{0}(t) + G^{osc}_{\pm} \star_{t,x} \partial_{t}^2 \rho^{0}
			\right].
	\]
	Using the fact that $\lambda_{\pm}(0)^2 = - \tau_{0}^2$ and $a_{\pm}(0) = \pm \frac{i\tau_{0}}{2}$, we get
	\[
		\sum_{\pm} \frac{1}{\lambda_{\pm}(i\partial_{x})^2} G^{osc}_{\pm}(0,x) 
		= \sum_{\pm} \int e^{ik\cdot x} \frac{a_{\pm}(k)}{\lambda_{\pm}(k)^2} \, dk 
		= \int_{\{|k| \lesssim 1\}} e^{ik \cdot x } \mathcal{O}(|k|^2) \, dk,
	\]
	noting that the zeroth order terms cancelled each other. Therefore 
	\begin{multline*}
		\rho^{0} + \sum_{\pm} G^{osc}_{\pm} \star_{t,x} \rho^{0}
		= \sum_{\pm} \frac{1}{\lambda_{\pm}(i\partial_{x})} G^{osc}_{\pm}(t) \star_{x} \rho^0 (0) + \sum_{\pm} \frac{1}{\lambda_{\pm}(i\partial_{x})}  G^{osc}_{\pm} (t)\star_{x} \partial_{t} \rho^{0}(0) \\
		+ \widetilde{\mathcal{P}}_{2} (i\partial_{x})\rho^{0} +(1- \widetilde{\mathcal{P}}_{3} (i\partial_{x})) \rho^{0} + \mathcal{P}_{4} (i \partial_{x}) \partial_{t} \rho^{0} + \sum_{\pm} \frac{1}{\lambda_{\pm}(i \partial_x)^2} G^{osc}_{\pm} \star_{t,x} \partial_{t}^2 \rho^{0}
	\end{multline*}
	where $\mathcal{P}_{4} (i\partial_{x})$ denotes the Fourier multiplier which is sufficiently smooth, compactly supported in $\{ |k| \lesssim 1\}$, and $|\mathcal{P}_{4}(k)| \lesssim |k|^2$. Setting $\mathcal{P}_{2} := \widetilde{\mathcal{P}}_{2}  + (1 - \widetilde{\mathcal{P}}_{3})$, the proof of the proposition is complete.
\end{proof}

\subsection{Proof of Theorem \ref{theo-main}}

We now give the proof of Theorem \ref{theo-main}. Recall the representation \eqref{eq:rep}. First we bound $\rho^{osc}_{\pm}$. Using \eqref{Gosc-phys}, we have
\[
	\| G^{osc}_{\pm}(t) \star_{x} \rho^{0}(0) \|_{L^{p}_{x}} \lesssim t^{-d(1/2 - 1/p)}
\]
and
\[
	\| G^{osc}_{\pm}(t) \star_{x} \partial_{t}\rho^{0}(0) \|_{L^{p}_{x}} \lesssim t^{-d(1/2 - 1/p)}
\]
for $p \in [2,\infty]$.
On the other hand,
\begin{equation*}
	\begin{aligned}
		\| G^{osc}_{\pm} \star_{t,x} \partial_{t}^2 \rho^{0} \|_{L^{\infty}_{x}}
		&\lesssim \int_{0}^{t/2} \| G^{osc}_{\pm} (t-s) \|_{L^{\infty}_{x}} \| \partial_{t}^2 \rho^{0}(s) \|_{L^{1}_{x}} \, ds+ \int_{t/2}^{t} \| G^{osc}_{\pm} (t-s) \|_{L^{2}_{x}} \| \partial_{t}^2 \rho^{0}(s) \|_{L^{2}_{x}} \, ds \\
		&\lesssim \int_{0}^{t/2} (t-s)^{-d/2} \langle s \rangle^{-2} \, ds + \int_{t/2}^{t} \langle s \rangle^{-d/2 - 2} \, ds \\
		&\lesssim t^{-d/2}, 
	\end{aligned}
\end{equation*}
and
\begin{equation*}
	\begin{aligned}
		\| G^{osc}_{\pm} \star_{t,x} \partial_{t}^2 \rho^{0} \|_{L^{2}_{x}}
		\lesssim \int_{0}^{t} \| G^{osc}_{\pm}(t-s)\|_{L^{2}_{x}} \| \partial_{t}^2 \rho^{0}(s) \|_{L^{1}_{x}} \,  ds \lesssim \int_{0}^{t} \langle s \rangle^{-2} \, ds \lesssim 1.
	\end{aligned}
\end{equation*}
These imply $L^{p}$ bound on $G^{osc}_{\pm} \star_{t,x} \partial_{t}^2 \rho^{0}$ for $p \in [2,\infty]$. 
Finally, $1/\lambda_{\pm}(i\partial_{x})$ is a bounded operator on $L^{p}$ for $1 \le p \le \infty$ since $1/\lambda_{\pm}(k)$ is sufficiently regular and compactly supported. Combining these, we get
	\[
		\| \rho^{osc}_{\pm}(t) \|_{L^{p}_{x}} \lesssim t^{-d (1/2 - 1/p)}
	\]
for $p \in [2, \infty]$.
Next, we bound $\rho^{r}$.
By a similar argument, for $1 \le p \le \infty$,
\begin{equation*}
	\begin{aligned}
		\| G^{r} \star_{t,x}  \rho^{0} \|_{L^{p}_{x}}
		&\lesssim \int_{0}^{t/2} \| G^{r} (t-s) \|_{L^{p}_{x}} \| \rho^{0}(s) \|_{L^{1}_{x}} ds + \int_{t/2}^{t} \| \Delta_x^{-1}G^{r} (t-s) \|_{L^{1}_{x}} \| \Delta_x\rho^{0}(s) \|_{L^{p}_{x}} ds \\
		&\lesssim \int_{0}^{t/2} \langle t-s \rangle^{-3 - d (1-1/p)} \, ds + \int_{t/2}^{t} \langle t-s 	\rangle^{-1} \langle s \rangle^{-2-d(1-1/p)} \, ds \\
		&\lesssim t^{-2-d(1-1/p)} \log (1+t).
	\end{aligned}
\end{equation*}

Recall from Proposition \ref{prop-densityrep} that $\mathcal{P}_{2} = \widetilde{\mathcal{P}}_{2} + (1- \widetilde{\mathcal{P}}_{3})$ where $\widetilde{\mathcal{P}}_{2}(k)$ is sufficiently regular, compactly supported and $1- \widetilde{\mathcal{P}}_{3}(k)$ is a smoothed version of $1_{\{|k| \gtrsim 1\}}$. Hence $\mathcal{P}_{2}(i \partial_{x})$ is bounded on $L^{p}$ for $1 \le p \le \infty$. Similarly $\mathcal{P}_{4}(i \partial_{x})$ is also bounded on $L^{p}$ for $1 \le p \le \infty$.
Therefore, recalling \eqref{bd-P2P4}, we have
\[
	\| \mathcal{P}_{2} (i\partial_{x}) \rho^{0}(t) \|_{L^{p}_{x}} \lesssim \| \partial_x^2 \rho^{0}(t) \|_{L^{p}_{x}} \lesssim \langle t \rangle^{-2-d(1-1/p)}
\]
and
\[
	\| \mathcal{P}_{4} (i\partial_{x}) \partial_{t}\rho^{0}(t) \|_{L^{p}_{x}} \lesssim \| \partial_{t} \partial_x \rho^{0}(t) \|_{L^{p}_{x}} \lesssim \langle t \rangle^{-2-d(1-1/p) }.
\]
Combining all these estimates, we obtain that
\[
	\| \rho^{r}(t) \|_{L^{p}_{x}} \lesssim \langle t \rangle^{-2-d(1-1/p)} \log(1+t)
\]
for $p \in [1,\infty]$, which concludes the proof of Theorem \ref{theo-main}.

\bibliographystyle{plain}

\end{document}